\documentclass[11pt]{amsart}

\usepackage[utf8]{inputenc}
\usepackage{amsmath}
\usepackage{amsfonts}
\usepackage{amsthm}
\usepackage[all,2cell]{xy}
\usepackage{setspace}
\usepackage{stmaryrd}
\usepackage{fullpage}
\usepackage{hyperref}
\usepackage{tikz}
\usepackage{caption}
\usepackage{subcaption}
\usepackage{enumitem}
\UseTwocells

\usepackage{amsfonts,amsthm,latexsym,amsmath,amssymb,amscd,epsf,graphicx,mathtools} 

\newcommand{\CC}{\mathbb{C}}

\newcommand{\PP}{\mathbb{P}}
\newcommand{\RR}{\mathbb{R}}
\newcommand{\TT}{\mathbb{T}}
\newcommand{\cala}{\mathcal{A}}
\newcommand{\calb}{\mathcal{B}}
\newcommand{\calbl}{\mathcal{B}^{\circlearrowleft}}
\newcommand{\calbr}{\mathcal{B}^{\circlearrowright}}
\newcommand{\cale}{\mathcal{E}}
\newcommand{\calc}{\mathcal{C}}
\newcommand{\calf}{\mathcal{F}}
\newcommand{\calh}{\mathcal{H}}
\newcommand{\calhl}{\mathcal{H}_{\circlearrowleft}}
\newcommand{\calhr}{\mathcal{H}_{\circlearrowright}}
\newcommand{\calo}{\mathcal{O}}
\newcommand{\calx}{\mathcal{X}}
\newcommand{\caly}{\mathcal{Y}}
\newcommand{\calz}{\mathcal{Z}}
\newcommand{\calu}{\mathcal{U}}
\newcommand{\calv}{\mathcal{V}}
\newcommand{\calw}{\mathcal{W}}

\newcommand{\nc}{\mathbf{nc}}
\newcommand{\rg}{\Upsilon}

\newcommand{\del}{\partial}

\newcommand{\cofib}{\operatorname{cofib}}
\newcommand{\fib}{\operatorname{fib}}
\newcommand{\Hom}{\operatorname{Hom}}

\newcommand{\Loc}{\operatorname{Loc}}
\newcommand{\Map}{\operatorname{Map}}
\newcommand{\Mod}{\operatorname{Mod}}
\newcommand{\map}{\underline{\operatorname{Map}}}
\newcommand{\Mor}{\boldsymbol{\operatorname{Map}}}

\newcommand{\Perf}{\operatorname{Perf}}
\newcommand{\perf}{\underline{\operatorname{Perf}}}

\newcommand{\Spec}{\operatorname{Spec}}
\newcommand{\Sym}{\operatorname{Sym}}
\newcommand{\Symp}{\operatorname{Symp}}
\newcommand{\Fuk}{\operatorname{Fuk}}
\newcommand{\FS}{\operatorname{FS}}
\newcommand{\HH}{\operatorname{HH}}
\newcommand{\HC}{\operatorname{HC}}
\newcommand{\tot}{\operatorname{tot}}
\newcommand{\Rep}{\operatorname{Rep}}
\newcommand{\id}{\operatorname{id}}
\newcommand{\SH}{\operatorname{SH}}
\newcommand{\Fun}{\operatorname{Fun}}
\newcommand{\ev}{\operatorname{ev}}
\newcommand{\ch}{\operatorname{ch}}
\newcommand{\QCoh}{\operatorname{QCoh}}
\newcommand{\colim}{\operatorname{colim}}
\newcommand{\catk}{\operatorname{Cat}_k}
\newcommand{\tr}{\operatorname{tr}}
\newcommand{\add}{\operatorname{add}}
\newcommand{\chain}{\operatorname{C}}
\newcommand{\pr}{\operatorname{pr}}
\newcommand{\hml}{\operatorname{H}}
\newcommand{\traj}{\operatorname{Traj}}
\newcommand{\free}{\operatorname{Free}}
\newcommand{\Tw}{\operatorname{Tw}}
\newcommand{\Ob}{\operatorname{Ob}}
\newcommand{\Gen}{\operatorname{Gen}}
\newcommand{\nondeg}{\operatorname{Nondeg}}
\newcommand{\rev}{\operatorname{rev}}

\theoremstyle{plain}
\newtheorem{thm}{Theorem}
\numberwithin{thm}{section}
\newtheorem{lem}[thm]{Lemma}
\newtheorem{prop}[thm]{Proposition}
\newtheorem{cor}[thm]{Corollary}

\theoremstyle{definition}
\newtheorem{defn}[thm]{Definition}
\newtheorem{example}[thm]{Example}
\newtheorem{notation}[thm]{Notation}
\newtheorem{question}[thm]{Question}
\newtheorem{problem}[thm]{Problem}
\newtheorem{construction}[thm]{Construction}

\theoremstyle{remark}
\newtheorem{rmk}[thm]{Remark}

\title{Calabi-Yau structures,  spherical functors, and shifted symplectic structures}
\author{Ludmil Katzarkov, Pranav Pandit and Ted Spaide}
\date{}

\begin{document}

\maketitle
\begin{abstract}
A categorical formalism is introduced for studying various features of the symplectic geometry of Lefschetz fibrations and the algebraic geometry of Tyurin degenerations. This approach is informed by homological mirror symmetry, derived noncommutative geometry, and the theory of Fukaya categories with coefficients in a perverse Schober. The main technical results include (i) a comparison between the notion of relative Calabi-Yau structures and a certain refinement of the notion of a spherical functor, (ii) a local-to-global gluing principle for constructing Calabi-Yau structures, and (iii) the construction of shifted symplectic structures and Lagrangian structures on certain derived moduli spaces of branes. Potential applications to a theory of derived hyperk\"ahler geometry are sketched.

\end{abstract}

\tableofcontents

\section{Introduction}

This paper develops a program aimed at studying various features of Picard-Lefschetz theory using the language and methods of higher category theory and derived geometry, with a view toward applications in algebraic geometry, symplectic geometry, and homological mirror symmetry.  Some of the potential applications to fundamental questions in classical geometry, and hyperk\"ahler geometry in particular, are sketched in section \ref{sec:future-directions}, where we outline the contours of a twistorial approach to a theory of derived hyperk\"ahler geometry.  The main results of this paper are summarized in subsection \ref{subsec:results} below. Before giving a precise statement of the results, we begin with a leisurely informal discussion of the motivation and background for these results.

\subsection{Categorical Picard-Lefschetz theory: spherical functors and monodromy.}\label{subsec:categorical-Picard-Lefschetz}

The categorical approach to Picard-Lefschetz theory derives from viewing Picard-Lefschetz theory through the lens of symplectic geometry. Arnol'd \cite{Arnold-symp-mon} observed that the monodromy transformations of Picard-Lefschetz theory are in fact symplectomorphisms, leading to a shift in perspective, from a topological view of Lefschetz fibrations, to the richer and more refined symplectic viewpoint.  The theory of \emph{symplectic} Lefschetz fibrations was introduced by Donaldson \cite{Donaldson-symplectic-Lefschetz}, and many of the categorical ideas are implicit in his work. Following seminal ideas of Donaldson, Fukaya and Kontsevich, Seidel introduced and developed a beautiful and far-reaching theory of Fukaya-type categories associated to Lefschetz fibrations \cite{Seidel-book}, generalizing many features of classical Picard-Lefschetz theory to the realm of symplectic topology. The point of view that we will take in this paper, following Seidel, is that the symplectic topology of a Lefschetz fibration is captured by the interaction between the Fukaya-Seidel category of the fibration and the Fukaya category of its generic fiber. As we will show below, this interaction is encoded in the data of certain \emph{spherical functors}.

 Let \(w: X \rightarrow \mathbb{C}\) be a Lefschetz fibration. For simplicity and concreteness, let us assume that \(X\) is a quasi-projective variety over \(\mathbb{C}\), and that \(w\) is a proper holomorphic map. The function \(w\) has a finite set \( \{p_1,p_2,...,p_n\}\) of critical values. On \(\mathbb{C}-\{p_1,...,p_n\}\), \(w\) defines a locally trivial fiber bundle. To a point \(t \in \mathbb{C}-\{p_1,...,p_n\}\), we can associate the Fukaya category \(\Fuk(X_t)\) \cite{Fooh1, Fooh2} of the fiber \(X_t := w^{-1}(t)\). This is an \(\infty\)-category that is linear over the Novikov field \(\mathbb{C}((t^{\mathbb{R}}))\). The objects of this category are, roughly speaking, Lagrangian submanifolds of \(X_t\) equipped with unitary local systems. The space of morphisms between two objects is given by Lagrangian Floer cohomology. As \(t\) varies, we obtain a local system of \(\infty\)-categories over the complement of the set of critical values of \(w\): indeed, if \(\gamma: [0,1] \rightarrow \mathbb{C} -\{p_1,...,p_n\}\) is a path, then symplectic parallel transport gives rise to a symplectomorphism \(X_{\gamma(0)} \rightarrow X_{\gamma(1)}\), which in turn induces an equivalence of categories \(\Fuk(X_{\gamma(0)}) \rightarrow \Fuk(X_{\gamma(1)})\). 

On the other hand, we have the Fukaya-Seidel category \(\FS(X,w)\) of the Lefschetz fibration. Objects of this category are, roughly speaking, Lagrangian submanifolds of \(X\) decorated with unitary local systems, with the property that the intersection of \(w(L)\) with the complement of some compact set is contained in the positive real axis. If we fix a smooth fiber \(Y:= X_t\) of \(w\), then there is a natural functor

\[
\cap: \FS(X,w) \rightarrow \Fuk(Y)
\]

given by sending a Lagrangian \(L\) in \(X\) to \(L \cap X_t\). This leads naturally to the following:

\begin{question}\label{question:cap-functor}
What extra structure does the functor \(\cap: \FS(X,w) \rightarrow \Fuk(Y)\) carry?
\end{question}

A partial answer to this question is given by the  notion of spherical functor introduced by Anno and Logvinenko \cite{AL}. A functor \(F: \calx \rightarrow \caly \) between stable \(\infty\)-categories is \emph{spherical} if 

\begin{itemize}
\item[-] \(F \) admits a left adjoint \(F^*\) and a right adjoint \(F^!\).
\item[-] The homotopy cofiber \(T_{\caly}\) of the counit \(F \circ F^! \rightarrow \id_{\caly}\) is an autoequivalence of \(\caly\)
\item[-] The homotopy fiber \(T_{\calx}\) of the unit \(\id_{\calx} \rightarrow F^! \circ F \) is an autoequivalence of \(\calx\).
\end{itemize}

As originally observed by Kapranov and Schechtman \cite{KaSch}, this structure should be viewed as a natural categorification of the structure of a perverse sheaf on the unit disc \(D\) in \(\mathbb{C}\), with Whitney stratification consisting of the strata \(\{0\}\) and \(D - \{0\}\). Indeed, to a perverse sheaf \(\calf\) on \(D\) one can associate a pair of vector spaces: the vanishing cycles \(\Phi(\calf)\) of \(\calf\), and the nearby cycles \(\Psi(\calf\)). These vector spaces come with a pair of maps 
\[
\xymatrix{
 \Phi(\calf) \ar@<.5ex>[r]^{u} &  \Psi(\calf) \ar@<.5ex>[l]^{v} \\
}
\]
with the property that \(T_{\Psi} := \id_{\Psi} - uv \) and \(T_{\Phi} := \id_{\Phi} - vu \) are invertible morphisms of vector spaces. Furthermore, the classification of perverse sheaves on \(D\) \cite{Beilinson-perverse-gluing, GGM} says that the datum of a perverse sheaf on \(D\) is equivalent to the data of pair of vector spaces and maps satisfying the condition above.

Theorems \ref{intro:thm-relcy-iff-spherical} and \ref{intro:thm-A-model-relcy} together imply that the functor \( \cap: \FS(X,w) \rightarrow \Fuk(Y)\) is spherical, under hypotheses that are expected to hold in great generality. We refer the reader to Section \ref{sec:A-model} for the precise statement. This fact has been anticipated by the experts, and a different proof was outlined by Abouzaid in \cite{A}. The emphasis of our investigation is on the interaction of the spherical functor with Calabi-Yau structures, to which we turn now.

Let \(\calx\) be an \(\infty\)-category linear over some field \(k\). Recall that a Serre functor \cite{BK-Serre} for \(\calx\) is an autoequivalence \(S_{\calx}: \calx \rightarrow \calx\) with the property that for any two objects \(x, y\) in \(\calx\) we have 
\[
\Mor(x,y) \simeq \Mor(x, S_{\calx}y)^{\vee}
\]
where \(\Mor(x,y)\) is the \(k\)-module spectrum of maps from \(x\) to \(y\). For example, if \(\calx = \Perf(X)\) is the category of perfect complexes on a smooth projective variety \(X\) of dimension \(d\), then \(S_{\calx} = (-) \otimes \omega_X[d]\) where \(\omega_X\) is the canonical line bundle. A Serre functor is unique up to isomorphism if it exists, and a Serre functor exists for any smooth and proper category.  A \emph{weak} Calabi-Yau structure of dimension $d$ on $\calx$ is an equivalence of functors between the Serre functor $S_{\calx}$ of $\calx$ and $\id_{\calx}[d]$.

If we are given a Lefschetz fibration \(w: X \rightarrow \mathbb{C}\) as above, with \(w\) proper,  then the generic smooth fiber \(Y\) is a compact oriented manifold. Poincare duality on \(Y\) is reflected in the fact that the category \(\caly:= \Fuk(Y)\) carries a Calabi-Yau structure \cite{Gan2}, and in particular \(S_{\caly} \simeq \id_{\caly}[d]\), where \(2d\) is the real dimension of \(Y\). On the other hand, the Serre functor \(S_{\calx}\) of \(\calx := \FS(X,w)\) is non-trivial. This leads naturally to the following question:

\begin{question}\label{question:spherical-cy-interaction}
How is the Serre functor of \(\FS(X,w)\) related to the data of the spherical adjunction \(\FS(X,w) \leftrightarrows \Fuk(Y)\) and the monodromy of the Lefschetz fibration \(w: X \rightarrow \mathbb{C}\)?
\end{question}

The monodromy around infinity of the fibration \(w\) gives rise to an autoequivalence \(\sigma_*\) of \(\FS(X,w)\). In \cite{Seidel-symplectic-Hochschild}, Seidel observed, using symplectic geometry, that there is a natural map \(\sigma_* \rightarrow \id \). In the language of spherical functors that we have introduced above, \(\sigma_*\) is the functor \(T_{\calx} := \fib(\id_{\calx} \rightarrow \cap^! \circ \cap)\). Seidel attributes to Kontsevich the idea that \(\sigma_* \simeq S_{\calx}[-(d+1)]\). 

Motivated by these considerations, we introduce the notion of what we will call a \emph{compatible spherical functor}, building on the notion of spherical functor introduced by Anno and Logvinenko \cite{AL} and a subsequent addition by Katzarkov, Kontsevich and Pantev \cite{KKP} (see Definition \ref{def:spherical}). Roughly speaking, given an arbitrary linear \(\infty\)-category \(\calx\), a \(d\)-Calabi-Yau category \(\caly\), and a spherical functor \(F: \calx \rightarrow \caly\), we say that \(F\) is \emph{compatible with the Calabi-Yau structure}, or simply that \(F\) is a compatible spherical functor, if we are given an equivalence \(T_{\calx} \simeq S_{\calx}[-(d+1)]\) satisfying a compatibility condition. This is a relative version of the notion of a weak Calabi-Yau structure in the sense that it reduces to the latter when $\caly = 0$.  Theorems \ref{intro:thm-relcy-iff-spherical} and \ref{intro:thm-A-model-relcy} imply that the functor \(\cap: \FS(X,w) \rightarrow \Fuk(Y)\) is in fact a spherical functor that is compatible with the natural Calabi-Yau structure on \(\Fuk(Y)\), under certain technical hypotheses.

\subsection{Gluing spherical functors: perverse Schobers and categorical surgery.}\label{subsec:gluing-spherical}

The monodromy around infinity of the Lefschetz fibration can be computed from the monodromy around the individual singular fibers. From the discussion of the last paragraph of \S \ref{subsec:categorical-Picard-Lefschetz} above, we should expect that this manifests itself categorically as a ``decomposition'' of the Serre functor of \(\FS(X,w)\). Let us try to formulate this more precisely.

Let the notation be as in \S \ref{subsec:categorical-Picard-Lefschetz}. Choose a disc \(U_i\) centered at each of the critical values \(p_i \in \mathbb{C}\) of \(w: X \rightarrow \mathbb{C}\), such that the discs are pairwise disjoint. Let \(X_i := w^{-1}(U_i)\) and let \(w_i: X_i \rightarrow U_i\) be the restriction of \(w\) to \(X_i\). Then we can define a Fukaya-Seidel category \(\calx_i = \FS(X_i,w_i)\) associated to \((X_i, w_i)\) exactly as we did for \((X,w)\). More precisely, if we choose a point \(q_i^\infty\) on \(\partial U_i\), then the objects of \(\FS(X_i, w_i)\) are decorated Lagrangians that, outside some compact set containing \(p_i\),  project to the ray joining \(p_i\) to \(q_i^\infty\). The morphisms are defined exactly as for \(\FS(X,w)\) (see \cite{Seidel-book}). Let us call \(\FS(X_i, w_i)\) the local Fukaya-Seidel category at \(p_i\). This is a categorification of the space of vanishing cycles at \(p_i\).

\begin{question}\label{question:gluing}
Can the Fukaya-Seidel category \(\FS(X,w)\) be reconstructed from the data of the local categories \(\FS(X_i, w_i)\) and the collection of spherical functors \(\cap_i: \FS(X_i, w_i) \rightarrow \Fuk(Y)\), where \(Y\) is a generic smooth fiber of the Lefschetz fibration?
\end{question}

Kontsevich has proposed a powerful sheaf-theoretic framework for Fukaya-type categories \cite{kontsevich-symp}, which provides a natural setting within within which to pose and study such local-to-global questions. Given a topological fibration \(\pi: X \rightarrow B\), we can compute the cohomology of \(X\) as the derived global sections of a sheaf on \(B\) whose stalks are the cohomology of the fibers: \(H^*(X, \mathbb{Z}) = R\Gamma(B, R\pi_*\mathbb{Z})\). In the same vein, if \(\pi: X \rightarrow B\) is a symplectic fibration of symplectic manifolds, then Kontsevich proposes that one should recover Fukaya-type categories associated to \(X\) as a certain global object associated to  a \emph{sheaf of categories} over \(B\).

More precisely, Kontsevich proposes that one should introduce a notion of a \emph{perverse sheaf of linear \(\infty\)-categories}. A symplectic fibration  \(\pi: X \rightarrow B\) should give rise to such a ``perverse sheaf'' \(\mathfrak{X}\), whose generic stalk is the Fukaya category of the generic fiber of \(\pi\). By passing to a suitable additive invariant, such as periodic cyclic homology, one should obtain a perverse sheaf on \(B\), in the usual sense of the term.  

Given a singular Lagrangian \(\mathcal{L} \rightarrow B\), one should have an induced constructible sheaf of \(\infty\)-categories \(\mathfrak{X}_{\mathcal{L}}\) on \(\mathcal{L}\), and we define \(\Fuk(\mathcal{L}, \mathfrak{X})\) to be the global sections of \(\mathfrak{X}_{\mathcal{L}}\) (there is a dual version involving cosheaves; for simplicity we stick to the case of sheaves here). The \emph{semi-classical Fukaya category of B with coefficients in} \(\mathfrak{X}\) can then be defined to be
\[
\Fuk^{\mathrm{sc}}(B, \mathfrak{X}) = \colim_{\mathcal{L} \rightarrow B}  \Fuk(\mathcal{L}, \mathfrak{X})
\] 

The Fukaya category \(\Fuk(B, \mathfrak{X})\) of \(B\) with coefficients in \(\mathfrak{X}\) should then be a certain deformation of \(\Fuk^{\mathrm{sc}}(B, \mathfrak{X})\) by pseudo-holomorphic discs in \(B\).  Conjecturally, when \(\mathfrak{X}\) comes from a symplectic fibration \(\pi: X \rightarrow B\),  \(\Fuk(B, \mathfrak{X})\) should recover an appropriate version of the partially wrapped Fukaya category of \(X\).

In the special case that this paper is concerned with, the base has real dimension 2. In this situation, the notion of a perverse sheaf of categories has been given a concrete mathematical incarnation under the moniker ``perverse Schobers'' by Kapranov and Schechtman in \cite{KaSch}. Let us restrict, as we did above, to the case where the base is topologically \(\mathbb{C}\). Then a perverse Schober \(\mathfrak{X}\) on \(\mathbb{C}\) with singularities at \(\{p_1,...,p_n\}\) is essentially given by the data of categories \(\caly, \calx_1,..., \calx_n\), and spherical functors \(F_i: \calx_i \rightarrow \caly\) (see \S2 of \cite{KaSch}). The particular presentation of the Schober depends on a ``system of cuts'' in the base; we refer to \emph{loc. cit} for details about how the presentation depends on this choice. The category \(\caly\) is the stalk of the local system of categories \(\mathfrak{X}_{ \lvert \mathbb{C}- \{ p_1,...,p_n \} }\) at some fixed point \(t \in \mathbb{C}-\{p_1,...,p_n\} \). When the Schober ``arises from'' a Lefschetz fibration, then we have \(\calx_i = \FS(X_i, w_i)\), \(\caly = \Fuk(Y) = \Fuk(X_t)\) and \(F_i = \cap_i\). 

Note that in this situation, a singular Lagrangian is just a ribbon graph \(\Gamma\) in \(\mathbb{C}\). The sheaf \(\mathfrak{X}_{\Gamma}\) associates to 
\begin{itemize}
\item[-] a generic smooth point of \(\Gamma\) the category \(\caly\), 
\item[-] an \(n\)-valent vertex located at point in \(\mathbb{C}- \{ p_1,...,p_n \}\) the category \(\caly \otimes \Rep(A_{n-1})\)
\item[-] a \(1\)-valent vertex that coincides with the critical point \(p_i\) of \(w\) the category \(\calx_i\)
\end{itemize}

Conjecturally, we should have
\[\Fuk(\mathbb{C}, \mathfrak{X}) \simeq \FS(X,w)
\]

With this background in mind we can formulate

\begin{question}\label{question:CY-Schober}
\begin{enumerate}[label=(\roman*)]
\item Monodromy around infinity induces an autoequivalence of the generic fiber \(\caly\) of the perverse Schober \(\mathfrak{X}\). How can describe this autoequivalence in terms of the presentation of the Schober by spherical functors \(F_i: \calx_i \rightarrow \caly\)? In other words, can we effectively compute the Serre functor of the global sections of the Schober in terms of the Serre functors \(S_{\calx_i}\) and the spherical functors \(F_i\)?
\item Can (compatible) spherical functors be glued together to produce new (compatible) spherical functors?
\item Let \(\mathfrak{X}\) be a perverse Schober, such that the generic fiber is Calabi-Yau, and all of the spherical functors in some presentation are compatible with the Calabi-Yau structure (see Definition \ref{def:spherical}). Is there a natural compatible spherical structure on the map from the global sections of this Schober to the generic fiber?
\end{enumerate}

\end{question}

The phenomena that this question seeks to address have been studied using a different language in \cite{Seidel-symplectic-Hochschild, Seidel-nc-divisor} and \cite{KV}. In \cite{Seidel-nc-divisor}, Seidel introduces the notion of a \emph{noncommutative divisor} to capture the structures associated with the symplectic geometry of a  Lefschetz fibration in the language of \(A_{\infty}\)-algebras. Essentially the same algebraic structure was introduced earlier in a different context by Tradler and Zeinalian in \cite{Tradler-Zeinalian}, under the moniker \emph{\(V_{\infty}\)-algebra}. Independently, Kontsevich and Vlassopoulos have introduced a generalization of this structure in \cite{KV}. Our approach differs in that, whereas the aforementioned papers rely on explicit formulas and resolutions within the \(A_{\infty}\)-formalism, here we seek to formulate everything intrinsically, and in a manifestly model-independent manner. It would be an interesting question to compare the structures and constructions we introduce here with those in \emph{loc. cit}. As in those papers, our constructions are most naturally viewed from the perspective of derived noncommutative geometry, to which we turn now.

\subsection{Noncommutative oriented cobordisms  and Tyurin degenerations.}\label{subsec:Tyurin} 
Derived noncommutative geometry, pioneered by Kontsevich and developed in \cite{KS, KKP}, is a powerful paradigm in which to formulate and study various dualities predicted by string theory, such as mirror symmetry. The fundamental underlying principle is that it is useful to identify a ``space'' with the physical theory it defines. In the context of topological string theory, the relevant physical theory is an extended \(2d\)-topological field theory (TFT). According to the cobordism hypothesis \cite{Lurie-TFT}, such a TFT is determined by the \(k\)-linear \(\infty\)-category that it assigns to a framed point. With this as motivation, one defines a derived noncommutative space, or \(\nc\)-space for short, to be a \(k\)-linear stable \(\infty\)-category. This point of view leads naturally to the homological/categorical interpretation of physical mirror symmetry introduced in \cite{Kontsevich-HMS}. The central problem of \(\nc\)-geometry is to extract the partition function of the TFT on various manifolds from the data of the category, and to describe various geometric operations on the target spaces for the A-model TFT and B-model TFT in purely categorical terms.

A smooth and proper category determines an extended \(2d\)-TFT defined on \emph{framed} manifolds. In order to obtain a theory that is defined on \emph{oriented} manifolds, the category must be equipped with an additional structure: namely, a Calabi-Yau structure. The notion of a \emph{weak} Calabi-Yau structure introduced in \S\ref{subsec:categorical-Picard-Lefschetz} (an isomorphism between the Serre functor and a shift of the identity functor), turns out to be insufficient for this purpose. To rectify this, Kontsevich and Soibelman introduced \cite{KS} the notion of a Calabi-Yau structure of dimension \(d\) on $\calx$, which is a class of degree \(-d\) 

\[
\HH_*(\calx) \rightarrow k[-d]
\]

in the dual of the cyclic homology $\HC_{\ast}(\calx) \simeq \HH_{\ast}(\calx)_{hS^{1}}$ of $\calx$ which induces an isomorphism of the Serre functor with $\id_{\calx}[d]$ under the identification \(\HH_*(\calx)^{\vee} \simeq \Mor(\id_{\calx}, S_{\calx})\) upon forgetting the $S^1$ equivariance data. Kontsevich-Soibelman \cite{KS} and Costello \cite{Costello-CY} showed how to construct an oriented topological field theory from the data of a category equipped with a Calabi-Yau structure. Lurie's classification of topological field theories \cite{Lurie-TFT} places these results within a very general context. In \emph{loc. cit.}, he introduces the notion of a Calabi-Yau object in an arbitrary symmetric monoidal $(\infty,2)$-category, and shows that Calabi-Yau objects in a certain  $(\infty,2)$-category of $k$-linear categories can be identified with Calabi-Yau categories in the sense of \cite{KS}.

There is another sense in which the term orientation is associated with Calabi-Yau structures. If \(X\) is a compact oriented \(C^{\infty}\)-manifold of real dimension \(d\), then capping with the fundamental class defined by the orientation defines a map of degree \(-d\)

\[
\cap [X]: C^*(X,k) \rightarrow k[-d]
\]
 
from the space of \(k\) valued cochains on \(X\) to the base field \(k\). This map implements Poincar\'e duality in the sense that the composite map
\[
\xymatrix{
C^*(X,k) \otimes C^*(X,k) \ar[r]^(0.6){\cup} & C^*(X,k) \ar[r]^{\cap [X]} & k[-d] \\
}
\]

defines a perfect pairing on \(C^*(X,k)\).

One can view the a manifold \(X\), or more generally any topological space \(X\), as a constant derived stack. From this point of view, \(C^*(X,k)\) can be viewed as the space of functions on \(X\), i.e., the (derived) global section \(\Gamma(X, \mathcal{O}_X)\). Now, if \(X\) is an arbitrary derived stack satisfying certain finiteness conditions, then one defines \cite{PTVV} an \emph{\(\mathcal{O}\)-orientation of dimension \(d\)} on \(X\) to be a map 
\[
\mathrm{or}: \Gamma(X, \mathcal{O}_X) \rightarrow k[-d]
\]
satisfying properties analogous to those satisfied by \(\cap [X]\) above. If \(X\) is a projective variety of dimension \(d\), then it is straightforward to see that \emph{an \(\mathcal{O}\)-orientation of dimension \(d\) on \(X\) is a Calabi-Yau structure of dimension \(d\) on \(X\)}. The formal similarity between the definition of a Calabi-Yau structure on a category and an \(\mathcal{O}\)-orientation on a derived stack is manifest. Furthermore, it is true that a projective variety \(X\) admits an \(\mathcal{O}\)-orientation of dimension \(d\) if and only if the linear \(\infty\)-category \(\Perf(X)\) admits a Calabi-Yau structure of dimension \(d\).

In manifold theory, it is essential to understand how manifolds can be glued out of more elementary manifolds \emph{with boundary}. Tyurin \cite{Tyurin-Fano-CY} has emphasized that the following two situations should be seen as analogous:
\begin{itemize}
\item[-] a smooth oriented \(C^{\infty}\)-manifold \(X\) with boundary \(\partial X \subset X \), as encoded by the pullback map \(C^*(X,k) \rightarrow C^*(\partial X, k)\). Note that the boundary inherits a natural orientation.

\item[-] an anticanonical divisor \(D:= \partial X\) in a Fano variety \(X\), and the associated pullback map on derived global sections \(\Gamma(X, \mathcal{O}_X) \rightarrow \Gamma(\partial X, \mathcal{O}_{\partial X})\). Note that since \(\partial X\) is \emph{anticanonical}, it acquires a natural Calabi-Yau structure. 
\end{itemize}

Broadening our world of spaces to include all derived stacks, these two examples can be placed on a common footing. Motivated by the theory of Lagrangian structures in shifted symplectic geometry, Calaque has introduced the notion of a \emph{nondegenerate boundary structure} \cite{Calaque}, which is essentially an \(\mathcal{O}\)-orientation on a ``derived stack with boundary''.  Roughly speaking, it consists of an orientation \(\Gamma(\partial X, \mathcal{O}_{\partial X}) \rightarrow k[-d]\) together with a homotopy commutative diagram
\[
\xymatrix{
\Gamma(X,\mathcal{O}_X) \ar[r] \ar[d]  & \Gamma(\partial  X, \mathcal{O}_{\partial X}) \ar[d] \\
0 \ar[r] & k[-d]
}
\]
witnessing a nullhomotopy of the pullback of the orientation, that is required to satisfy a certain nondegeneracy condition. 

Tyurin's fundamental insight was that, just as one can understand an oriented manifold by viewing it as being obtained by gluing two oriented manifolds with boundary (i.e., by splitting along a codimension one submanifold), one can study a Calabi-Yau variety \(X\) by \emph{degenerating} it to the union of two Fano varieties \(X_1, X_2\) glued along a common anticanonical divisor \(Z\): \(X \rightsquigarrow X_1 \cup_{Z} X_2\). From the perspective of the \(B\)-model, this corresponds to a deformation of the category \(\Perf(X_1) \times_{\Perf(Z)} \Perf(X_2)\) to the category \(\Perf(X)\). Since the latter category is Calabi-Yau, it is natural to expect that the former category carries a natural Calabi-Yau structure.

To better understand this situation, we introduce the notion of a relative Calabi-Yau structure on a functor $F: \calx \rightarrow \caly$  (Definition \ref{def:relcy}); this is a noncommutative analogue of the nondegenerate boundary structures of \cite{Calaque}. Our definition builds on a suggestion  made by To\"en in \cite{Toen-DAG}. Expanding on To\"en's suggestion and ideas that emerged initially in discussions between Dyckerhoff and the second named author, Brav and Dyckerhoff  have also introduced and studied notions of relative Calabi-Yau structures in \cite{BD}. Roughly speaking, a relative Calabi-Yau structure on \(F: \calx \rightarrow \caly\) consists of a Calabi-Yau structure \(\phi: \HH_*(\caly) \rightarrow k[-d]\), together with ``isotropy data'' given by a functor \(\Delta^1 \times \Delta^1 \rightarrow \Perf(k)\) 
\[
\xymatrix{
\HH_*(\calx) \ar[r] \ar[d] &\HH_*(\caly) \ar[d]^{\phi} \\
0 \ar[r] & k[-d] \\
}
\]

satisfying a certain nondegeneracy condition. 

In \cite{BD}, it is shown that, in the situation above, the restriction functor \(i^*: \Perf(X_i) \to \Perf(Z)\) carries a natural relative Calabi-Yau structure. Now, the mirror to this functor is the cap functor \(\cap: \FS(X^{\vee}_i,w_i) \to \Fuk(Z^{\vee})\) for a Landau-Ginzburg model \((X^{\vee}_i,w_i)\) mirror to \(X_i\). Thus we are led to the following question, to which an affirmative answer is provided by Theorem \ref{intro:thm-A-model-relcy} (see \S\ref{sec:A-model}):

\begin{question}
How can we construct a natural relative Calabi-Yau structure on the cap functor \(\cap: \FS(X,w) \to \Fuk(Y)\) of a Landau-Ginzburg model without appealing to mirror symmetry (i.e., using symplectic geometry)? 
\end{question}

As mentioned earlier, a key feature of \(\cap: \FS(X,w) \to \Fuk(Y)\) is that it is a spherical functor. Given a spherical functor \(F: \calx \to \caly\) with Calabi-Yau target, there is natural compatibility structure that one can ask for, relating the Calabi-Yau structure to \(F\) (Definition \ref{def:spherical}).

\begin{question}\label{question:spherical-vs-relcy}
What is the relationship between the notion of a compatible spherical functor and the notion of a relative Calabi-Yau structure? 
\end{question}

Theorem \ref{intro:thm-relcy-iff-spherical} if a functor \(F\) admits both left and right adjoints then a \emph{weak} Calabi-Yau structure on \(F\) are the same thing as the structure of a compatible spherical functor.

\begin{rmk}It is implied by a conjecture in \cite{KS}(Conjecture 10.2.8) that given a weak Calabi-Yau structure, there exists a \emph{strong} Calabi-Yau structure for which the underlying equivalence $S_{\calx} \simeq \id_{\calx}[d]$ is the given weak CY-structure. The relative analogue of this conjecture would be the statement that every compatible spherical functor can be lifted to a relative Calabi-Yau structure. In light of this, it seems interesting to investigate the relationship between compatible spherical functors and \emph{strong} relative Calabi-Yau structures. This investigation is particularly germane to the problem of constructing Calabi-Yau structures on perverse Schobers \cite{KaSch}, and showing that such a structure gives rise to a Calabi-Yau structure on the Fukaya category of the base manifold with coefficients in the Schober, as is expected \cite{kontsevich-symp}.
\end{rmk}

We return to the problem of constructing a Calabi-Yau structure on the fiber product category \(\Perf(X_1) \times_{\Perf(Z)} \Perf(X_2)\). Let  \(\caly\) be a category equipped with a \(d\)-Calabi-Yau structure \(\phi_{\caly}\) and \(\calz\) be equipped with a \(d\)-Calabi-Yau structure \(\phi_{\calz}\). By an oriented noncommutative cobordism from \((\caly, \phi_{\caly})\) to \((\calz, \phi_{\calz})\), we mean a functor \(F: \calx \rightarrow \caly \times \calz\), together with a relative Calabi-Yau structure on \(F\) for which the Calabi-Yau structure on \(\caly \times \calz\) is given by \((- \phi_{\caly}, \phi_{\calz})\). The discussion in the paragraphs above motivates the following:

\begin{question}\label{question:cobordism}
Can noncommutative oriented cobordisms be glued together? More precisely:
\begin{enumerate}
\item If \(\calx_i \to \calz\) carry relative Calabi-Yau structures, then is \(\calx_1 \times_{\calz} \calx_2\) equipped with a natural Calabi-Yau structure?
\item If each of the functors defining a perverse Schober is equipped with a relative Calabi-Yau structure, does the global sections of the Schober over a Lagrangian skeleton inherit a (relative) Calabi-Yau structure?
\end{enumerate}
\end{question}

One of our main results, Theorem \ref{intro:thm-main-gluing}, is an affirmative answer to the first part of this question. The remaining parts will be the subject of a future investigation.  

\begin{rmk}
The notion of a Pre-CY structure introduced by Kontsevich and Vlassopoulos \cite{KV} long predates the concept of relative Calabi-Yau structures, and appears to be very closely related to the latter. They construct in \emph{loc. cit} a TQFT starting from the data of a Pre-CY structure. It seems very plausible that the gluing constructions involved in the construction of their TQFT are related to the gluing construction of Theorem \ref{intro:thm-main-gluing}. It would be very interesting to understand the precise relationship between the two approaches.
\end{rmk}

One may also consider the inverse problem: given two Fano varieties \(X_1, X_2\) glued along a common anticanonical divisor \(Z\), when may we consider \(X_1\cup_Z X_2\) a degeneration of a smooth Calabi-Yau \(X\)?  For sufficiently nice \(X_i\) and \(Z\), Kawamata and Namikawa \cite{KawaNami} show this is equivalent to the criterion that the normal bundles \(N_{Z/X_1}\) and \(N_{Z/X_2}\) are inverse.  This criterion is then the same as certain spherical twists being inverses.  Following Doran, Harder, and Thompson \cite{DHT} we note that this situation is entirely analogous to the case where we have two Landau-Ginzburg models \(w_i:X_i\to\CC\) with the same generic fibre, which we can then glue to form a single model \(w:X\to\CC\).  If the monodromies on the fibre of each \(X_i\) are inverses, then the monodromy at infinity of \(w\) is trivial, so we can extend this to some \(\tilde{w}:\tilde{X}\to\mathbb{P}^1\).  We may consider \(\Fuk(\tilde{X})\) to be a deformation of \(\Fuk(X)\) by instanton corrections; this latter category is obtained by gluing the Fukaya-Seidel categories of each \(X_i\).  In \cite{DHT} the authors conjecture that the two cases are equivalent under homological mirror symmetry, and show several cases where this is true.  We may consider all these cases within the realm of noncommutative geometry, which lets us treat them on an equal footing.

\begin{question}\label{question:KW}
Is there a noncommutative version of the theory of Tyurin degenerations and Friedman-Kawamata-Namikawa smoothings in the language of noncommutative cobordisms?
\end{question}

\subsection{From noncommutative orientations to shifted symplectic structures.}\label{subsec:nc-symp} Shifted symplectic structures are analogues of symplectic structures in the world of derived geometry. The main novel features in the derived context are:

\begin{itemize}

\item[-] For a differential form (such as a pre-symplectic form) to be \emph{closed} is a \emph{property/condition} in ordinary geometry, while it is an \emph{extra structure} in derived geometry.

\item[-] A \emph{\(n\)-shifted} symplectic structure induces an equivalence \(\mathbb{T}_{X} \rightarrow \mathbb{L}_{X}[n]\) between the tangent complex and a \emph{shift} of the cotangent complex. 
\end{itemize}

Shifted symplectic structures were first introduced in the context of supermanifolds in \cite{AKSZ} in order to construct certain topological field theories. This theory has been vastly generalized in \cite{PTVV} to the world of derived $\infty$-stacks. As in that paper, we will work in an algebro-geometric context; thus, the shifted symplectic structures in this paper are analogues of \emph{holomorphic} symplectic structures, rather than \(C^{\infty}\) ones. Throughout this paper, we will use the powerful language developed in \cite{PTVV}, and we refer the reader to that paper for a more detailed discussion of shifted symplectic geometry.

There are two main ways one construct new shifted symplectic stacks from old ones:

\begin{itemize}
\item[-] Forming the derived mappping stack from an oriented stack to a shifted symplectic stack

\item[-] Forming the derived intersection of two Lagrangians in a shifted symplectic stack
\end{itemize}

One of the motivating goals of this paper is to formulate and exploit noncommutative analogues of these constructions. After recalling these constructions, we will explain how the results proven here contribute toward this goal.

Let \(X\) and \(Y\) be derived Artin stacks, with \(Y\) having an \(n\)-shifted symplectic form.  A \emph{\(d\)-orientation} \([X]\) on \(X\) is a ``fundamental class'' \([X]:\Gamma(X,\calo_X)\to k[-d]\) satisfying certain nondegeneracy properties.  Given such an orientation, we can construct a symplectic form on \(\map(X,Y)\):

\begin{thm}[\cite{PTVV}, Theorem 2.5]\label{thm:PTVV-transgression}  Let \(Y\) be a derived Artin stack, and let \(X\) be an \(\calo\)-compact derived stack with a d-orientation \([X]\).  Assume the derived mapping stack \(\map(X,Y)\) is a derived Artin stack locally of finite presentation over \(k\).  Then we have a transgression map:
\[
\int_{[X]}\ev^*(-):\Symp(Y,n)\to\Symp(\map(X,Y),n-d).
\]
\label{thm:sympmapthm}
\end{thm}

\begin{thm}[\cite{Calaque}, Theorem 2.9]\label{thm:Calaque-transgression}
 Let \(Y\) be a derived Artin stack equipped with an \(n\)-shifted symplectic structure. Let \(X\) be an \(\calo\)-compact derived stack with a d-orientation \([X]\), and let \(i: X \rightarrow X'\) be a morphism equipped with a nondegenerate boundary structure.  Assume that the mapping stacks \(\map(X,Y)\) and \(\map(X',Y)\) are derived Artin stacks locally of finite presentation over \(k\). Then the pullback map:
\[
i^*: \map(X,Y') \to \map(X,Y)
\]
is equipped with a natural Lagrangian structure.

\end{thm}

Several examples of orientations are given in \cite{PTVV}, following Theorem 2.5. Here is one:

\begin{example}\label{ex:CY-variety}
Let \(X\) be a smooth and \emph{compact} Calabi-Yau variety.  If \(X\) has complex dimension \(d\) and we have an isomorphism \(\omega_X\simeq\calo_X\), then projection of \(\Gamma(X,\calo_X)\) onto the degree \(d\) cohomology \(\hml^d(X,\calo_X)[-d]\), followed by the isomorphism
\[
\hml^d(X,\calo_X)\simeq \hml^d(X,\omega_X)\simeq k
\]
provides a map \([X]:\Gamma(X,\calo_X)\to k[-d]\). The fact that this map defines a \(d\)-orientation is essentially the content of Serre duality on \(X\). By Theorem \ref{thm:PTVV-transgression}, \(\perf(X) = \map(X, \perf)\) is equipped with a natural \((2-d)\)-shifted symplectic structure, since \(\perf\) carries a canonical \(2\)-shifted symplectic structure. 
\end{example}

\begin{question}\label{question:symplectic-non-compact}
What can be said in Example \ref{ex:CY-variety} if we drop the hypothesis that \(X\) is compact?
\end{question}

Theorem \ref{intro:thm-meromorphic-shifted-symplectic} is the answer that we offer in this paper. It says that the moduli stack of perfect complexes \emph{with compact support} carries a shifted symplectic structure, provided that the variety can be compactified in a manner compatible with the trivialization of the canonical bundle.

\begin{question}\label{question:nc-mapping-space}
What are the noncommutative analogues of Theorem \ref{thm:PTVV-transgression} and Theorem \ref{thm:Calaque-transgression}?
\end{question}

The expected answer to this question is the following: if a functor \(F: \calx \to \caly\) is equipped with a (strong) relative Calabi-Yau structure, and the moduli spaces of objects in these categories are locally geometric stacks, then the induced map on moduli spaces of objects carries a natural Lagrangian structure. We do not prove this statement here. However, we give evidence for this statement, by adding to the following to the list of relative Calabi-Yau functors for which the statement is true:

\begin{itemize}
\item[-] the functor \(\Perf_c(U) \to 0\) for \(U\) a ``non-compact Calabi-Yau'' (Theorem  \ref{intro:thm-meromorphic-shifted-symplectic}).

\item[-] the functor \(i_*: \Perf(D) \to \Perf(X)\) for \(D\) a smooth divisor in a smooth and compact Calabi-Yau variety \(X\) (Theorem \ref{intro:thm-pushforward-shifted-symplectic}). 
\end{itemize}

\begin{rmk}
Strictly speaking, we only prove that \(i_*: \Perf(D) \to \Perf(X)\) is a compatible spherical functor, which in turn implies that it carries a \emph{weak} relative Calabi-Yau structure. We expect that this can be promoted to a strong structure. One can also interpret Theorem \ref{intro:thm-pushforward-shifted-symplectic} as providing evidence for this claim.
\end{rmk}

The next theorem implies, in particular, that the derived critical locus \(\mathrm{Crit}(f)\) of any function \(f\) carries a shifted symplectic structure, since \(\mathrm{Crit}(f)\) is the intersection of the zero section with the graph of \(df\) in the cotangent bundle.

\begin{thm}[\cite{PTVV}, Theorem 2.9]\label{thm:PTVV-Lag-int}
Let \((X,\omega)\) be a derived stack with \(n\)-shifted symplectic structure, and let \(L \to X\) and \(L' \to X\) be morphisms of derived stacks equipped with Lagrangian structures. Then \(L \times_{X} L'\) is equipped with a natural \((n-1)\)-shifted symplectic structure.
\end{thm}

It is natural to ask whether this theorem has a noncommutative analogue. In fact, we have already encountered this question before (Question \ref{question:cobordism}), in a slightly different context. Theorem \ref{intro:thm-main-gluing}, which says that the fiber product of relative Calabi-Yau functors carries a Calabi-Yau structure, is a noncommutative analogue of Theorem \ref{thm:PTVV-Lag-int}.

\subsection{Main results}\label{subsec:results}

The main results of this paper are partial answers to the questions raised in sections \ref{subsec:categorical-Picard-Lefschetz} through \ref{subsec:nc-symp} above. Our first result, which partially answers Question \ref{question:spherical-vs-relcy}, can be interpreted as expressing a close relationship between the notion of relative orientation in the sense of \textbf{nc}-geometry on the one hand, and the monodromy of Lefschetz fibrations as captured by the notion of a compatible spherical functor (Definition \ref{def:spherical}). 

\begin{thm}\label{intro:thm-relcy-iff-spherical}
Let \(\caly\) be a Calabi-Yau category, and let $F: \calx \rightarrow \caly$ be a functor that admits left and right adjoints. Then \(F\) has a weak relative right Calabi-Yau structure if and only if it has the structure of a compatible spherical functor.
\end{thm}

The remaining sections are devoted to providing evidence for the existence of relative Calabi-Yau structures on compatible spherical functors arising from symplectic geometry and algebraic geometry. In Section \ref{sec:A-model}, we prove the following theorem, giving one possible answer to Question \ref{question:spherical-cy-interaction}, which was raised at the beginning of this introduction:

\begin{thm}\label{intro:thm-A-model-relcy}
Let \(w: X \rightarrow \mathbb{C}\) be a \emph{admissible} Landau-Ginzburg model (Definition \ref{def:admissible-LG}) with generic fiber \(Y\). Then the cap functor \(\cap: \FS(X,w) \rightarrow \Fuk(Y)\) carries a natural relative Calabi-Yau structure (see Theorem \ref{A-model-rel-CY} for a precise statement). 
\end{thm}

It follows from Theorem \ref{intro:thm-relcy-iff-spherical} and the existence of an adjoint to \(\cap\) given by the Orlov functor \(\cup\), that the functor  \(\cap: \FS(X,w) \rightarrow \Fuk(Y)\) carries a natural compatible spherical structure with respect to the natural Calabi-Yau structure on \(\Fuk(Y)\). Zachary Sylvan has independently arrived at a similar result in the context of partially wrapped Fukaya categories \cite{Zack-Sylvan}.

The next theorem is central to this paper. It is an answer to Question \ref{question:cobordism}, a noncommutative analogue of Theorem \ref{thm:PTVV-Lag-int}, and a major step toward the construction of derived noncommutative version of Weinstein's category of Lagrangian spans. This theorem plays an important role in our discussion of Calabi-Yau structures on perverse Schobers in \S\ref{subsec:surgery} (see e.g. Remark \ref{rmk:PSC-CY}).

\begin{thm}\label{intro:thm-main-gluing}
Suppose \(\calx,\caly,\) and \(\mathcal{Z}\) are categories with right \(d\)-Calabi-Yau structures, and suppose \(\mathcal{U}\to\calx\times\caly\) and \(\mathcal{V}\to\caly\times\mathcal{Z}\) have relative Calabi-Yau structures.  Let \(\mathcal{W}=\mathcal{U}\times_{\caly}\mathcal{V}\) be the pullback.  Then \(\mathcal{W}\to\calx\times\mathcal{Z}\) has a natural relative Calabi-Yau structure.
\end{thm}

The derived moduli stacks of objects in \(d\)-Calabi-Yau categories carry \((2-d)\)-shifted symplectic structures. Furthermore, relative Calabi-Yau structures on categories induce Lagrangian structures on the moduli space of objects. These statements provide the link between the results about relative Calabi-Yau structures proven in this paper, and the main results of \cite{PTVV, Calaque}. The proofs of these statements are similar to the proof of the statement that the moduli of perfect complexes on a Calabi-Yau variety carries a shifted symplectic structures \cite{PTVV}, and will appear elsewhere. In this paper, we will content ourselves with proving the following two theorems about the existence of shifted symplectic structures and Lagrangian structures, which provide evidence for the general statement above, as explained in \S\ref{subsec:nc-symp}.

\begin{thm}\label{intro:thm-meromorphic-shifted-symplectic}
Let \(X\) be a smooth \(d\)-dimensional variety, \(D\subset X\) a divisor, and \(U=X\backslash D\).  Let \(\alpha\) be a meromorphic section of \(\calo_X(K_X)\) which is holomorphic nonvanishing on \(U\).  Then \(\alpha\) induces a \((2-d)\)-shifted symplectic structure on \(\perf_c(U)\).
\end{thm}

Since \(\perf_c(U)\) is a \(d\)-Calabi-Yau category, this is a special case of the existence of shifted symplectic structures on moduli spaces of objects in categories.

\begin{thm}\label{intro:thm-pushforward-shifted-symplectic}
Let $D$ be a smooth divisor in a smooth and proper Calabi-Yau manifold $X$, and let $i: D \rightarrow X$ denote the inclusion. The map on moduli spaces of perfect complexes induced by the pushforward functor $i_{*}: \Perf(D) \rightarrow \Perf(X)$ carries a Lagrangian structure. 
\end{thm}

\subsection{Organization of the paper}  

This document is organized in a modular fashion, and for the most part, the individual sections can be read independently of each other.  

Section \ref{sec:spherical} is devoted to giving the definitions of compatible spherical functors (Definition \ref{def:spherical}) and weak relative Calabi-Yau structures (Definition \ref{def:relcy}), and proving Theorem \ref{intro:thm-relcy-iff-spherical}, which says that these two notions are in fact essentially equivalent. Readers who have familiarized themselves with these two definitions and the statement of Theorem \ref{intro:thm-relcy-iff-spherical}, can skip to other sections, which are all logically independent of each other. 

Section \ref{sec:cy-psc} explains how compatible spherical functors can be glued together to produce Calabi-Yau structures on global sections of perverse Schobers. The key result of this section is Theorem \ref{intro:thm-main-gluing}, to which the entire subsection \S\ref{subsec:main-gluing-theorem} is devoted. \S\ref{subsec:surgery} explains how this theorem can be used to construct CY-structures on global sections of perverse Schobers (Remark \ref{rmk:PSC-CY}). The rest of \S\ref{subsec:surgery} is devoted to illustrating how ``categorical surgery'' (the modification of a perverse Schober by changing the gluing data) can be carried out, and the effect this has on monodromy, through explicit computations. In particular, we show in Example \ref{example:Kronecker} how the Kronecker quiver with \(n\) arrows can be obtained via this procedure, starting from the mirror of \(\mathbb{P}^{2}\).

The sole purpose of Section \ref{sec:A-model} is outline a proof of Theorem \ref{intro:thm-A-model-relcy}, which asserts the existence of a relative Calabi-Yau structure on the natural functor \( \cap: \FS(X,w) \rightarrow \Fuk(Y)\), where \(Y\) is a generic smooth fiber of a good Picard-Lefschetz fibration \(w: X \rightarrow \mathbb{C}\). 

Section \ref{sec:shifted-symplectic} is devoted to shifted symplectic structures on certain derived stacks, and more specifically, to the proofs of Theorem \ref{intro:thm-meromorphic-shifted-symplectic} and Theorem \ref{intro:thm-pushforward-shifted-symplectic}. The brief review of shifted symplectic structures at the beginning of this section is intended mainly to fix notation, and the reader is referred to \cite{PTVV} for the necessary background on the subject. 

The final section, \S\ref{sec:future-directions}, is devoted to directions for further research inspired by the current paper. It outlines an approach to a theory of derived hyperk\"ahler geometry and categorical hyperk\"ahler geometry using the twistor approach to hyperk\"ahler manifolds.

\subsection{Notation and conventions}
Throughout the paper, we will use the following notation:

\begin{itemize}

\item[-] $k$ is a fixed ground field.

\item[-] $\calx$, $\caly$, ... denote small $k$-linear stable $\infty$-categories.

\item[-] $X$, $Y$, ... usually denote derived stacks over $k$ or symplectic manifolds.

\item[-] $\HH_{*}(\calx)$ and $\HC_{*}(\calx)$ denote, respectively, the \emph{chain complexes computing} the Hochschild and cyclic homologies of $\calx$.

\item[-] $\Map(X,Y)$ is the mapping space in an $\infty$-category.

\item[-] $\map(X,Y)$ is the internal hom in a cartesian closed $\infty$-category.

\item[-] $\Mor(x,y)$ is the ($k$-module) spectrum valued internal hom in a ($k$-linear) stable $\infty$-category. In particular, $\Omega^{\infty}(\Mor(x,y)) \simeq \Map(x,y)$. 

\item[-] $\Perf(X)$ denotes the $\infty$-category of perfect complexes on a derived stack $X$, while $\perf(X) := \map(X, \perf)$ is the moduli stack of objects in $\Perf(X)$. 

\item[-] Unless explicitly stated otherwise, \emph{category} means \emph{$k$-linear stable $\infty$-category}, and all constructions, such as limits and colimits, should be understood in the $\infty$-categorical sense. \\

\end{itemize}

\subsection*{Acknowledgments}
We would like to thank Mohammed Abouzaid, Denis Auroux, Damien Calaque, Sheel Ganatra, Andrew Harder, Maxim Kontsevich, Tony Pantev, Nick Rozenblyum, Carlos Simpson, Zachary Sylvan, Hiro Lee Tanaka and Yiannis Vlassopoulos for useful conversations related to the subject matter of this paper. We are especially grateful to Denis Auroux, Mohammed Abouzaid and Sheel Ganatra for patiently answering our many questions about the A-model, and to Tony Pantev and Carlos Simpson for a careful reading of our outline of derived hyperk\"ahler geometry, and for numerous invaluable comments and suggestions about the same. The authors were supported by Simons research grant, NSF DMS 150908, ERC Gemis, DMS1265230, DMS1201475 OISE1242272 PASI, Simons collaborative Grant - HMS and HSE Megagrant. The first named author is supported by a Simons Investigators Award, and is partially supported by the Laboratory of Mirror Symmetry NRU HSE, RF Government grant, ag.\ No.14.641.31.0001.

\section{Compatible spherical functors and relative Calabi-Yau structures}\label{sec:spherical}

The purpose of this section is twofold. First, we give the definitions of spherical functors compatible with a given Calabi-Yau structure (Definition \ref{def:spherical}) and weak relative right Calabi-Yau structures (Definition \ref{def:relcy}). Second, we prove that these two notions are equivalent; this is the content of Proposition \ref{prop:relcy-implies-spherical} and Theorem \ref{thm:spherical-implies-relcy}. We begin by recalling some relevant definitions. Our notation and the discussion that follows closely mirrors that in \cite{Ganatra-Perutz-Sheridan}, to which we refer the reader for further details.

\begin{notation}\label{notation:modules}
Let \(k\) be a field, and let \(\calx\) be a \(k\)-linear \(\infty\)-category. 

\begin{enumerate}
\item The \(k\)-linear \(\infty\) category \(\Fun(\calx, \Mod_k) = \map_{\catk}(\calx, \Mod_k)\) of functors from \(\calx\) to the \(\infty\)-category of \(k\)-module spectra is called the category of \emph{left \(\calx\)-modules}, and denoted \(\Mod_{\calx}\). 

\item The category of \emph{right} modules is obtained by replacing \(\calx\) with \(\calx^{\mathrm{op}}\) in the definition above. 

\item We will denote by \(\calx^e\) the \(\infty\)-category \(\calx \otimes \calx^{\mathrm{op}}\). The category of left \(\calx^e\) modules is also called the category of \(\calx\)-\(\calx\)-bimodules. 

\item We will write \(\Mor_{\calx}(M,N)\) for \(\Mor_{\Fun(\calx, \Mod_k)}(M,N) = \Mor_{\Mod_{\calx}}(M,N)\) when the meaning is clear from context. 

\end{enumerate}
\end{notation}

\begin{defn}\label{def:id-duals}
Let \(\calx\) be a \(k\)-linear \(\infty\)-category.

\begin{enumerate}
\item The \emph{diagonal bimodule} or \emph{identity bimodule} \(\calx_{\Delta}\) is the functor
\(
\calx_{\Delta}: \calx^{\mathrm{op}} \otimes \calx \rightarrow \Mod_k
\)
defined by
\[
\calx_{\Delta}(x,y) := \Mor_{\calx}(x,y)
\]
where the right hand side is the \(k\)-module spectrum of maps from \(x\) to \(y\) in the category \(\calx\).

\item The right dual of the identity bimodule is the functor
\(
\calx^{\vee}: \calx \otimes \calx^{\mathrm{op}} \rightarrow \Mod_k
\)
defined by 
\[
\calx^{\vee}(x,y) := \Mor_{\calx}(x,y)^{\vee} = \Mor_{k}(\Mor_{\calx}(x,y),k)
\]
for all objects \(x\) and \(y\) in \(\calx\). 
\end{enumerate}

\end{defn}

\begin{defn}\label{def:locally-proper}
A \(k\)-linear \(\infty\)-category \(\calx\) is \emph{locally proper} if for all objects \(x\), \(y\) in \(\calx\), the mapping spectrum \(\Mor(x,y)\) is a compact object of \(\Mod_k\), i.e., if \(\Mor(x,y) \in \Perf(k)\). 
\end{defn}

\begin{rmk}\label{rmk:right-dual-is-Serre}
If \(\calx\) is locally proper, then a functor representing the bimodule \(\calx^{\vee}\) of Definition \ref{def:id-duals} is the same thing as a Serre functor for \(\calx\). Indeed, suppose that the bimodule \(\calx^{\vee}\) is representable by a functor \(S_{\calx}: \calx \rightarrow \calx\), i.e., we have an equivalence of functors
\[
\calx^{\vee}(-,-) \simeq \Mor(-, S_{\calx}(-))
\]
Then by definition of \(\calx^{\vee} \), we have an equivalence \(\Mor(x,y)^{\vee} \simeq \Mor(y, S_{\calx}(x))\) for all \(x\), \(y\) in \(\calx\). Using the fact that \(\Mor(x,y)\) is a perfect module, we have  \(\Mor(x,y) \simeq \Mor(x,y)^{\vee \vee} \simeq \Mor(y, S_{\calx}(x))^{\vee} \) showing that \(S_{\calx}\) is a Serre functor. 
\end{rmk}

\begin{rmk}\label{rmk:functors-to-HH}
Recall that the Hochschild homology complex \(\HH_*(\calx)\) of a small \(k\)-linear \(\infty\)-category can be computed by the formula
\[
\HH_*(\calx) \simeq \calx_{\Delta} \otimes_{\calx \otimes \calx^{\mathrm{op}}} \calx_{\Delta}
\]

Using this formula for \(\HH_*(\calx)\), and the standard tensor-hom adjunction, we have the following string of equivalences of chain complexes:
\[
\HH_*(\calx)^{\vee} \simeq \Mor_k( \calx_{\Delta} \otimes_{\calx \otimes \calx^{\mathrm{op}}} \calx_{\Delta}, k) \simeq \Mor_{\calx \otimes \calx^{\mathrm{op}}}(\calx_{\Delta}, \Mor_k(\calx_{\Delta},k)) \simeq \Mor_{\calx \otimes \calx^{\mathrm{op}}}(\calx_{\Delta}, \calx^{\vee})
\]

The bimodule \(\calx_{\Delta}\) is representable by the identity functor on \(\calx\). If \(\calx^{\vee}\) is representable, then it is representable by the Serre functor \(S_{\calx}\), so we have \(\Mor_{\calx \otimes \calx^{\mathrm{op}}}(\calx_{\Delta}, \calx^{\vee}) \simeq \Mor_{\Fun(\calx, \calx)}(\id_{\calx}, S_{\calx})\). Thus we see that the data of map \(\id_{\calx}[d] \rightarrow S_{\calx}\) is equivalent to the data of a map \(\HH_*(\calx) \rightarrow k[-d]\). 
\end{rmk}

The Hochschild chain complex \(\HH_*(\calx)\) of a category carries a natural \(S^1\)-action that is manifest in the definition of Hochschild homology via the cyclic bar complex, or via topological field theories \cite{Lurie-TFT}. An \(S^1\)-action on a chain complex over \(k\) is the same thing as the structure of a module over \(C^*(S^1,k) = k[B]/B^2\); in other words a differential of homological degree \(1\). The circle action on \(\HH_*(\calx)\) is given by Connes B-operator. In order to have a definition of Calabi-Yau structures that is adequate for producing oriented TFTs, and for producing shifted symplectic structures on the moduli of objects, it is necessary to incorporate the circle action into the definition of a Calabi-Yau structure. Recall the definition of a Calabi-Yau structure introduced in \cite{KS}, as described in \cite{Ganatra-Perutz-Sheridan}:

\begin{defn} \label{def:cy}
Let \(\caly\) be a locally proper category. Let \(\Xi: \HH_*(\caly)^{\vee} \rightarrow  \Mor_{\caly \otimes \caly^{\mathrm{op}}}(\caly_{\Delta}, \caly^{\vee})\) be the natural morphism described in Remark \ref{rmk:functors-to-HH}.

\begin{enumerate}

\item A \emph{weak \(d\)-dimensional right Calabi-Yau structure} is map \(\phi:\HH_*(\caly)\to k[-d]\) such that \(\Xi(\phi):\caly_{\Delta}[d]\to\caly^{\vee}\) is a weak equivalence.

\item  A \emph{\(d\)-dimensional right Calabi-Yau structure} on \(\caly\) is a morphism \(\tilde{\phi}:\HH_*(\caly)_{S^1}\to k[-d]\) such that the composite map \(\HH_*(\caly) \rightarrow \HH_*(\caly)_{S^1} \rightarrow k[-d]\) is a weak Calabi-Yau structure.

\end{enumerate}
\end{defn}

\begin{rmk}
Note that \(\Mor_{k}(\HH_*(\caly)_{S^1},k)\simeq\Mor_{S^1}(\HH_*(\caly),k)\), where on the right hand side \(k\) has the trivial \(S^1\) structure.  Thus we may consider \(\tilde{\phi}\) as an \(S^1\)-equivariant map \(\HH_*(\caly)\to k[-d]\). If \(\caly\) is smooth and proper, we can identify \(\caly^{\vee}\) with the Serre functor \(S_{\caly}\), and a (weak) right Calabi-Yau \(\phi\) structure is equivalent to the data of the equivalence of functors \( \id_{\caly}[d] \simeq S_{\caly}\) corresponding to \(\Xi(\phi)\) under the identification \(\Mor_{\caly \otimes \caly^{\mathrm{op}}}(\caly_{\Delta}, \caly^{\vee}) \simeq \Mor_{\Fun(\caly, \caly)}(\id_{\caly}, S_{\caly})\). 
\end{rmk}

\begin{rmk}
Note that the natural map from \(\HH_*(\caly)\) to the \emph{homotopy} orbits \(\mathrm{HC}_*(\caly) := \HH_*(\caly)_{S^1}\) has a highly non-trivial cofiber in general. In particular, the choosing a lift of a weak Calabi-Yau structure to a Calabi-Yau structure involves giving an ``infinite tower of higher coherence data''.

\end{rmk}

We can now formulate the definition of a relative Calabi-Yau structure, which is a noncommutative analogue of the notion of boundary structure introduced by Calaque \cite{Calaque}. Our definition follows a suggestion of To\"en \cite{Toen-DAG}.

\begin{defn}\label{def:isotropy}
Let \(\calx\) and \(\caly\) be locally proper \(k\)-linear \(\infty\)-categories and let \(F\: \calx \rightarrow \caly\) be an \(\infty\)-functor. 

\begin{enumerate}

\item Suppose that \(\caly\) is equipped with a \emph{weak} right Calabi-Yau structure \(\phi: \HH_*(\caly) \rightarrow k[-d]\) (Definition \ref{def:cy}, (1)). A \emph{weak isotropy structure} for \(F\) with respect to \(\phi\) is a functor \(\Delta^1 \times \Delta^1 \rightarrow \Perf(k)\) as follows:
\begin{equation}
\label{diag:weak-isotropy} \tag{*}
\xymatrixcolsep{5pc} \xymatrixrowsep{5pc} \xymatrix{
\HH_*(\calx) \ar[d] \ar[r]^{\HH_*(F)} & \HH_*(\caly) \ar[d]^{\phi} \\ 
0 \ar[r] & k[-d]
} 
\end{equation}
witnessing a nullhomotopy of \(\phi \circ \HH_*(F)\).

\item Suppose that \(\caly\) is equipped with a right Calabi-Yau structure \(\tilde{\phi}: \HC_*(\caly):= \HH_*(\caly)_{hS^1} \rightarrow k[-d]\) (Definition \ref{def:cy}, (2)). A (strong) \emph{isotropy structure} for \(F\) with respect to \(\phi\) is a 2-cell witnessing the commutativity of the following diagram:
\begin{equation}
\label{diag:strong-isotropy} \tag{**}
\xymatrixcolsep{5pc} \xymatrixrowsep{5pc} \xymatrix{
\HC_*(\calx) \ar[d] \ar[r]^{\HC_*(F)} & \HC_*(\caly) \ar[d]^{\tilde{\phi}} \\
0 \ar[r] & k[-d]
}
\end{equation}
\end{enumerate}
\end{defn}

\begin{rmk}\label{rmk:strong-gives-weak}
Every strong isotropy structure has an underlying weak isotropy structure obtained by forgetting the \(S^1\)-equivariance data in the nullhomotopy. Or, to say the same thing differently, one obtains the underlying weak isotropy structure by concatenating the diagram \eqref{diag:strong-isotropy} with the natural diagram
\begin{equation}
\label{diag:HH-HC} \tag{***}
\xymatrixcolsep{5pc} \xymatrixrowsep{5pc} \xymatrix{
\HH_*(\calx) \ar[r]^{\HH_*(F)} \ar[d] & \HH_*(\caly) \ar[d] \\
\HC_*(\calx) \ar[r]_{\HC_*(F)} & \HC_*(\caly) \\ 
}
\end{equation}
to obtain a functor \(\Delta^2 \times \Delta^1 \rightarrow \Perf(k)\) whose outer square is the required diagram \(\Delta^1 \times \Delta^1 \rightarrow \Perf(k)\) of the form \eqref{diag:weak-isotropy}. The vertical maps in \eqref{diag:HH-HC} are the canonical quotient maps \(\HH_*(-) \rightarrow \HH_*(-)_{hS^1}\).  
\end{rmk}

\begin{construction}\label{const:fiber-sequence}
Let \(\calx\) and \(\caly\) be locally proper \(k\)-linear \(\infty\)-categories and let \(F\: \calx \rightarrow \caly\) be an \(\infty\)-functor. For arbitrary objects \(x,y\) in \(\calx\), consider the diagram \eqref{diag:nondeg-part}. 

\begin{equation}
\label{diag:nondeg-part}\tag{\(\dagger\)}
\begin{aligned}
\xymatrixcolsep{5pc} \xymatrixrowsep{5pc}\xymatrix{
\Mor_{\calx}(x,y) \otimes \Mor_{\calx}(y,x) \ar[d]_{\id} \ar[r]^{F \otimes \id} & \Mor_{\caly}(Fx, Fy) \otimes \Mor_{\calx}(y,x) \ar[d]^{\id \otimes F}\\
 \Mor_{\calx}(x,y) \otimes \Mor_{\calx}(y,x) \ar[d]_{m_{\calx}} \ar[r]^{F \otimes F} & \Mor_{\caly}(Fx, Fy) \otimes \Mor_{\caly}(Fy, Fx) \ar[d]^{m_{\caly}} \\ 
\Mor_{\calx}(x,x) \ar[r]^{F} \ar[d]_{\tr_x} & \Mor_{\caly}(Fx,Fx) \ar[d]^{\tr_{Fx}} \\
\HH_*(\calx) \ar[r]^{\HH_*(F)}  & \HH_*(\caly) \\
}
\end{aligned}
\end{equation}

Here the vertical maps in the middle square are given by multiplication in the category, and the vertical maps in the bottom square are the universal trace maps (boundary-bulk maps). The upper square can clearly be promoted to a commutative square, i.e., a homotopy \((\id \otimes F) \circ (F \otimes \id) \simeq (F \otimes F) \circ (\id \otimes \id)\). Since \(F\) is an \(\infty\)-functor, for any choice of multiplication maps in \(\calx\) and \(\caly\), the middle square is promoted to a commutative square in a canonical way. Similarly, functoriality of \(\HH\) and the boundary-bulk map give us a 2-cell witnessing the commutativity of the bottom square. In fact, the diagram above can be promoted to a functor \(\Delta^3 \times \Delta^1 \rightarrow \Perf(k)\).

Now suppose that the functor \(F: \calx \rightarrow \caly\) is equipped with a weak isotropy structure. Gluing the commutative square \eqref{diag:weak-isotropy} to the outer commutative square of the diagram \eqref{diag:nondeg-part}, one obtains a commutative square

\begin{equation}
\label{diag:adj-non-deg} \tag{\(\dagger \dagger\)}
\xymatrixrowsep{5pc} \xymatrix{
\Mor_{\calx}(x,y) \otimes \Mor_{\calx}(y,x) \ar[r]^{F \otimes \id} \ar[d] & \Mor_{\caly}(Fx, Fy) \otimes \Mor_{\calx}(y,x) \ar[d] \\
0 \ar[r] & k[-d] \\
}
\end{equation}

By the adjunction between \(\otimes\) and \(\Mor\) in \(\Mod_k\), this gives rise to a commutative square

\begin{equation}
\label{diag:non-deg} \tag{\(\dagger \dagger \dagger\)}
\xymatrixrowsep{5pc} \xymatrix{
\Mor_{\calx}(x,y) \ar[r]^{F} \ar[d] & \Mor_{\caly}(Fx, Fy) \ar[d] \\
0 \ar[r] & \Mor_{\calx}(y,x)^{\vee}[-d] \\
}
\end{equation}

\end{construction}

\begin{rmk}\label{rmk:nondeg-triangle}
The right hand vertical map in \eqref{diag:non-deg} sits in a diagram
\[
\xymatrixcolsep{5pc} \xymatrix{
\Mor_{\caly}(Fx.Fy) \ar[r]^{\phi^{\sharp}_{\caly}} \ar[dr] & \Mor_{\caly}(Fy, Fx)^{\vee}[-d] \ar[d]^{F^{\vee}} \\
& \Mor_{\calx}(y.x)^{\vee}[-d] \\
}
\]
realizing it as a composition of the \(k\)-linear dual of \(F\) and the Serre duality equivalence on \(\caly\). 
\end{rmk}

\begin{defn}\label{def:relcy}
Let \(\calx\) and \(\caly\) be locally proper \(k\)-linear stable \(\infty\)-categories, and let \(F: \calx \rightarrow \caly\) be an \(\infty\)-functor. 
\begin{enumerate}
\item A \emph{weak relative Calabi-Yau structure} of dimension \(d\) on \(F\) is a weak isotropy structure \eqref{diag:weak-isotropy} (Definition \ref{def:isotropy}) for which the induced diagram \eqref{diag:non-deg} given by Construction \ref{const:fiber-sequence} is a pushout-pullback square (in which case we will sat that the isotropy data satisfies the \emph{non-degeneracy condition}). 

\item A \emph{strong relative Calabi-Yau structure} of dimension \(d\) on \(F\) is a strong isotropy structure \eqref{diag:strong-isotropy} (Definition \ref{def:isotropy}) for which the underlying weak isotropy structure of Remark \ref{rmk:strong-gives-weak}  defines a weak relative Calabi-Yau structure.  

\end{enumerate}
\end{defn}

We now recall the definition of a spherical functor as given in \cite{AL}, and describe an additional structure that may exist on spherical functors.

\begin{defn}\label{def:spherical}
Let \(\calx\) and \(\caly\) be categories and \(F:\calx\to\caly\) a functor, with left and right adjoints \(F^*\) and \(F^!\).  Let  \(\rho:F^!F\to T\) be the cofiber map of the unit \(\eta_{F^!F}: \id_{\calx} \rightarrow F^! \circ F\) of the adjunction \(F \dashv F^!\), and let  \(\eta_{FF^*}\) be the unit of the adjunction \( F^* \dashv F\). Then \(F\) is \emph{spherical} if:
\begin{enumerate}
\item The cofiber of the unit \(T:=\cofib(\id_{\calx}\to F^!\circ F)\) is an autoequivalence of \(\calx\).
\item The composition \(\tau=(\rho F^*)\circ F^!\eta_{FF^*}\) 

\begin{equation}\label{eq:sphiso}
\xymatrixcolsep{5pc} \xymatrix{
F^!\ar[r]^{F^!\eta_{FF^*}} \ar[dr]_{\tau} &  F^!\circ F\circ F^* \ar[d]^{\rho F^*} \\
& T\circ F^* \\
}
\end{equation}
is an isomorphism of functors.\\

\end{enumerate}

Furthermore, suppose \(\caly\) has a Calabi-Yau structure \(\phi\) of dimension \(d\), and \(\calx\) has a Serre functor \(S_{\calx}\).  By general results on adjoint functors, we have \(F^!\simeq S_{\calx}F^*S_{\caly}^{-1}\).  Let \(\kappa\) denote the isomorphism \(F^!\simeq S_{\calx}F^*S_{\caly}^{-1}\simeq S_{\calx}[-d]F^*\).  Then a \emph{\(\phi\)-compatible structure} on \(F\) is an isomorphism \(\alpha:T\simeq S_{\calx}[-d]\) such that the map \(F^!\xrightarrow{\tau} TF^*\xrightarrow{\alpha F^*} S_{\calx}[-d]F^*\) is homotopic to \(\kappa\).
\begin{equation}
\xymatrix{
F^!\ar[r]\ar@{=}[d]&F^!FF^*\ar[r]&TF^*\ar[d]^{\alpha F^*}\\
F^!\ar[r]^(0.40){\sim}&S_{\calx}F^*S_{\caly}^{-1}\ar[r]^\sim&S_{\calx}[-d]F^*
}
\label{eq:sphcond}
\end{equation}

A spherical functor with \(\phi\)-compatible structure will be called an \(\phi\)-compatible spherical functor.  When \(\phi\) is understood we will say compatible spherical functor.
\end{defn}

\begin{rmk}
The definition of spherical functors above coincides with that in \cite{AL}.  The first compatibility condition appears in \cite{KKP}, as part of their definition of spherical functors.

If \(F\) is a compatible spherical functor, the composition \(F^!F\to T\simeq S_{\calx}[-d]\) is uniquely determined.  By the proof of Theorem 2.13 in \cite{Ku15}, the morphism \(F^!F\to S_{\calx}[-d]\) is the composition
\[
F^!F\xrightarrow{\kappa F} S_{\calx}F^*F\xrightarrow{S_{\calx}\eta} S_{\calx}.
\]
Thus \(\alpha\) is determined by a homotopy between the composition \(\id_{\calx}\to F^!F\to S_{\calx}\) and \(0\).

Let \(\caly\) be a Calabi-Yau structure of dimension \(d\).  For a general spherical functor \(F:\calx\to\caly\) we have \(T\circ F^*\simeq F^!\simeq S_X[-d]\circ F^*\) but we cannot a priori show \(T\simeq S_{\calx}[-d]\).
\end{rmk}

We note a condition that is often easier to check than (\ref{eq:sphcond}) above
\begin{lem}
Let \(F:\calx\to\caly\) be a spherical functor, such that \(\caly\) has a Calabi-Yau structure \(\phi\) of dimension \(d\), and \(\calx\) has a Serre functor \(S_{\calx}\).  Let \(\alpha:T\simeq S_\calx[-d]\) be an isomorphism.  Suppose that \(\alpha\rho=(S_{\calx}[-d]\epsilon)\circ(\kappa F)\), where \(\epsilon\) is the counit of \(F^*F\).
\[
\xymatrix{
F^!F\ar[r]^{\rho}\ar[d]^{\kappa F}&T\ar[d]^\alpha\\
S_{\calx}[-d]F^*F\ar[r]^{S_{\calx}[-d]\epsilon}&S_{\calx}[-d]
}
\]
Then \(\alpha\) is a compatible spherical structure for \(F\).
\end{lem}
\begin{proof}
It remains to show that the diagram (\ref{eq:sphcond}) commutes.  We have a diagram
\[
\xymatrix{
F^!\ar[r]\ar[d]^{\kappa}&F^!FF^*\ar[r]\ar[d]^{\kappa FF^*}&TF^*\ar[d]^{\alpha F^*}\\
S_{\calx}[-d]F^*\ar[r]&S_{\calx}[-d]F^*FF^*\ar[r]&S_{\calx}[-d]F^*
}
\]
The left square clearly commutes and the right square commutes by our assumption.  The morphism \(S_{\calx}[-d]F^*\to S_{\calx}[-d]F^*\) is the identity by the unit-counit relations, so taking the bottom path yields the map \(\kappa:F^!\to S_{\calx}[-d]F^*\).  But the top map is precisely \(\tau\), so we get \((\alpha F^*)\circ\tau\sim \kappa\), as desired.
\end{proof}

\begin{example}\label{spherical-examples}
We have three natural families of compatible spherical functors
:
\begin{enumerate}
\item Let \(Y\) be a smooth Calabi-Yau variety of dimension \(d\) and \(i:D\to Y\) a smooth divisor.  Then \(F=i_*:\Perf(D)\to \Perf(Y)\) is spherical with \(F^*=i^*\) and \(F^!=i^!\).  For \(E\in\Perf(D)\), we have \(i^!i_*E\simeq E\oplus E(D)[-1]\), where \(E(D)=E\otimes \calo_Y(D)|_{D}\).  The unit \(E\to i^!i_*E\) is inclusion into the first factor, so that
\[
TE=\cofib(E\to i^!i_*E)\simeq E(D)[-1]\simeq (E\otimes\omega_D[d-1])[-d]\simeq S_X(E)[-d],
\]
where \(\calo_X(D)\simeq\omega_D\) by the adjunction formula.  This gives an isomorphism \(T\simeq S_X[-d]\).  Note that \(\rho:i^!i_*E\simeq E\oplus E(D)[-1]\to TE\) is just the projection onto the second factor.  We also have \(i^*i_*E\simeq E(-D)[1]\oplus E\), and the counit \(i^*i_*E\to E\) is also projection onto the second factor.  Thus the diagram
\[
\xymatrix{
i^!i_*\ar[r]\ar[d]&(-)\otimes \omega_D[-1]\ar[d]\\
S_X[-d]i^*i_*\ar[r]&S_X[-d]
}
\]
commutes and we have a compatible spherical structure.
\item Let \(X\) be a Fano variety of dimension \(d+1\) and \(a:Y\to X\) a smooth anticanonical divisor.  Then \(F=a^*:\Perf(X)\to \Perf(Y)\) is spherical with \(F^*=a_!=a_*(-)(Y)[-1]\) and \(F^!=a_*\).  For \(E\in\Perf(X)\), \(a_*a^*E\simeq \tot(E(-Y)\to E)\), where \(E(-Y)\to E\) is the defining section of \(Y\), and the map \(E\to a_*a^*E\) is given by the diagram
\[
\xymatrix{
0\ar[r]\ar[d]&E\ar[d]^{id}\\
E(-Y)\ar[r]&E
}.
\]
Then \(TE=\cofib(E\to a_*Ea^*E)\simeq E(-Y)[1]\simeq (E\otimes \omega_X[d+1])[-d]\simeq S_{X}[-d](E)\), where we have again used the adjunction formula.  Thus we have an isomorphism \(T\simeq S_X[-d]\).  Note that the map \(a_*a^*E\to TE\) projects onto the \(E(-Y)\) term, sending the \(E\) term to \(0\).  We also check that \(a_!a^*E\simeq \tot(E\to E(Y))\).  Again, the map \(a_!a^*E\to E\) projects onto the \(E\) term, sending \(E(Y)\) to \(0\).  Thus we have a commutative diagram
\[
\xymatrix{
a_*a^*\ar[r]\ar[d]&(-)\otimes \omega_X[1]\ar[d]\\
S_X[-d]a_!a^*\ar[r]&S_X[-d]
}
\]
commutes and we have a compatible spherical structure.
\item Let \(w: X\to\CC\) be a Lefschetz fibration with smooth fiber \(Y\).  Then \(F=\cap:\FS(X,w)\to \Fuk(Y)\) is spherical with \(F^!=\cup\) and \(F^*=\cup\circ S^{-1}[1]\) \cite{AS}.  We expect this is a compatible spherical functor.  
\end{enumerate}

\end{example}

Let \(\calx\) and \(\caly\) be proper categories, let \(F:\calx\to\caly\) be a functor, and  let \(\tilde{\phi}:\HH_*(\caly)_{S^1}\to k[-d]\) be a right Calabi-Yau structure.  Recall (Definition \ref{def:relcy}) that a right relative Calabi-Yau structure on \(F\) is a homotopy \(F^*\tilde{\phi}\sim 0\) such that for all \(x,y\in \calx\), the sequence
\[
\Mor_{\calx}(x,y)\to\Mor_{\caly}(Fx, Fy)\simeq \Mor_{\caly}(Fy, Fx)^\vee[-d]\to \Mor_{\calx}(y,x)^\vee[-d]
\]
is a fiber sequence.  Here the middle equivalence is from the equivalence \(\caly_{\Delta}(Fx, Fy)\simeq \caly^{\vee}(Fx, Fy)[-d]\). Similarly, if \(\phi:\HH_*(\caly)\to k[-d]\) is a weak right Calabi-Yau structure, then a weak right relative Calabi-Yau structure on \(F\) is a homotopy \(F^*\phi\sim0\) such that the above sequence is a fiber sequence. Here is what we can say about the existence of relative Calabi-Yau structures for the three families of functors described in Example \ref{spherical-examples}:

\begin{example}\label{relcy-examples}
\begin{enumerate}
\item Let \(Y\) be a smooth Calabi-Yau variety and \(i:D\to Y\) a smooth divisor.  Then the map  \(i_*:\perf(D)\to \perf(Y)\) on derived stacks induced by the functor  \(i_*:\Perf(D)\to \Perf(Y)\) carries a Lagrangian structure. This suggests that the functor \(i_*:\Perf(D)\to \Perf(Y)\) should carry a relative Calabi-Yau structure.

\item Let \(X\) be a Fano variety and \(a:Y\to X\) a smooth anticanonical divisor.  Then the functor \(a^*:\Perf(X)\to \Perf(Y)\) has a relative Calabi-Yau structure. \cite{Calaque}

\item Let \(w:X\to\CC\) be a Lefschetz fibration with smooth fiber \(Y\).  Then the functor \(\cap:\FS(X,w)\to \Fuk(Y)\) has a weak relative Calabi-Yau structure.By mirror symmetry (this example is mirror to the previous one), we expect this functor to carry a relative Calabi-Yau structure.
\end{enumerate}
\end{example}

At this point it is natural to conjecture that a compatible spherical functor will always carry a relative Calabi-Yau structure.  We will show weaker results:

\begin{prop}\label{prop:relcy-implies-spherical}
Let \(\calx\) and \(\caly\) be smooth and proper categories and suppose \(F: \calx \rightarrow \caly\) has a \emph{weak} relative right Calabi-Yau structure. Suppose that $F$ admits left and right adjoints. Then $F$ is a compatible spherical functor

\end{prop}

\begin{proof}
Let \(F^*\) and \(F^!\) be left and right adjoints of \(F\), respectively.  Letting \(T=\cofib(\id_{\calx}\to F^!\circ F)\), we give an isomorphism \(T\simeq S_\calx[-d]\).  For all \(x,y\in\calx\), we have
\[
\xymatrix{
\Mor(x,y)\ar[r]\ar@{=}[d]&\Mor(x,F^!Fy)\ar[r]\ar[d]^[@]{\sim}&\Mor(x,Ty)\ar[d]\\
\Mor(x,y)\ar[r]&\Mor(Fx,Fy)\ar[r]&\Mor_{\calx}(x,S_{\calx}[-d]y),
}
\]
where the bottom row is the fiber sequence given by the relative Calabi-Yau structure on \(F\).  Since the first two vertical maps are isomorphisms, so is the third; this gives \(T\simeq S_\calx[-d]\).

For the second condition, we have a diagram
\[
\xymatrix{
\Mor(x,F^!y)\ar[d]\ar[r]^\sim&\Mor(Fx,y)\ar[d]\ar[r]^\sim&\Mor(y,Fx)^\vee[-d]\ar[d]\ar[dr]^[@]{\sim}\\
\Mor(x,F^!FF^*y)\ar[r]^\sim&\Mor(Fx,FF^*y)\ar[r]^\sim&\Mor(FF^*y,Fx)^\vee[-d]\ar[r]&\Mor(F^*y,x)^\vee[-d]\ar[d]^[@]{\sim}\\
&&&\Mor(x,S_{\calx}[-d]F^*y)
}
\]
Here the squares commute by naturality of adjunction and functoriality of the Serre functor, respectively, and the triangle commutes because \(F^*\to F^*FF^*\to F^*\) is the identity.

Note that the map \(\Mor(x,F^!FF^*y)\to\Mor(x,S_{\calx}[-d]F^*y)\) of the bottom row is the one obtained from the composition \(F^!F\to T\to S_\calx[-d]\) described above.  Thus taking the bottom path is precisely the composition \(F^!\to F^!FF^*\to TF^*\) we want.  The top path is an isomorphism, and in particular, is exactly the map \(\kappa\).  Thus our result is proven.
\end{proof}

We now turn to the converse:

\begin{thm}\label{thm:spherical-implies-relcy}
Let \(\calx\) and \(\caly\) be smooth and proper categories and suppose \(\caly\) has a weak right Calabi-Yau structure \(\phi:\HH_*(\caly)\to k[-d]\).  Let \(F:\calx\to\caly\) be a compatible spherical functor.  Then \(F\) has a \emph{weak} relative Calabi-Yau structure.
\end{thm}
\begin{proof}
We have isomorphisms \(\HH_*(\calx)^\vee\simeq \Mor_{\calx^e}(\calx,\calx^{\vee})\simeq \Mor_{\Fun_\calx}(\id_\calx, S_\calx)\) (\cite{BD} 2.8, 2.12).  Under this identification the pullback \(F^*:\HH_*(\caly)^\vee\to \HH_*(\calx)^\vee\) sends \(\alpha\in\Mor(\caly, \caly^{\vee})\) to the composition
\[
\calx\to (F^e)^*\caly_{\Delta}\xrightarrow{(F^e)^*\alpha} (F^e)^*\caly^{\vee}\simeq ((F^e)^*\caly_{\Delta})^*\to \calx^{\vee}
\]
where \(F^e: \calx^{e} \rightarrow \calx^{e}\) (see Notation \ref{notation:modules}) is the functor induced by \(F\). Identifying an endofunctor \(T:\calx\to\calx\) with the bimodule \(M_T(x,y)=\Mor_\calx(x,Ty)\), this becomes the composition
\[
\id_\calx\to F^!F\to F^!S_\caly F\simeq S_\calx F^*F\to S_\calx.
\]

In our case, we have \(S_{\caly}\simeq \id_{\caly}[d]\) via \(\phi\), under which identification \(\alpha=\phi=\id_{\id_{\caly}}\) in degree \(-d\).  Furthermore the structure map \(F^!F\to\cofib(\id_\calx\to F^!F)\simeq S_\calx\) is given by
\[
F^!F\simeq S_\calx F^*F\to S_\calx.
\]
(\cite{Ku15}, proof of Theorem 2.13).  Thus the pullback \(F^*\phi\) is given by
\[
\id_\calx\to F^!F\to\cofib(\id_\calx\to F^!F),
\]
which has a canonical homotopy to \(0\).

For nondegeneracy, we have a diagram
\[
\xymatrix{
\Mor_\calx(x,y)\ar[r]\ar@{=}[d]&\Mor_\calx(x,F^!Fy)\ar[r]\ar[d]^(0.45)[@]{\sim}&\Mor_\calx(x,S_{\calx}[-d]y)\ar[d]^(0.45)[@]{\sim}\\
\Mor_\calx(x,y)\ar[r]&\Mor_\caly(Fx,Fy)\ar[r]&\Mor_\calx(y,x)^\vee[-d].
}
\]
Here the first square clearly commutes; it remains to show commutativity of the second.  Expanding this square a little, we have:
\[
\xymatrix{
\Mor(x,F^!Fy)\ar[r]^(0.4)[@]{\sim}\ar[dd]^(0.45)[@]{\sim}&\Mor(x,S_{\calx}[-d]F^*Fy)\ar[r]\ar[d]^(0.45)[@]{\sim}&\Mor(x,S_{\calx}[-d]y)\ar[d]^(0.45)[@]{\sim}\\
&\Mor(F^*Fy,x)^\vee[-d]\ar[r]\ar[d]^(0.45)[@]{\sim}&\Mor(y,x)^\vee[-d]\ar@{=}[d]\\
\Mor(Fx,Fy)\ar[r]^(0.45)[@]{\sim}&\Mor(Fy,Fx)^\vee[-d]\ar[r]&\Mor(y,x)^\vee[-d].
}
\]
For the rectangle on the left, note that the top map is induced by \(\kappa\) by the compatibility condition, and going down, right, and back up is exactly the definition of \(\kappa\).  The two squares on the right are clear, so we have our commutativity and thus our result.
\end{proof}

As mentioned in \cite{KKP}, the information of a (compatible) spherical functor is roughly analogous to the naive definition of a Calabi-Yau category \(S_{\caly}\simeq \id_{\caly}[d]\) and likely needs to be supplemented with higher homotopical data.

\begin{question}
How can we naturally describe the higher-homotopical data on a compatible spherical functor that is needed to promote it to a relative Calabi-Yau structure? Can every compatible spherical functor be promoted in this way?
\end{question}

\section{Calabi-Yau structures on Perverse Schobers}\label{sec:cy-psc}

The main purpose of this section is to prove Theorem \ref{thm:main-gluing-theorem}, our main gluing theorem for Calabi-Yau structures, and to discuss its ramifications. The section is organized as follows. \S\ref{subsec:main-gluing-theorem} is entirely devoted to the statement and proof of Theorem \ref{thm:main-gluing-theorem}. In \S\ref{subsec:surgery}, the we discuss the implications of this theorem for the study of Calabi-Yau structures on Fukaya categories with coefficients in a perverse Schober. We do not develop a general theory of CY-structures on perverse Schobers here; that will appear elsewhere. Instead, we focus on some examples to illustrate the geometric content of the gluing theorem.  By reformulating the gluing construction in terms of spherical functors using Theorem \ref{intro:thm-relcy-iff-spherical}, we have attempted to bring out the relationship of the result of \S\ref{subsec:main-gluing-theorem} to the monodromy of Lefschetz fibrations. Furthermore, we study the effect of modifying the spherical functors defining a Schober (``categorical surgery''), and show how to obtain Kronecker quivers by performing categorical surgery on the LG-mirror of \(\mathbb{CP}^2\) (Example \ref{example:Kronecker}).

\subsection{Calabi-Yau spans and the main gluing theorem}\label{subsec:main-gluing-theorem}

In order to formulate the main gluing theorem for Calabi-Yau structures, we first introduce the categorical analogue of a cobordism between oriented manifolds, and some relevant notation.

\begin{notation}\label{notation:reverse-orientation}
Let \(\caly\) be a \(k\)-linear stable \(\infty\)-category equipped with a strong (resp. weak) \(d\)-Calabi-Yau structure (Definition \ref{def:cy}) \(\tilde{\phi_{\caly}}: \HC_*(\caly) \rightarrow k[-d]\) (resp.\ \(\phi_{\caly}: \HH_*(\caly) \rightarrow k[-d]\)). When clear from the context, we will omit \(\phi_{\caly}\) from the notation, and simply write \(\caly\) for the Calabi-Yau category \((\caly, \tilde{\phi_{\caly}})\) (resp. the weak Calabi-Yau category \((\caly, \phi_{\caly})\)). We will write \(\overline{\caly}\) to denote the Calabi-Yau category \((\caly, - \tilde{\phi_{\caly}})\) (resp.\ the weak Calabi-Yau category \((\caly, -\phi_{\caly})\)). Here \(-\phi_{\caly}\) is determined up to contractible ambiguity, and therefore, so is \(\overline{\caly}\).
\end{notation}

\begin{defn}\label{def:cobordism}
Recall the definition of a relative Calabi-Yau structure (Definition \ref{def:relcy}), and the notation introduced in \ref{notation:reverse-orientation}. An \emph{oriented cobordism} from a \(d\)-Calabi-Yau category \(\calx\) to a  \(d\)-Calabi-Yau category \(\caly\) is a functor \(\calz \rightarrow \overline{\calx} \times \caly\) equipped with a relative Calabi-Yau structure. \emph{Weak} oriented cobordisms are defined similarly, by simply replacing all the Calabi-Yau structures by weak Calabi-Yau structures. 
\end{defn}

The following theorem, which should be compared to Theorem 6.2 of \cite{BD} and Theorem 2.9 of \cite{PTVV}, states the oriented cobordisms can be glued together in a natural way. 

\begin{thm}\label{thm:main-gluing-theorem}
Let \(\calx, \caly\) and \(\calz\) be \(k\)-linear stable \(\infty\)-categories equipped with right \(d\)-Calabi-Yau structures \(\phi_{\calx}, \phi_{\caly}\) and \(\phi_{\calz}\), respectively. Let \(\calu \rightarrow \overline{\calx} \times \caly\) and \(\calv \rightarrow \overline{\caly} \times \calz\) be oriented cobordisms (resp. weak oriented cobordisms) in the sense of Definition \ref{def:cobordism}. Then the natural functor \(\calw := \calu \times_{\caly} \calv \rightarrow \overline{\calx} \times \calz\) is equipped in a canonical way with the structure of an oriented cobordism (resp. weak oriented cobordism). 
\label{thm:spancompose}
\end{thm}

The proof of this theorem is given later in the section, immediately after Construction \ref{constr:glued-isotropy-data}, which gives a construction of isotropy data on \(\calw \to \calx \times \calz\). We begin first with some elementary lemmas that will be needed in the proof. The reader may wish to jump ahead to Construction \ref{constr:glued-isotropy-data}, referring back to these lemmas when necessary.

\begin{lem}\label{lem:cat-to-map-pullback}
Let
\[
\xymatrix{
\calu \ar[r]^{H} \ar[d]_{I} & \caly \ar[d]^{G} \\
\calx \ar[r]_{F} & \calz \\
}
\]
be a pullback square of \(\infty\)-categories. Then for each \(u, v\) in \(\calu\), there is a pullback square of spaces
\[
\xymatrix{
\Map(u,v) \ar[r] \ar[d] & \Map(Hu, Hv) \ar[d] \\
\Map(Iu, Iv) \ar[r] & \Map(FIu, GHv) \\
}
\]
\end{lem}

\begin{proof}
First, note that the commutativity of the first square is given by an equivalence \(G \circ H \simeq F \circ I\), which in turn gives equivalences \(\Map(FIu, FIv) \simeq \Map(FIu, GHv) \simeq \Map(GHu, GHv)\). The lower right map in the second square is given by the composite of the map \(\Map(Iu, Iv) \rightarrow \Map(FIu, FIv)\) with the equivalence \(\Map(FIu, FIv) \rightarrow \Map(FIu, GHv)\), and a similar construction gives the right vertical map. 

Let \(\Delta^n\) denote the \(n\)-simplex, thought of as an \(\infty\)-category in the quasicategory model. Let \(\calc\) be an \(\infty\)-category, and let \(c,d\) be objects in \(\calc\). Then there is a natural pullback square of \(\infty\)-categories
\[
\xymatrix{
\Map_{\calc}(c,d) \ar[r] \ar[d] & \Fun(\Delta^1, \calc) \ar[d]\\
\Delta^{0} \ar[r]^{(c,d)} & \calc \times \calc \\
}
\]
where the right vertical map is obtained by applying \(\Fun(-, \calc)\) to the natural map \(\Delta^{0} \coprod \Delta^{0} \rightarrow \Delta^1\) of simplicial sets. 

Applying \(\Fun(\Delta^1, -)\) and \(\Fun(\Delta^0 \coprod \Delta^0, -)\) to our original pullback square of categories, we obtain two pullback squares of \(\infty\)-categories

\[
\xymatrix{
\Fun(\Delta^1, \calu) \ar[r] \ar[d]  & \Fun(\Delta^1, \caly) \ar[d] \\
\Fun(\Delta^1, \calx) \ar[r] & \Fun(\Delta^1, \calz) \\
}
\]

and 
\[
\xymatrix{
\calu \times \calu \ar[r] \ar[d] & \caly \times \caly \ar[d] \\
\calx \times \calx \ar[r] & \calz \times \calz \\
}
\]
respectively. Furthermore, pulling back along the map  \(\Delta^{0} \coprod \Delta^{0} \rightarrow \Delta^1\) gives a map from the first square to the second. By the previous paragraph, the homotopy fiber of this map is equivalent to the square of mapping spaces in the statement of the lemma. Since limits commute with limits, this square is a pullback square. 
\end{proof}

\begin{defn}\label{def:bdd-below}
A stable \(\infty\)-category \(\calc\) is locally bounded below if for each \(c, d\) in \(\calc\) there exists \(n \in \mathbb{Z}\) such that \(\Mor(c,d)\) is n-connective, i.e., such that \(\pi_k \Mor(c,d) = 0\) for \(k < n\). 
\end{defn}

\begin{lem}\label{lem:cat-to-mor-pullback}
Let
\[
\xymatrix{
\calu \ar[r]^{H} \ar[d]_{I} & \caly \ar[d]^{G} \\
\calx \ar[r]_{F} & \calz \\
}
\]
be a pullback square in the \(\infty\)-category of stable \(\infty\)-categories and exact functors. Suppose each of the categories is locally bounded below (Definition \ref{def:bdd-below}). Then for each \(u, v\) in \(\calu\), there is a pullback square of spectra
\[
\xymatrix{
\Mor(u,v) \ar[r] \ar[d] & \Mor(Hu, Hv) \ar[d] \\
\Mor(Iu, Iv) \ar[r] & \Mor(FIu, GHv) \\
}
\]
\end{lem}

\begin{proof}
The idea is to reduce the statement to Lemma \ref{lem:cat-to-map-pullback} using the connectivity hypothesis. Since the categories are locally bounded below, there exists \(n \in \mathbb{Z}\) such that the \(n\)-fold suspensions of the spectra appearing in the statement of the lemma are all \(1\)-connective. Since \(\Sigma^n \Mor(-,-) \simeq \Mor(-, \Sigma^n - )\), by replacing \(v\) with \(\Sigma^n v = v[n]\) we may assume that the morphism spectra appearing in the statement of the lemma are all \(1\)-connective. The commutativity of this diagram of spectra is clear from functoriality.

Since the category of spectra admits pullbacks, there is a pullback square

\begin{equation}
\label{diag:spectral-pullback} \tag{\( \natural\)}
\xymatrix{
K \ar[r] \ar[d] & \Mor(Hu, Hv) \ar[d] \\
\Mor(Iu, Iv) \ar[r] & \Mor(FIu, GHv) \\
}
\end{equation}

Applying Lemma \ref{lemma:square-equals-triangle} to the square \eqref{diag:spectral-pullback}, and using the long exact sequence of on homotopy groups, we deduce that \(K\) is \(0\)-connective. By the universal property of pullbacks, there is a natural morphism \(\Mor(u,v) \rightarrow K\) of connective spectra. Our goal is to show that this map is an equivalence. Since \(\Omega^{\infty}\) is conservative when restricted to connective spectra, it suffices to show that \(\Omega^{\infty} \Mor(u,v) \rightarrow \Omega^{\infty} K\) is an equivalence.

The functor \(\Omega^{\infty}\) from spectra to spaces is a right adjoint, and therefore preserves all limits. Applying \(\Omega^{\infty}\) to the square \eqref{diag:spectral-pullback}, we obtain a pullback square in the \(\infty\)-category of spaces

\[
\xymatrix{
\Omega^{\infty} K \ar[r] \ar[d] & \Map(Hu, Hv) \ar[d] \\
\Map(Iu, Iv) \ar[r] & \Map(FIu, GHv) \\
}
\]

By Lemma \ref{lem:cat-to-map-pullback}, \(\Map(u,v) \simeq \Omega^{\infty} \Mor(u,v)\) is also characterized as a pullback of the same maps. By the universal property of pullbacks, it follows that the map  \(\Omega^{\infty} \Mor(u,v) \rightarrow \Omega^{\infty} K\) is an equivalence, completing the proof.
\end{proof}

It is well known that the data of a pushout-pullback square in an abelian category is equivalent to the data of an exact sequence. The following lemma is a homotopical analogue of this fact, with abelian categories being replaced by stable \(\infty\)-categories, and exact sequences by fiber sequences. 

\begin{lem}\label{lemma:square-equals-triangle}
Let \(A, B, C\) and \(D\) be objects in a stable \(\infty\)-category \(\calc\), and let \(B \oplus C\) be a biproduct (product and coproduct) of \(B\) and \(C\), which exists since \(\calc\) is stable.  Let \(f \in \Map(B,D), g \in \Map(C,D), h \in \Map(A,B), i \in \Map(A,C), k \in \Map(B \oplus C, D)\) and \(l \in \Map(A, B \oplus C)\).

Suppose \(k\) maps to the connected component of \((f,-g)\)  under the equivalence \(\Map(B \oplus C, D) \simeq \Map(B,D) \times \Map(C,D)\) induced by the universal property of \(B \oplus C\), and \(l\) maps to the component of \((h,i)\) under the equivalence \(\Map(A, B \oplus C) \simeq \Map(A,B) \times \Map(A,C)\) .

\begin{enumerate}

\item  The following spaces are homotopy equivalent:

\begin{enumerate}
\item The space of paths from \(f \circ h\) to \(g \circ i\) in \(\Map(A,D)\), i.e., the space of 2-cells witnessing the commutativity of the square
\begin{equation}
\label{diagram:square}\tag{\(\flat\)}
\xymatrix{
A \ar[r]^{h} \ar[d]_{i} & B \ar[d]^{f} \\
C \ar[r]_{g} & D\\ 
}
\end{equation}
\item The space of paths from \(k \circ l\) to \(0\) in \(\Map(A,D)\), i.e., the space of 2-cells witnessing the commutativity of the diagram

\begin{equation}
\label{diagram:triangle}\tag{\(\flat \flat\)}
\xymatrix{
A\ar[r]^{l} \ar[d] & B \oplus C \ar[d]^{k} \\
0 \ar[r] & D\\
}
\end{equation}
\end{enumerate}

\item A commutative square of the form \eqref{diagram:square} is bicartesian if and only if its image of the form \eqref{diagram:triangle} under the homotopy equivalence of (1) is bicartesian.

\end{enumerate}

\end{lem}

\begin{proof}
Since \(\calc\) is stable, it admits a natural spectral enrichment. In particular, the mapping spaces are grouplike infinite loop spaces; for each pair of objects \(E, F\) in \(\calc\), there is a mapping spectrum \(\Mor(E,F)\), and \(\Map(E,F) \simeq \Omega^{\infty} \Mor(E,F)\). Thus, for any \(e \in \Map(E,F)\), there is a map \(\add_{e}: \Map(E,F) \rightarrow \Map(E,F)\), well defined up to homotopy, such that \([\add_e(e')] = [e] + [e']\) in \(\pi_0\Map(E,F)\); and furthermore, \(\add_{e}\) is a homotopy equivalence with a homotopy inverse given by \(\add_{-e}\), where \(-e \in \Map(E,F)\) is an element such that \(\add_{e}(-e)\) is a zero map. 

 Note that the statement of (1) is independent of the choice of composites \(f \circ h\) and \(g \circ i\), since it depends only on the connected components \([f \circ h]\) and \([g \circ i]\), and composition is well defined up to homotopy. Now we turn to the proof of (1). Let \(-g\) be an additive inverse for \(g\) and consider the map \(\add_{(-g) \circ i} : \Map(A,D) \rightarrow \Map(A,D)\). By the universal property of the biproduct \(B \oplus C\), we have \([\add_{(-g) \circ i}(f \circ h)] = [k \circ l]\) for any composite \(k \circ l\) of \(k\) and \(l\). One the other hand, using the fact that composition of morphisms extends to a map of spectra \(\Mor(A,C) \otimes \Mor(C,D) \rightarrow \Mor(A,D)\), we see that \([\add_{(-g) \circ i}(g \circ i)] = [\add_{-(g \circ i)}(g \circ i)] = [0]\). This, \(\add_{-g \circ i}: \Map(A,D) \rightarrow \Map(A,D)\) is a homotopy equivalence that carries the connected component of \(g \circ i\) to the connected component of zero maps, and carries the connected component of \(f \circ h\) to the connected component of \(k \circ l\). Passing to path spaces, this proves (1). 

Part (2) of the lemma follows immediately from part (1), and the universal property of bicartesian squares.

\end{proof}

\begin{lem}\label{lem:weak-isotropy-for-cobordisms}
Let \(\calx, \caly\) and \(\calu\) be locally proper \(k\)-linear stable \(\infty\)-categories, and let \(F = (F_{\calx}, F_{\caly}) : \calu \rightarrow \calx \times \caly\) be a functor. Let \(d \in \mathbb{N}\) and let \(\phi_{\calx}: \HH_*(\calx) \rightarrow k[-d]\) and \(\phi_{\caly}: \HH_*(\caly) \rightarrow k[-d]\) be weak right Calabi-Yau structures (Definition \ref{def:cy}). Then there is a natural homotopy equivalence between the following spaces:

\begin{enumerate}
\item  The space of weak isotropy structures (see Definition \ref{def:isotropy}) on \(F\) with respect to the weak right Calabi-Yau structure \((- \phi_{\calx}, \phi_{\caly})\) on \(\calx \times \caly\).

\item The space of paths from \(\phi_{\calx} \circ \HH_*(F_{\calx})\) to \(\phi_{\caly} \circ \HH_*(F_{\caly})\) in \(\Map_{\Perf(k)}(\HH_*(\calu) , k[-d])\), i.e., the space of 2-cells witnessing the commutativity of the following square
\[
\xymatrixcolsep{5pc} \xymatrix{
\HH_*(\calu) \ar[r]^{\HH_*(F_{\caly})} \ar[d]_{\HH_*(F_{\calx})} & \HH_*(\caly) \ar[d]^{\phi_{\caly}} \\
\HH_*(\calx) \ar[r]_{\phi_{\calx}} & k[-d]\\
}
\] 

\end{enumerate}

\end{lem}

\begin{proof}
The lemma follows immediately from part (1) of Lemma \ref{lemma:square-equals-triangle}.
\end{proof}

\begin{lem}\label{lem:isotropy-for-cobordisms}
Let \(\calx, \caly\) and \(\calu\) be locally proper \(k\)-linear stable \(\infty\)-categories, and let \(F = (F_{\calx}, F_{\caly}) : \calu \rightarrow \calx \times \caly\) be a functor. Let \(d \in \mathbb{N}\) and let \(\phi_{\calx}: \HH_*(\calx) \rightarrow k[-d]\) and \(\phi_{\caly}: \HH_*(\caly) \rightarrow k[-d]\) be weak right Calabi-Yau structures (Definition \ref{def:cy}). Then there is a natural homotopy equivalence between the following spaces:

\begin{enumerate}
\item  The space of isotropy structures (Definition \ref{def:isotropy}) on \(F\) with respect to the right Calabi-Yau structure \((- \tilde{\phi_{\calx}}, \tilde{\phi_{\caly}})\) on \(\calx \times \caly\).

\item The space of paths from \(\tilde{\phi_{\calx}} \circ \HC_*(F_{\calx})\) to \(\tilde{\phi_{\caly}} \circ \HC_*(F_{\caly})\) in \(\Map_{\Perf(k)}(\HC_*(\calu) , k[-d])\), i.e., the space of 2-cells witnessing the commutativity of the following square
\[
\xymatrixcolsep{5pc} \xymatrix{
\HC_*(\calu) \ar[r]^{\HC_*(F_{\caly})} \ar[d]_{\HC_*(F_{\calx})} & \HC_*(\caly) \ar[d]^{\phi_{\caly}} \\
\HC_*(\calx) \ar[r]_{\phi_{\calx}} & k[-d]\\
}
\] 

\end{enumerate}

\end{lem}

\begin{proof}
The lemma follows immediately from part (1) of Lemma \ref{lemma:square-equals-triangle}.
\end{proof}

\begin{construction}\label{constr:glued-isotropy-data}
All the categories appearing in the following construction are \(k\)-linear stable \(\infty\)-categories. Let \(d \in \mathbb{N}\), and suppose that the categories \(\calx, \caly\) and \(\calz\) are equipped with weak right \(d\)-Calabi-Yau structures \(\phi_{\calx}, \phi_{\caly}\) and \(\phi_{\calz}\). Let \(F = (F', F''): \calu \rightarrow \calx \times \caly\) and \(G = (G', G''): \calv \rightarrow \caly \times \calz\) be functors equipped with weak isotropy data with respect to the Calabi-Yau structures \((-\phi_{\calx}, \phi_{\caly})\) on \(\calx \times \caly\) and \((-\phi_{\caly}, \phi_{\calz})\) on \(\caly \times \calz\). Let \(\calw := \calu \times_{\caly} \calv\), and consider the diagram: 

\begin{equation}\label{diag:glued-isotropy}\tag{\(\sharp\)}
\xymatrix{
\HH_*(\calw) \ar[r] \ar[d] &  \HH_*(\calv) \ar[r] \ar[d] & \HH_*(\calz) \ar[d]^{\phi_{\calz}} \\
\HH_*(\calu) \ar[r] \ar[d]  & \HH_*(\caly) \ar[r]_{\phi_{\caly}} \ar[d]^{\phi_{\caly}} & k[-d] \ar[d]^{\id} \\
\HH_*(\calx) \ar[r]_{\phi_{\calx}} & k[-d] \ar[r]_{\id} & k[-d] \\
}
\end{equation}

All the unmarked arrows in this diagram are given by applying the functor \(\HH_*\) to the natural diagram of categories 

\begin{equation}\label{diag:glued-category}\tag{\P}
\xymatrix{
\calw \ar[r]^{H''} \ar[d]_{H'} & \calv \ar[r]^{G''} \ar[d]^{G'} & \calz \\
\calu \ar[r]_{F''} \ar[d]_{F'} & \caly  \\
\calx \\
}
\end{equation}

The upper left square in \eqref{diag:glued-isotropy} is equipped with the structure of a homotopy square by virtue of the functoriality of \(\HH_*\). The lower right square is equipped with the trivial homotopy commutative structure. By virtue of Lemma \ref{lem:weak-isotropy-for-cobordisms}, the lower left square and the upper right square are equipped with commutative structures induced by the isotropy data on the functors  \(F: \calu \rightarrow \calx \times \caly\) and \(G: \calv \rightarrow \caly \times \calz\), respectively. Since each of the four adjacent squares in the  diagram \eqref{diag:glued-isotropy}  is equipped with a homotopy commutative structure, it follows that the outer square is equipped with a homotopy commutative structure. Applying Lemma \ref{lem:weak-isotropy-for-cobordisms} again, we deduce that commutativity data for the outer square in  \eqref{diag:glued-isotropy} equips the natural functor \(\calw \rightarrow \calx \times \calz\) with a weak isotropy structure, with respect to the Calabi-Yau structure \((-\phi_{\calx}, \phi_{\calz})\).

If each of the functors  \(F: \calu \rightarrow \calx \times \caly\) and \(G: \calv \rightarrow \caly \times \calz\) is equipped with a \emph{strong} isotropy structure, then the argument of the paragraph above, applied to the analogue of the diagram  \eqref{diag:glued-isotropy} with \(\HH_*\) replaced by \(\HC_*\), constructs a strong isotropy structure on the functor \(\calw \rightarrow \calx \times \calz\).

\end{construction}

\begin{proof}[Proof of theorem \ref{thm:main-gluing-theorem}]\label{proof:main-gluing}
Let \(\calu \rightarrow \calx \times \caly\) and \(\calv \rightarrow \caly \times \calz\) be as in the statement of the theorem, and let \(\calw \simeq \calu \times_{\caly} \calv\). Construction \ref{constr:glued-isotropy-data} endows the natural functor (see diagram \eqref{diag:glued-category}) \( \calw \rightarrow \calx \times \calz\) with isotropy data. It remains to show that this isotropy data satisfies the nondegeneracy condition of Definition \ref{def:relcy}.  To this end, consider the following diagram

\begin{equation}
\label{diag:glued-nondeg} \tag{\(\natural \natural\)}
\xymatrix{
\Mor_{\calw}(u,v) \ar[r] \ar[d] & \Mor_{\calv}(H''u, H''v) \ar[r] \ar[d] & \Mor_{\calz}(G''H''u, G''H''v) \ar[d] \\
\Mor_{\calu}(H'u, H'v) \ar[r] \ar[d] & \Mor_{\caly}(F''H'u, G'H''v) \ar[r] \ar[d] & \Mor_{\calv}(H''v, H''u)^{\vee}[d] \ar[d] \\
\Mor_{\calx}(F'H'u, F'H'v) \ar[r] & \Mor_{\calu}(H'v, H'u)^{\vee}[d] \ar[r] & \Mor_{\calw}(v,u)^{\vee}[d] \\
}
\end{equation}

The upper left square of this diagram is the commutative square of Lemma \ref{lem:cat-to-mor-pullback} (see also Lemma \ref{lem:cat-to-map-pullback}). Recall that the central term in this diagram can be described in several ways: we have natural equivalences \(\Mor_{\caly}(G'H''u, G'H''v) \simeq \Mor_{\caly}(F''H'u, G'H''v) \simeq \Mor_{\caly}(F''H'u, F''H'v) \). The lower vertical map in the central column is the composite \((F'')^{\vee}[d] \circ \Xi(\phi_{\caly})\)  of the equivalence \(\Xi(\phi_{\caly}): \Mor_{\caly}(F''H'u, F''H'v) \rightarrow \Mor_{\caly}(F''H'v, F''H'u)^{\vee}[d]\) induced by the Calabi-Yau structure on \(\caly\) with the map \((F'')^{\vee}[d]: \Mor_{\caly}(F''H'v, F''H'u)^{\vee}[d] \rightarrow \Mor_{\calu}(H'v, H'u)^{\vee}[d]\). Similar remarks apply for the right hand map in the central row, the left hand map in the bottom row, and the upper map in the rightmost column. The reader is referred to Construction \ref{const:fiber-sequence}, and Diagram \eqref{diag:non-deg} therein, for a discussion of this construction.

Our categories are all locally proper, and therefore locally bounded below. Since \(\calw \simeq \calu \times_{\caly} \calv\), the hypotheses of Lemma \ref{lem:cat-to-mor-pullback} are satisfied, and we conclude that the square in the upper left hand corner is a bicartesian square of \(k\)-module spectra. Using the identification \( \Mor_{\caly}(F''H'u, F''H'v) \simeq \Mor_{\caly}(F''H'v, F''H'u)^{\vee}[d]\) induced by the Calabi-Yau structure on \(\caly\), the square in the lower right hand corner is identified with the \(k\)-linear dual, shifted by \(d\), of the square in the upper left corner. Since \((-)^{\vee}[d]\) is an exact functor, it follows that the lower right square is a pullback square as well.

Now let us consider the square in the lower left corner. Since the functor \(F: \calu \rightarrow \overline{\calx} \times \caly\) is equipped with a relative Calabi-Yau structure, the underlying isotropy structure gives rise to a \emph{commutative} diagram

\begin{equation}
\label{diag:triang-lower-left} \tag{\(\diamondsuit\)}
\xymatrix{
\Mor_{\calu}(H'u, H'v) \ar[r] \ar[d] & \Mor_{\calx}(F'H'u, F'H'v) \oplus \Mor_{\caly}(F''H'u, F''H'v) \ar[d]\\
0 \ar[r] & \Mor_{\calu}(H'v, H'u)^{\vee}[d]\\
}
\end{equation}
by Construction \ref{const:fiber-sequence}. By Lemma \ref{lemma:square-equals-triangle}, this implies that the lower left square in Diagram \eqref{diag:glued-nondeg} is homotopy commutative. Since the isotropy structure on \(F\) defines a Calabi-Yau structure, it is non-degenerate (Definition \ref{def:relcy}), which means that \eqref{diag:triang-lower-left} is in fact a \emph{pullback square}. Applying Lemma \ref{lemma:square-equals-triangle} once again, we conclude that the lower left square in \eqref{diag:glued-nondeg} is a pullback square. The same argument shows that the square in the upper right corner of Diagram  \eqref{diag:glued-nondeg} is a pullback square. 

Thus we have shown that all the four adjacent squares in  Diagram  \eqref{diag:glued-nondeg} are pullback-pushout squares of \(k\)-module spectra. It follows that the outer square of this diagram is also a pullback-pushout square.  Applying Lemma \ref{lemma:square-equals-triangle} to the outer square, we conclude that the corresponding diagram

\[
\xymatrix{
\Mor_{\calw}(H'u, H'v) \ar[r] \ar[d] & \Mor_{\calx}(F'H'u, F'H'v) \oplus \Mor_{\calz}(G''H''u, G''H''v) \ar[d]\\
0 \ar[r] & \Mor_{\calw}(v, u)^{\vee}[d]\\
}
\]

is a pullback square. An elementary, albeit somewhat tedious, diagram-chase (whose details we leave to the reader) shows that this square is in fact the square obtained by applying Construction \ref{const:fiber-sequence} to the ``glued'' isotropy structure on \(\calw \rightarrow \calx \times \calz\) obtained by Construction \ref{constr:glued-isotropy-data}. Thus we have  proven that the glued isotropy structure on  \(\calw \rightarrow \calx \times \calz\) is non-degenerate, which is what we set out to do.
\end{proof}

\subsection{CY-structures and surgery on Schobers}\label{subsec:surgery}

Theorem \ref{thm:main-gluing-theorem} has the following familiar analogue in topology.  Let \(M_1,M_2\) be oriented manifolds with boundary, and partition the boundary of each one as \(\del M_1= N_1\coprod N'_{12}\) and \(\del M_2=N_{12}\coprod N_2\).  Giving each boundary component the induced orientation, let us further suppose that \(N'_{12}\simeq N_{12}\) via an orientation-reversing homeomorphism; we write \(\del M_1=N_1\coprod \overline{N_{12}}\).  Then we can glue \(M_1\) and \(M_2\) along \(N_{12}\) to get a new manifold \(M=M_1\coprod_{N_{12}}M_2\) with boundary \(N_1\coprod N_2\) (see Figure \ref{fig:maniglue}).

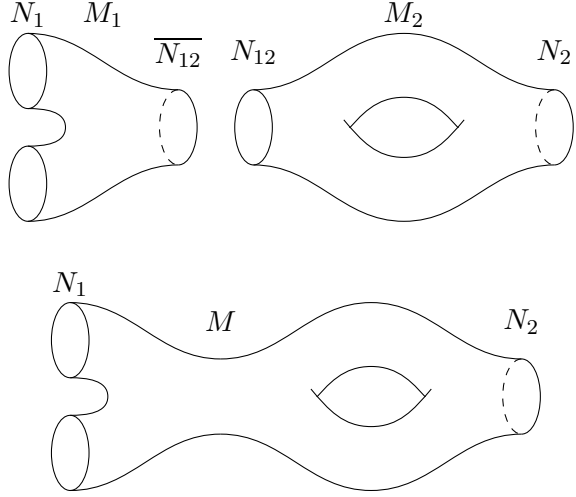
\begin{figure}
\centering
\begin{tikzpicture}
\draw (0,1.25) ellipse (.25 and .5);
\draw (0,2.75) ellipse (.25 and .5);
\node at (0,3.5) {\(N_1\)};
\draw (0,1.75) to [out=0, in=270] (.5,2) to [out=90, in=0] (0,2.25);
\draw (2,1.5) arc (-90:90:.25 and .5);
\draw[dashed] (2,2.5) arc (90:270:.25 and .5);
\node at (2,3) {\(\overline{N_{12}}\)};
\draw (0,0.75) to [out=0, in=180] (2,1.5);
\draw (0,3.25) to [out=0, in=180] (2,2.5);
\draw (3,2) ellipse (.25 and .5);
\node at (3,3) {\(N_{12}\)};
\node at (1,3.5) {\(M_1\)};
\draw (3,1.5) to [out=0,in=180] (5,0.75) to [out=0,in=180] (7,1.5);
\draw (3,2.5) to [out=0,in=180] (5,3.25) to [out=0,in=180] (7,2.5);
\draw (7,1.5) arc (-90:90:.25 and .5);
\node at (7,3) {\(N_2\)};
\draw[dashed] (7,2.5) arc (90:270:.25 and .5);
\draw (4.2, 2.1) to [out=-50,in=180] (5,1.6) to [out=0,in=230] (5.8,2.1);
\begin{scope}
\clip (4,2) rectangle (6,3);
\draw (4.2,1.9) to [out=50, in=180] (5,2.4) to [out=0, in=130] (5.8,1.9);
\end{scope}
\node at (5,3.5) {\(M_2\)};
\end{tikzpicture}
\vspace{.5cm}

\begin{tikzpicture}
\draw (0,1.25) ellipse (.25 and .5);
\draw (0,2.75) ellipse (.25 and .5);
\node at (0,3.5) {\(N_1\)};
\draw (0,1.75) to [out=0, in=270] (.5,2) to [out=90, in=0] (0,2.25);
\node at (2,3) {\(M\)};
\draw (0,0.75) to [out=0, in=180] (2,1.5);
\draw (0,3.25) to [out=0, in=180] (2,2.5);
\draw (2,1.5) to [out=0,in=180] (4,0.75) to [out=0,in=180] (6,1.5);
\draw (2,2.5) to [out=0,in=180] (4,3.25) to [out=0,in=180] (6,2.5);
\draw (6,1.5) arc (-90:90:.25 and .5);
\node at (6,3) {\(N_2\)};
\draw[dashed] (6,2.5) arc (90:270:.25 and .5);
\draw (3.2, 2.1) to [out=-50,in=180] (4,1.6) to [out=0,in=230] (4.8,2.1);
\begin{scope}
\clip (3,2) rectangle (5,3);
\draw (3.2,1.9) to [out=50, in=180] (4,2.4) to [out=0, in=130] (4.8,1.9);
\end{scope}
\end{tikzpicture}
\caption{Gluing two manifolds along a boundary component}
\label{fig:maniglue}
\end{figure}

We can relate this to the categorical case as follows.  Letting \(\Loc(M_i)\) be the category of dg-local systems on \(M_i\), we expect a relative Calabi-Yau structure on the pullback \(\Loc(M_i)\to \Loc(\del M_i)\).  Furthermore, we have isomorphisms 
\(\Loc(\del M_1)\simeq\Loc(N_1\coprod \overline{N_{12}})\simeq \Loc(N_1)\times \Loc(N_{12})\) and similarly \(\Loc(\del M2)\simeq\Loc(N_{12})\times \Loc(N_2)\).  We also have \(\Loc(M)\simeq \Loc(M_1)\times_{\Loc(N_{12})}\Loc(M_2)\).  Then the theorem gives us that \(\Loc(M)\to \Loc(N_1)\times\Loc(N_2)\) is compatible spherical, which corresponds to the fact that \(M\) has boundary \(N_1\coprod N_2\).

Now we consider the categorical generalization of the following situation from symplectic geometry.  Let \(w:X\to \CC\) be a Landau-Ginzburg model with smooth fiber \(Y\) and compact critical locus.  Let \(U_1, U_2\) be bounded open sets of \(\CC\) such that \(U_1\cup U_2\) contains all critical values of \(w\) and \(U_1\cap U_2=\emptyset\).  Further assume (for \(i=1,2\)) that the maps \(w|_{U_i}:X|_{U_i}\to U_i\) can be extended to \(w_i:X_i\to \CC\) for some space \(X_i\) and map \(w_i\) such that \(w_i\) is a fibre bundle above an open set containing \(\CC-U_i\).  For example, if each \(U_i\) is a convex region this latter condition is certainly possible, and in particular if \(w\) is a Picard-Lefschetz fibration we can partition the (isolated) critical values between two such regions.

In this situation we expect the structure of \(X\) to be related to the structure of the \(U_i\), and on a categorical level we expect to obtain \(\FS(X,w)\) and \(\cap:\FS(X,w)\rightleftarrows \Fuk(Y):\cup\) from \(\FS(X_i,w_i)\) and \(\cap_i:\FS(X_i,w_i)\rightleftarrows \Fuk(Y):\cup_i\).

More generally, suppose \(\caly\) has a weak right \(d\)-Calabi-Yau structure and \(F_i:\calx_i\to \caly\) is a spherical functor for \(i=1,2\).  Viewing the case of Fukaya-Seidel categories as the prototypical spherical functor, we seek to glue \(\calx_1\) and \(\calx_2\) in a similar way to obtain a composite \(\calx\) and spherical functor \(F:\calx\to\caly\).

We will represent \(\FS(X_i,w_i)\) or its objects by the diagram

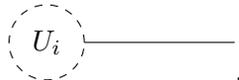
\begin{center}
\begin{tikzpicture}
\draw[dashed] (0,0) circle [radius=0.5];
\draw (0.5,0) -- (2.5,0);
\node at (0,0) {\(U_i\)};
\end{tikzpicture}.
\end{center}

Here we think of the exiting line as having a fiber \(\Fuk(Y)\), and the application of \(\cap_i\) on an object \(L\in \FS(X_i,w_i)\) as looking at the intersection of a \(L\) with a fiber above the line.  In general, for each outgoing path, we can construct a functor in this way.  Analogously we will represent \(\calx_i\) by

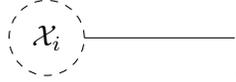
\begin{center}
\begin{tikzpicture}
\draw[dashed] (0,0) circle [radius=0.5];
\draw (0.5,0) -- (2.5,0);
\node at (0,0) {\(\calx_i\)};
\end{tikzpicture}.
\end{center}

In the symplectic case, \(\cap_i\) has a right (left) adjoint \(\cup_i\) (\(\cup_i^L\)) given by parallel transport of a Lagrangian on a loop from \(+\infty\) to \(+\infty\) clockwise (counter-clockwise) around \(U_i\).  We will represent these by the pictures

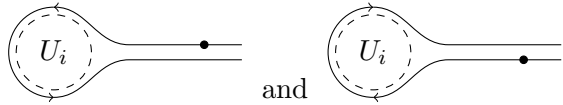
\begin{center}
\begin{tikzpicture}
\draw[dashed] (0,0) circle [radius=0.5];
\draw[->] (2.5,.1) to (1,.1) to [out=180,in=0] (0,.6);
\draw (0,-.6) to [out=0,in=180] (1,-.1) to (2.5,-.1);
\draw[->] (0,.6) arc (90:270:.6);
\draw [fill] (2,.1) circle [radius=0.05];
\node at (0,0) {\(U_i\)};
\end{tikzpicture}
\space and \space
\begin{tikzpicture}
\draw[dashed] (0,0) circle [radius=0.5];
\draw (2.5,.1) to (1,.1) to [out=180,in=0] (0,.6);
\draw[<-] (0,-.6) to [out=0,in=180] (1,-.1) to (2.5,-.1);
\draw[<-] (0,.6) arc (90:270:.6);
\draw [fill] (2,-.1) circle [radius=0.05];
\node at (0,0) {\(U_i\)};
\end{tikzpicture}
\end{center}

respectively, and we will draw similar pictures for the adjoints \(F^!\), \(F^*:\caly_i\to\calx_i\).

Using Theorem \ref{thm:spancompose}, the functor \(\calx_1\times_{\caly}\calx_2\to 0\) is compatible spherical.  In the Landau-Ginzburg situation above, this corresponds to the weak Calabi-Yau structure on \(\Fuk(X)\), the Fukaya category of compact Lagrangians, which is not exactly what we want (Figure \ref{fig:matching}).  Instead, we model \(\FS(X,w)\) with the following construction.  Let \(\FS_\pm(X_2,w_2)\) be the two-sided Fukaya-Seidel category of \((X_2,w_2)\) consisting of Lagrangians in \(X_2\) which, outside of some compact set \(K\subset \CC\), consist of fibrewise Lagrangians parallel transported along rays to \(+\infty\) or \(-\infty\).  In this scenario, we have two restriction functors \(\cap_-, \cap_+:\FS_\pm(X_2,w_2)\to \Fuk(Y)\), which are the fibers at \(-\infty\) and \(+\infty\).  We draw this as in Figure \ref{fig:twosidedFS}, with \(\cap_-\) and \(\cap_+\) corresponding to the left and right path, respectively.  Then we can construct the fiber product
\[
\FS(X,U_1,U_2,w)=\FS(X_1,w_1)\times_{\cap_1,\Fuk(Y),\cap_-}\FS_\pm(X_2,w_2),
\]
which we expect to coincide with \(\FS(X,w)\).  Furthermore, we have a map \(\cap':\FS(X,U_1,U_2,w)\to \FS_\pm(X_2,w_2)\xrightarrow{\cap_+}\Fuk(Y)\) which we expect to agree with \(\cap:\FS(X,w)\to \Fuk(Y)\).  See Figure \ref{fig:FSglue}; the outgoing path is \(\cap'\).

\begin{figure}
\centering
\begin{tikzpicture}
\draw[dashed] (0,0) circle [radius=0.5];
\draw (0.5,0) -- (2.5,0);
\draw[dashed] (3,0) circle [radius=0.5];
\node at (0,0) {\(\calx_1\)};
\node at (3,0) {\(\calx_2\)};
\end{tikzpicture}
\caption{The category \(\calx_1\times_{\caly}\calx_2\)}
\label{fig:matching}
\end{figure}

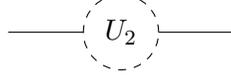
\begin{figure}
\centering
\begin{tikzpicture}
\draw[dashed] (0,0) circle [radius=0.5];
\draw (-1.5,0) -- (-0.5,0);
\draw (0.5,0) -- (1.5,0);
\node at (0,0) {\(U_2\)};
\end{tikzpicture}
\caption{The category \(\FS_{\pm}(X_2,w_2)\)}
\label{fig:twosidedFS}
\end{figure}

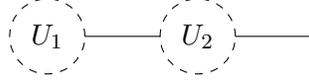
\begin{figure}
\centering
\begin{tikzpicture}
\draw[dashed] (0,0) circle [radius=0.5];
\node at (0,0) {\(U_1\)};
\draw (0.5,0) -- (1.5,0);
\draw[dashed] (2,0) circle [radius=0.5];
\node at (2,0) {\(U_2\)};
\draw (2.5,0) -- (3.5,0);
\end{tikzpicture}
\caption{The category \(\FS(X_1,w_1)\times_{\cap_1,Fuk(F),\cap_-}\FS_\pm(X_2,w_2)\)}
\label{fig:FSglue}
\end{figure}

\begin{rmk}
Whether we actually have an equivalence \(\FS(X,U_1,U_2,w)\simeq \FS(X,w)\) seems to depend on a Mayer-Vietoris type result for Fukaya-Seidel categories. This should be more straightforward to verify if \(\FS(X,w)\) is generated by thimbles.
\end{rmk}

In the case of general spherical functors, we wish to construct a category \(\calx_{2,\pm}\) and two functors \(F_{\pm}:\calx_{2,\pm}\to\caly\) as above.  To do this, we note that Figure \ref{fig:twosidedwant} is equivalent ``up to homotopy'' to Figure \ref{fig:twosidedsynth}, and the latter has an existing categorical interpretation, which we will take as our \(\calx_{2,\pm}\).
\begin{figure}
\centering
\begin{subfigure}{0.4\textwidth}
\centering
\begin{tikzpicture}
\draw[dashed] (0,0) circle [radius=0.5];
\draw (-1.5,0) -- (-0.5,0);
\draw (0.5,0) -- (1.5,0);
\node at (0,0) {\(\calx_2\)};
\end{tikzpicture}
\caption{The desired \(\calx_{2,\pm}\)}
\label{fig:twosidedwant}
\end{subfigure}
\begin{subfigure}{0.4\textwidth}
\centering
\begin{tikzpicture}
\draw[dashed] (0,1) circle [radius=0.5];
\draw (0,0.5) -- (0,0);
\draw (-1.5,0) -- (1.5,0);
\node at (0,1) {\(\calx_2\)};
\end{tikzpicture}
\caption{The constructed \(\calx_{2,\pm}\)}
\label{fig:twosidedsynth}
\end{subfigure}
\caption{}
\end{figure}

Consider \(\Rep(A_2)\otimes \caly\).  There are three natural functors \(P_1,P_2,P_3:\Rep(A_2)\to \Rep(A_1)\simeq D^b(k)\) sending a representation \(E_1\to E_2\) to \(E_1, E_2\), and \(\cofib(E_1\to E_2)\) respectively.  These induce functors \(P_i\otimes \id:\Rep(A_2)\otimes\caly\to \Rep(A_1)\otimes\caly\simeq\caly\).  We represent this as in Figure \ref{fig:A2}; the left, top, and right paths correspond to \(P_1,P_2,\) and \(P_3\) respectively.

\begin{figure}
\centering
\begin{tikzpicture}
\draw (-1,0) -- (1,0);
\draw (0,0) -- (0,1);
\end{tikzpicture}
\caption{The category \(\Rep(A_2)\otimes\caly\)}
\label{fig:A2}
\end{figure}

Then we set
\[
\calx_{2,\pm}=\calx_2\times_{F_2, \caly, P_2\otimes \id}\Rep(A_2)\otimes\caly
\]
and maps \(F_-,F_+\) by compositions
\begin{align*}
F_-&:\calx_{2,\pm}\to \Rep(A_2)\otimes\calx_2\xrightarrow{P_1\otimes \id} \caly\\
F_+&:\calx_{2,\pm}\to \Rep(A_2)\otimes\calx_2\xrightarrow{P_3\otimes \id} \caly.
\end{align*}
See Figure \ref{fig:twosidedsynth} again.  The left and right paths out are the functors \(F_-\) and \(F_+\), respectively.  Now we wish to show that \(F_\pm=(F_+,F_-):\calx_{2,\pm}\to\caly\times\caly\) is spherical.
\begin{lem}
The functor \(P=(P_1,P_2,P_3):\Rep(A_2)\to \Rep(A_1)^{\times3}\) is compatible spherical.
\end{lem}
\begin{proof}
For \(i=1,2,3\) define \(Q_i:\Rep(A_1)\to \Rep(A_2)\) by
\begin{align*}
Q_1x&=(x\to 0)\\
Q_2x&=(x\xrightarrow{\id_x} x)\\
Q_3x&=(0\to x).
\end{align*}
It is easy to check that \(P_i\dashv Q_i\) for all \(i\), \(Q_{i+1}\dashv P_i\) for \(i=1,2\), and \(Q_1[-1]\dashv P_3\).  Then we have \(P^!(x,y,z)=Q_1x\oplus Q_2y\oplus Q_3z\) and \(P^*(x,y,z)=Q_2x\oplus Q_3 y\oplus Q_1z[-1]\).  Furthermore, for \((x\xrightarrow{f}y)\in \Rep(A_2)\), we have
\[
P^!P(x\to y)=(x\oplus y\to y\oplus \cofib(f))
\]
and
\[
\cofib(\id\to P^!P)(x\to y)=(y\to\cofib(f))=S_{\Rep(A_2)}(x\to y).
\]
\end{proof}
\begin{lem}
Let \(\caly\) have a weak right \(d\)-Calabi-Yau structure and \(\calz\) a weak right \(d'\)-Calabi-Yau structure.  Let \(F:\calx\to\caly\) be a compatible spherical functor.  Then \(F\otimes \id_\calz:\calx\otimes\calz \to \caly\otimes\calz\) is compatible spherical.
\label{lem:tensorspher}
\end{lem}
\begin{proof}
First, we check that \(\caly\otimes\calz\) has a weak right \((d+d')\)-Calabi-Yau structure.  Using the isomorphisms \(\id_\caly[d]\simeq S_\caly\) and \(\id_\calz[d']\simeq S_\calz\) we have, for \(y\otimes z,y'\otimes z'\in\caly\otimes \calz\)
\begin{align*}
\Mor(y\otimes z,y'\otimes z'[d+d'])&\simeq\Mor(y,y'[d])\otimes\Mor(z,z'[d'])\\
&\simeq\Mor(y',y)^\vee\otimes\Mor(z',z)^\vee\\
&\simeq\Mor(y'\otimes z',y\otimes z)^\vee,
\end{align*}
yielding an isomorphism \(\id_{\caly\otimes\calz}[d+d']\simeq S_{\caly\otimes\calz}\).  This is equivalent to the desired weak right \((d+d')\)-Calabi-Yau structure.  Similarly it is easy to check that \(S_{\calx\otimes\calx}\simeq S_{\calx}\otimes \id_\calz[d']\).

Now, the left and right adjoints of \(F\otimes \id\) are \(F^*\otimes \id\) and \(F^!\otimes \id\) respectively.  We have \begin{align*}
\cofib(\id_\calx\otimes \id_\calz\to (F^!\otimes \id_\calz)\circ(F\otimes \id_\calz))&\simeq \cofib(\id_\calx\to F^!F)\otimes \id_\calz\\
&\simeq S_{\calx\otimes\calz}[-(d+d')].
\end{align*}
It is easy to check that under this identification the map
\[
F^!\otimes \id_\calz\to (F^!\otimes \id_\calz)\circ (F\otimes \id_\calz)\circ (F^*\otimes \id_\calz)\to S_{\calx\otimes\calz}[-(d+d')]\circ (F^*\otimes \id_\calz)
\]
coincides with \(\kappa_{F\otimes \id_{\calz}}\simeq \kappa_{F}\otimes \id_{\id_{\calz}}\).
\end{proof}
\begin{cor}
The functor \(F_{\pm}=(F_+,F_-):\calx_{2,\pm}\to \caly\times\caly\) is compatible spherical.
\end{cor}
\begin{proof}
By the previous two lemmas, the functor
\[
(P_1,P_2,P_3)\otimes \id_{\caly}:\Rep(A_2)\otimes \caly\to \caly^{\times3}
\]
is compatible spherical.  By assumption, \(F_2:\calx_2\to \caly\) is compatible spherical.  Applying Theorem \ref{thm:spancompose}, we see that \(F_{\pm}\) is compatible spherical.
\[\xymatrixcolsep{5pc}
\xymatrix{
\calx_{2,\pm}\ar[r]\ar[d]&\Rep(A_2)\otimes\caly\ar[r]^-{(P_1,P_3)\otimes \id_{\caly}}\ar[d]^{P_2\otimes \id_{\caly}}&\caly\times\caly\\
\calx_2\ar[r]\ar[d]&\caly\\
0
}
\]
\end{proof}

With this in hand, we can construct \(\calx_{\tot}=\calx_1\times_{F_1,\caly,F_-}\calx_{2,\pm}\).  Applying Theorem \ref{thm:spancompose} once more, we see that the composition
\[
F_{\tot}:\calx_{\tot}\to\calx_{2,\pm}\xrightarrow{F_+}\caly
\]
is compatible spherical, assuming it has both adjoints.  These adjoints can be shown to exist by general results, but we will construct them explicitly shortly.  Before this, we note that
\[
\calx_{\tot}=\calx_1\times_{F_1,\caly,P_1\otimes \id_\caly}\Rep(A_2)\otimes\caly \times_{P_1\otimes \id_\caly,\caly,F_2}\calx_2
\]
and can be equivalently described as follows:
\begin{itemize}
\item The objects of \(\calx_{\tot}\) are triples \((x,y,f)\) with \(x\in \calx_1, y\in \calx_2\), and \(f:F_1x\to F_2y\).
\item For \((x,y,f), (x',y',f')\), \(\Mor((x,y,f),(x',y',f'))\) is the space of triples \((g,h,\alpha)\), with \(g:x\to x'\), \(h:y\to y'\), and \(\alpha:f'\circ F_1g\sim F_2h\circ f\).
\end{itemize}
With this description, \(F_{\tot}\) sends \((x,y,f)\) to \(\cofib(f)\).  Pictures of \(\calx_{\tot}\) are shown in Figure \ref{fig:gluecat}; the left diagram is in line with previous ones, but the (equivalent) right diagram stresses the symmetry between \(\calx_1\) and \(\calx_2\) and we will generally prefer it.
\begin{figure}
\centering
\begin{tikzpicture}
\draw[dashed] (-1.5,0) circle [radius=0.5];
\node at (-1.5,0) {\(\calx_1\)};
\draw[dashed] (0,1) circle [radius=0.5];
\node at (0,1) {\(\calx_2\)};
\draw (0,0.5) -- (0,0);
\draw (-1,0) -- (1.5,0);
\end{tikzpicture}
\hspace{3cm}
\begin{tikzpicture}
\draw (-1,-.75) -- (0,0) -- (-1,.75); 
\draw[fill=white,dashed] (-1,-.75) circle [radius=0.5];
\node at (-1,-.75) {\(\calx_1\)};
\draw[fill=white,dashed] (-1,.75) circle [radius=0.5];
\node at (-1,.75) {\(\calx_2\)};
\draw (0,0) -- (1,0);
\end{tikzpicture}
\caption{Two renditions of the category \(\calx_{\tot}\)}
\label{fig:gluecat}
\end{figure}

As a bit of notation, for \(i=1,2\), we let \(T_i'=\fib(F_iF_i^!\to \id_{\caly})\); in the case of the Picard-Lefschetz fibration \(T_i'[1]\) is the counterclockwise monodromy on \(\Fuk(Y)\) around \(U_i\).  Note that \(T_i'\) is an autoequivalence, and there is a natural choice of inverse \((T_1'x)^{-1}=\cofib(\id_{\caly}\to F_iF_i^*)\); in particular, we have a morphism \(F_iF_i^*\to (T_1'x)^{-1}\) (\cite{AL}, Theorem 1.1).
\begin{lem}
The right adjoint \(F_{\tot}^!\) of \(F_{\tot}\) is given by \(F_{\tot}^!x=(F_1^!T_2'x,F_2^!x,g_x)\), where \(g_x=\xi\circ\eta_{1,T_2'x}\) is the morphism
\[
F_1F_1^!T_2'x\xrightarrow{\eta_{1,T_2'x}} T_2'x \xrightarrow{\xi} F_2F_2^!x,
\]
where the \(\eta_{1,T_2'x}\) is the counit of \(F_1\dashv F_1^!\) and \(\xi\) comes from the fiber sequence \(T_2'\to F_2F_2^!\to \id_{\caly}\).  Similarly, the left adjoint \(F_{\tot}^*\) is given by \(F_{\tot}^*x=(F_1^*x,F_2^*(T_1')^{-1}x, h_x)\), where \(h_x\) is the morphism
\[
F_1F_1^*x\to (T_1')^{-1}x\to F_2F_2^*(T_1')^{-1}x;
\]
here the first map is the previously mentioned \(F_iF_i^*\to (T_1'x)^{-1}\) and the second is the unit of \(F_2^*\dashv F_2\).
\label{lem:glueadj}
\end{lem}
\begin{proof}
We show the case of \(F_{\tot}^!\); the proof for \(F_{\tot}^*\) is similar.  Fix \(x\in\caly\) and \((y,z,f)\in\calx_{\tot}\).  We have \(\Mor_{\caly}(F_{\tot}(y,z,f),x)\simeq\Mor_{\caly}(\cofib(f),x)\).  The latter of these is equivalent to the space of \((h,\alpha)\) in the diagram
\[
\xymatrix{
F_1y\ar[r]^f\ar[rd]_0\drtwocell\omit{^<-2>\alpha}&F_2z\ar[d]^h\\
&x
}
\]
That is, \(h:F_2z\to x\) and \(\alpha:h\circ f\sim0\).  Using the adjunction \(\Mor(F_2z,x)\simeq\Mor(z,F_2^!x)\), the former is equivalent to the choice of \(\tilde{h}:z\to F_2^!x\).  Further, we have a natural homotopy
\[
\xymatrix{
&F_2z\ar[dl]_{F_2\tilde{h}}\ar[d]^h\dltwocell\omit{^<-2>}\\
F_2F_2^!x\ar[r]&x
}.
\]
We then have a diagram
\[
\xymatrix{
F_1y\ar[r]\ar[d]_{F_2\tilde{h}\circ f}\ar[dr]^(.7){0}&F_2z\ar[d]\ar[dl]|\hole\\
F_2F_2^!x\ar[r]_{\eta_{2,x}}&x
};
\]
filling in the upper left triangle with a \(2\)-morphism witnessing the composition \(F_2\tilde{h}\circ f\), we are left with a horn we can fill in with a \(3\)-simplex, and in particular we get a homotopy \(\beta:\eta_x\circ F_2h\circ f\sim 0\); further, it is clear that given \(h\), the choice of \(\alpha\) and \(\beta\) are equivalent.

Finally, the choice of such a \(\beta\) is equivalent to a map \(k:F_1y\to \fib(\eta_{2,x})=T_2'x\) plus a homotopy \(\gamma:\xi\circ k\sim F_2\tilde{h}\circ f\), where \(\xi:T_2'x\to F_2F_2'x\) is the natural map.  This \(k\) and \(\gamma\) are then equivalent to \(\tilde{k}:y\to F_1^!T_2'x\) and a homotopy \(\gamma:\xi\circ\eta_{1,T_2'x}\circ \tilde{k}\sim F_2\tilde{h}\circ f\)
\[
\xymatrix{
F_1F_1^!T_2'x\ar[d]_{\eta_{1,T_2'x}}&F_1y\ar[d]^{F_2\tilde{h}\circ f}\ar[l]_{\tilde{k}}\dtwocell\omit{^<6>\gamma}\\
T_2'x\ar[r]^{\xi}&F_2F_2'x\ar[r]&x
}
\]
Taking \(g_x=\xi\circ\eta_{1,T_2'x}\), this \(\gamma\) is the same as commutativity data for the square
\[
\xymatrix{
F_1y\ar[r]^f\ar[d]^{\tilde{k}}\drtwocell\omit&F_2z\ar[d]^{\tilde{h}}\\
F_1F_1^!T_2'x\ar[r]^{g_x}&F_2F_2^!x
}.
\]
Thus we have \(\Mor_{\caly}(F_{\tot}(y,z,f),x)\simeq\Mor_{\calx_{\tot}}((y,z,f),F_{\tot}^!x)\).
\end{proof}
\begin{cor}
The functor \(F_{\tot}:\calx_{\tot}\to \caly\) is compatible spherical.
\label{cor:totsph}
\end{cor}

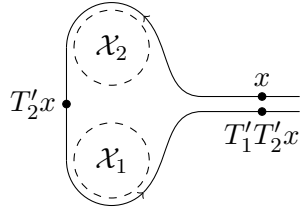
\begin{figure}
\centering
\begin{tikzpicture}
\draw [<-] (-1,.75) +(45:.6) to [out=-45,in=180] (.2,.1) to (1.5,.1);
\draw (-1.6, .75) arc (180:45:.6);
\draw (-1.6, .75) -- (-1.6, -.75);
\draw [->] (-1.6, -.75) arc (180:315:.6);
\draw (-1,-.75) +(-45:.6) to [out=45,in=180] (.2,-.1) to (1.5,-.1);
\draw[fill=white,dashed] (-1,-.75) circle [radius=0.5];
\node at (-1,-.75) {\(\calx_1\)};
\draw[fill=white,dashed] (-1,.75) circle [radius=0.5];
\node at (-1,.75) {\(\calx_2\)};
\draw[fill] (1,.1) circle [radius=0.05];
\draw[fill] (-1.6,0) circle [radius=0.05];
\draw[fill] (1,-.1) circle [radius=0.05];
\node[above] at (1,.1) {\(x\)};
\node[left] at (-1.6,0) {\(T_2'x\)};
\node[below] at (1,-.1) {\(T_1'T_2'x\)};
\end{tikzpicture}
\caption{The total right adjoint \(F_{\tot}^!\).}
\label{fig:radj}
\end{figure}

\begin{rmk}
We may represent the situation of Lemma \ref{lem:glueadj} with Figure \ref{fig:radj}.  The action of \(F_{\tot}^!\) transports an object \(x\) counter-clockwise around both \(\calx_1\) and \(\calx_2\).  The counter-clockwise loop around \(\calx_2\) gives us the \(F_2^!x\in\calx_2\) and leaves us with \(T_2'x\) coming out the other end.  Transporting \(T_2'x\) around \(\calx_1\) gives us the \(F_1^!T_2'x\) of the statement.  Further, the outgoing line of this loop is then \(T_1'T_2'x\), so that the total twist is given by \(T_1'T_2'\) (see Lemma \ref{lem:tottwist} below).
\end{rmk}

With these in hand, we can compute the total twist \(T'_{\tot}\) and Serre functor \(S_{\tot}=\cofib(\id_{\calx_{\tot}}\to F_{\tot}^!F_{\tot})[d]\).  For the former, we have
\begin{lem}
\(T'_{\tot}=T'_1T'_2[1]\)
\label{lem:tottwist}
\end{lem}
\begin{proof}
Consider the diagram
\[\xymatrixcolsep{5pc}
\xymatrix{
F_1F_1^!T_2'\ar[r]\ar[d]&T_2'\ar[r]\ar[d]&T_1'T_2'[1]\ar[d]\\
F_1F_1^!T_2'\ar[r]\ar[d]&F_2F_2^!\ar[r]\ar[d]&F_{\tot}F_{\tot}^!\ar[d]\\
0\ar[r]&\id_{\caly}\ar[r]&\id_{\caly}
}
\]
All rows and the two left columns are fiber sequences, so the right column is as well.
\end{proof}
Let us rewrite this result as \(T'_{\tot}[1]=(T'_1[1])(T'_2[1])\) and recall in the case of a Landau-Ginzburg model that \(T'_i[1]\) is the monodromy around \(U_i\).  This result then says that the total monodromy is obtained by composing the monodromy around each \(U_i\).

The Serre functor cannot be written in such a simple manner, but we can simplify it a little.  Let \(S_i\) denote the Serre functor of \(\calx_i\).
\begin{lem}
Let \((x,y,f)\in\calx_{\tot}\).  Then \(S_{\tot}(x,y,f)\simeq(c_1,c_2,k)[d]\), where \(c_1=\cofib(x\to F_1^!T_2'\cofib(f))\), \(c_2=\cofib(y,F_2^!\cofib(f))\), and \(k\) is induced by \(f\) and \(g_{\cofib(f)}\).
\[
\xymatrix{
x\ar[r]\ar[d]^f&F_1^!T_2'\cofib(f)\ar[r]\ar[d]^{g_{\cofib(f)}}&c_1\ar@{-->}[d]^k\\
y\ar[r]&F_2^!\cofib(f)\ar[r]&c_2
}
\]
Furthermore, we have \(c_1\simeq\cofib(F_1^!F_2c_2[-1]\to S_{1}x[-d])\) and \(c_2\simeq\cofib(F_2^!F_1x\to S_2y[-d])\), where the maps will be defined below.
\end{lem}
\begin{proof}
As mentioned above, since \(F_{\tot}\) is compatible spherical, we have \(S_{\tot}\simeq\cofib(\id_{\calx_{\tot}}\to F_{\tot}^!F_{\tot})[d]\), from which the first statement follows immediately.  For the second statement, consider the diagram
\[
\xymatrix{
0\ar[r]\ar[d]&F_2^!F_1x\ar@{=}[r]\ar[d]&F_2^!F_1x\ar[d]\\
y\ar[r]\ar@{=}[d]&F_2^!F_2y\ar[r]\ar[d]&S_2[-d]y\ar[d]\\
y\ar[r]&F_2^!\cofib(f)\ar[r]&c_2
}.
\]
All three rows and the first two columns are fiber sequences, so the last column is as well.  Thus we have \(c_2\simeq\cofib(F_2^!F_1x\to S_2y[-d])\).  For \(c_1\), we consider
\[
\xymatrix{
F_1x\ar[r]\ar[d]&T'_2\cofib(f)\ar[r]\ar[d]&F_2c_2\ar@{=}[d]\\
F_2y\ar[r]\ar[d]&F_2F^!_2\cofib(f)\ar[r]\ar[d]&F_2c_2\ar[d]\\
\cofib(f)\ar@{=}[r]&\cofib(f)\ar[r]&0.
}
\]
Since the columns and bottom two rows are fiber sequences, we get a fiber sequence \(F_1x\to T'_2\cofib(f)\to F_2c_2\).  Then consider
\[
\xymatrix{
x\ar[r]\ar@{=}[d]&F_1^!F_1x\ar[r]\ar[d]&S_1[-d]x\ar[d]\\
x\ar[r]\ar[d]&F_1^!T'_2\cofib(f)\ar[r]\ar[d]&c_1\ar[d]\\
0\ar[r]&F_1^!F_2c_2\ar@{=}[r]&F_1^!F_2c_2.
}
\]
Once again, the rows and first two columns are fiber sequences, so the third column is as well; rotating this gives the isomorphism \(c_1\simeq\cofib(F_2^!F_2c_2[-1]\to S_1x[-d])\).
\end{proof}

\begin{rmk}\label{rmk:PSC-CY}
The above results may be generalized as follows.  We note that 
\[\Rep(A_n)\simeq \Rep(A_{n-1})\times_{\Rep(A_1)}\Rep(A_2)\]  
By inductively applying Theorem \ref{thm:main-gluing-theorem} we may construct a map \(\Rep(A_n)\to\Rep(A_1)^{\times(n+1)}\) with a compatible spherical structure.  Furthermore, if \(\caly\) has a weak Calabi-Yau structure, then applying Lemma \ref{lem:tensorspher}, we have a compatible spherical structure on \(\Rep(A_n)\otimes\caly\to\caly^{\times(n+1)}\)

Now fix a ribbon graph \(\rg\) and let us denote some number \(e\) of leaves (\(1\)-valent vertices) as belonging to \emph{outgoing} edges; the other vertices will be \emph{internal} vertices.  Assume that any edge is incident to at least one internal vertex, and that there are no loops.  Let us fix a ``generic'' category \(\caly\), and, for each internal leaf, a category \(\calx_i\) equipped with a spherical functor \(F_i:\calx_i\to\caly\) defining a perverse Schober \(\mathfrak{X}\) on a thickening of \(\rg\).  Then the construction of \cite{KaSch} allows us to associate to \(\rg\) a sheaf of categories \(\mathfrak{X}_{\rg}\) on \(\rg\) (see also the introduction of this paper).  Considering \(\Gamma(\rg,\mathfrak{X}_{\rg})\), we have a structure map \(\Gamma(\rg,\mathfrak{X}_{\rg})\to\caly\) for each outgoing edge of \(\rg\).  We then obtain a map \(F_{\rg}:\Gamma(\rg,\mathfrak{X}_{\rg})\to\caly^{\times e}\) which we may explicitly describe as follows:
\begin{enumerate}
\item If \(\rg\) has a single internal vertex \(i\), and this is a leaf corresponding to a spherical functor \(F_i:\calx_i \to\caly\), then \(e=1\), \(\Gamma(\rg,\mathfrak{X}_G)\simeq\calx\), and \(F_{\rg}= F_i\).
\item If \(\rg\) has a single internal vertex, and this is an \(n\)-valent vertex with \(n>1\), then \(\Gamma(\rg,\mathfrak{X}_{\rg})\simeq\Rep(A_{n-1})\otimes\caly\) and \(F_{\rg}:\Rep(A_{n-1})\otimes\caly\to\caly^{\times n}\) is the map described above.
\item If \(\rg\) has more than one internal vertex, fix some internal \(n\)-valent vertex \(v\) and construct a graph \(\rg'\) by removing \(v\) and any outgoing edge incident to \(v\); for any non-outgoing edge incident to \(v\) and some \(v'\) we add a new outgoing edge between \(v'\) and a new leaf.  Let \(r\) be the number of such edges.  Let \(\rg''\) be the graph with a single \(n\)-valent vertex; if \(n=1\) we assign to this vertex the functor \(F_v\) assigned to \(v\).  Then \(\Gamma(\rg,\mathfrak{X}_{\rg})\simeq\Gamma(\rg',\mathfrak{X}_{\rg'})\times_{\caly^r}\Gamma(\rg'',\mathfrak{X}_{\rg''})\).  We inductively have maps \(F_{\rg'}:\Gamma(\rg',\mathfrak{X}_{\rg'})\to\caly^{e'}\) and \(F_{\rg''}:\Gamma(\rg'',\mathfrak{X}_{\rg''})\to\caly^{e''}\).  Analogously to the construction of the map \(\calx_{\tot}\to\caly\) we may then construct \(F_{\rg}:\Gamma(\rg,\mathfrak{X}_{\rg})\to\caly^e\).
\end{enumerate}

If, furthermore, the category \(\caly\) is equipped with a weak right Calabi-Yau structure and the functors \(F_i: \calx_i\to\caly\) are equipped with weak (resp.\ strong) relative Calabi-Yau structures, then the functor \(F_{\rg}\) also has a weak (resp.\ strong) relative Calabi-Yau structure.  In the first case, this is by assumption.  In the second case, this is the structure on \(\Rep(A_{n-1})\otimes\caly\to\caly^{\times n}\) described above.  In the third, we may apply Theorem \ref{thm:main-gluing-theorem}.

We further expect that one can find formulas for adjoints similar to those of Lemma \ref{lem:glueadj}, so that \(F_{\rg}\) has a compatible spherical structure.  We do not show this here, although we note that in the simple case that \(\rg\) is a binary tree (i.e. there are no cycles, exactly one outgoing edge, and all internal vertices are uni- or trivalent) we may simply inductively apply Lemma \ref{lem:glueadj}.
\end{rmk}

Let us describe a simple representation of the category \(\calx_{\tot}\).  For \(i=1,2\) we define maps \(J_i:\calx_i\to\calx_{\tot}\) by
\begin{align*}
J_1x&=(x[-1],0,0)\\
J_2y&=(0,y,0).
\end{align*}
In the case of a Picard-Lefschetz fibration, these correspond to the inclusions \(\FS(X_i,w_i)\to \FS(X,w)\) of thimbles ending in critical points of \(U_i\).  It is easy to check the following:
\begin{lem}
With \(\calx_i,\caly,\calx_{\tot},F_i,F_{\tot},J_i\) as above, we have
\begin{enumerate}
\item \(F_{\tot}J_i=F_i\).
\item For \(i=1,2\), \(J_i\) is full and faithful.
\item For \(x\in \calx_1\), \(y\in\calx_2\), we have \(\Mor_{\calx_{\tot}}(J_1x,J_2y)\simeq\Mor_{\caly}(F_1x,F_2y)\).
\item \(\calx_{\tot}\) has a semi-orthogonal decomposition \(\calx_{\tot}=\langle J_1\calx_1,J_2\calx_2\rangle\).
\end{enumerate}
\label{lem:Jfacts}
\end{lem}

\begin{example}\label{example:Kronecker}
We will now construct Kronecker quivers as Fukaya-Seidel categories using this method.  Let \(w':X'\to\CC\) be a mirror to \(\CC\PP^2\), so that \(w'\) has three critical points corresponding to \(\calo,\calo(1),\calo(2)\in D^b(\CC\PP^2)\), and let \(E\) be a smooth fiber.  Take an open disc \(U\) containing two of these points, \(z_1\) and \(z_2\), and extend \(X'|_U\) to a Picard-Lefschetz fibration \(w:X\to\CC\) so that its only critical points are \(z_1\) and \(z_2\).  Then \(\FS(X,w)\) is \(K_3\), the Kronecker quiver with three arrows.  More specifically, \(\FS(X,w)\) is generated by thimbles \(L_i\) ending at \(z_i\) such that \(\cap L_1\) and \(\cap L_2\) are Lagrangians in \(E\) intersecting in three points.

Let \(U_i\) be a small disc around \(z_i\), and let \((X_i,w_i)\) be extensions of the fibration over \(U_i\) to all of \(\CC\).  Then we have \(\FS(X_i,w_i)=\Rep(A_1)=D^b(k)\), with \(\cap_i:\FS(X_i,w_i)\to \Fuk(E)\) sending the generator to \(\cap L_i\).  Taking \(\calx_i=\FS(X_i,w_i)\) and \(F_i=\cap_i\), it is clear from Lemma \ref{lem:Jfacts} that \(\calx_{\tot}\) is again \(K_3\).

We can alter this example with a surgery.  Fix a symplectomorphism \(R\) of \(E\), say a Dehn twist, and a cover \(U=V_1\cup V_2\) with \(z_i\in V_i\), \(z_i\notin V_j\) for \(i\neq j\).  Let \(X_i=X|_{V_i}\).  Then we have \(X|_U\simeq X_1\coprod X_2/g\), where \(g\) is an isomorphism identifying \(X_1|_{V_1\cap V_2}\simeq X_2|_{V_1\cap V_2}\).  Choosing a trivialization of \(X|_{V_1\cap V_2}\), and let \(\tilde{R}\) be the fibrewise action of \(R\) on \(X|_{V_1\cap V_2}\).  Then we may glue \(V_1\) and \(V_2\) back together via \(R\) by setting \(X_R=X_1\coprod X_2/g\tilde{R}\).  Let \((X_R,w_R)\) be the resulting Landau-Ginzburg model.  Fixing a fiber \(E\) in the same component as \(z_2\), this has the effect of replacing \(\cap L_1\) with \(R\cap L_1\); if \(HF(R\cap L_1,\cap L_2)\simeq k^n\), we will have \(\FS(X_R,w_R)\simeq K_n\), the Kronecker quiver with \(n\) arrows.

Let \(R\) also denote the induced autoequivalence of \(\Fuk(E)\).  Then the surgery has the effect that \(\cap_1\) will be replaced by \(R\cap_1\).  Thus we construct the Kronecker quiver by the diagram
\[
\xymatrix{
K_n\ar[r]\ar[d]&\calx_{2,\pm}\ar[r]^{\cap_+}\ar[d]^{\cap_-}&\Fuk(E)\\
\calx_1\ar[r]^{R\cap_1}&\Fuk(E)
}
\]

Now let us look more at \(K_n\) from the point of view of the gluing.  We have \(\calx_1\simeq\calx_2\simeq D^b(k)\), so that the functors \(F_i:\calx_i\to \caly\) are determined by the objects \(y_i=F_ik\).  Furthermore, it is easy to check that \(F_i^!z=\Mor_\caly(y_i,z)\) and \(F_i^*z=\Mor_\caly(z,y_i)^\vee\).

Let us calculate the Serre functor \(S_{\tot}\).  By Lemma \ref{lem:Jfacts}, \(\calx_{\tot}\) has a generating exceptional collection \(\langle E_1,E_2\rangle\), where \(E_1=J_1k=(k[-1],0,0)\) and \(E_2=J_2k=(0,k,0)\).  Therefore we will determine \(S_{\tot}E_i\).

For any \(z\in\caly\), Lemma \ref{lem:glueadj} gives us
\[
F_{\tot}^!z=(\fib(\Mor_\caly(y_1,y_2)\otimes\Mor_\caly(y_2,z)\to\Mor_\caly(y_1,z)),\Mor_\caly(y_2,z),g_z)
\]
Here \(g_z\) is the composition
\begin{align*}
\fib(\Mor_\caly(y_1,y_2)\otimes\Mor_\caly(y_2,z)\to\Mor_\caly(y_1,z))\otimes y_1&\to \Mor_\caly(y_1,y_2)\otimes\Mor_\caly(y_2,z)\otimes y_1\\
&\to \Mor_{\caly}(y_2,z)\otimes y_2
\end{align*}
with the second map coming from evaluation \(\Mor(y_1,y_2)\otimes y_1\to y_2\).  Let us set \(V=\Mor_{\caly}(y_1,y_2)\) and recall that the Calabi-Yau structure on \(\caly\) gives us \(\Mor_\caly(y_2,y_1)\simeq V^\vee[-1]\).  Furthermore, \(\Mor_\caly(y_i,y_i)\simeq C^*(S^1)\simeq k\oplus k[-1]\).  Then
\[
F_{\tot}^!F_{\tot}E_1\simeq(\fib(V\otimes V^\vee[-1]\to k\oplus k[-1]),V^\vee[-1],g_{y_1})
\]
and
\begin{align*}
F_{\tot}^!F_{\tot}E_2&\simeq(\fib(V\otimes (k\oplus k[-1])\to V),k\oplus k[-1],g_{y_2})\\
&\simeq(V[-1], k\oplus k[-1], g_{y_2}).
\end{align*}
Then using \(S_{\tot}E_i\simeq \cofib(E_i\to F_{\tot}^!F_{\tot}E_i)\), we have
\begin{align*}
S_{\tot}E_1&\simeq(\fib(V\otimes V^\vee[-1]\to k[-1]),V^\vee[-1],g_{y_1})\\
&\simeq(sl(V)[-1],V^\vee[-1],g_{y_1})
\end{align*}
and
\[
S_{\tot}E_2\simeq (V[-1], k[-1],g_{y_2}).
\]
\end{example}

\section{An A-model relative Calabi-Yau structure}\label{sec:A-model}

Let \(w:X\to\CC\) be a Landau-Ginzburg model with smooth and compact fiber \(Y\) of dimension \(d\).  Then we have a functor \(\cap:\FS(X,w)\to \Fuk(Y)\) which is restriction to the fiber at infinity, and the Orlov functor \(\cup:\Fuk(Y)\to \FS(X,w)\) which is a Hamiltonian flow of a Lagrangian along an arc from \(+\infty\) to itself going around all the critical values.  Furthermore, \(\cap\) is a left adjoint of \(\cup\). This pair of adjoint functors is constructed in the paper \cite{AS}, which is currently in preparation. Abouzaid has announced \cite{A} a proof of the fact that \(\cap\) is spherical.

The mirror of \((X,w)\) is a Fano variety \(X^{\vee}\). Here we use the term ``mirror'' in the sense of homological mirror symmetry, so we have an equivalence \(\FS(X,w) \simeq \Perf(X^{\vee})\). By Homological Mirror Symmetry, since \(\Perf(Y)\) has a Calabi-Yau structure, we expect its mirror \(\Fuk(Y)\) to have one as well.  When \(Y\) is compact, Ganatra \cite{Gan2} has shown that \(\Fuk(Y)\) does indeed carry a right Calabi-Yau structure. Furthermore, the functor \(\cap:\FS(X,w)\to \Fuk(Y)\) is mirror to \(a^*:\Perf(X^{\vee})\to\Perf(Y^{\vee})\), the restriction functor from \(X^{\vee}\) to an anticanonical divisor \(Y^{\vee} \hookrightarrow X^{\vee}\). The latter functor is shown to carry a relative Calabi-Yau structure in \cite{BD}; this is a noncommutative/categorical version of a result of Calaque \cite{Calaque} stating that the induced map on moduli spaces of perfect complexes carries a Lagrangian structure. By mirror symmetry, we should expect \(\cap:\FS(X,w)\to \Fuk(Y)\) to have a relative Calabi-Yau structure. This section is devoted to outlining an argument demonstrating the existence of this structure.

 The main result of this section, Theorem \ref{A-model-rel-CY}, is essentially a formulation of the statement that \(\cap:\FS(X,w)\to \Fuk(Y)\) carries a relative CY-structure. Our formulation axiomatizes the inputs from symplectic geometry that are required in the proof. The reasons for treating these inputs as a black-box are twofold. First, proofs for many of the facts that we need from symplectic geometry are not yet available, and are the subject of works by experts in symplectic geometry that are currently in progress \cite{AS,AG,Gan1}. Second, we wish to make manifest the extent to which our argument is robust, and independent of specific features of the symplectic setup. \\

\noindent \textbf{Outline.} Here is an outline of the argument:

\begin{itemize}

\item[-] In Definition \ref{def:pre-isotropy}, we introduce extra structure on the functor that allows us to reduce the construction of isotropy data  to elementary topology, as described in Construction \ref{constr:pre-isotropy-gives-isotropy}.

\item[-] Definition \ref{def:admissible-LG} axiomatizes various formal features of explicit chain models for mapping spaces in \(\FS(X,w)\) and \(\Fuk(Y)\), including compatibility conditions with \(\cap\) and the Calabi-Yau structure on \(\Fuk(Y)\), that enter into the proof of the non-degeneracy of the isotropy structure given by Construction \ref{constr:pre-isotropy-gives-isotropy}. 

\end{itemize}

\begin{defn}\label{def:LG}
In this section, by a \emph{Landau-Ginzburg model} we mean a symplectic manifold \((X,\omega)\) equipped with a smooth morphism \(w: X \rightarrow \mathbb{C}\) with the following properties:

\begin{itemize}

\item[-] there is a finite collection of points \(\{p_1,...,p_n\}\) in \(\mathbb{C}\) such that  \(w\) defines a locally trivial fibration on \(\mathbb{C} -\{p_1,...,p_n\}\) whose generic fiber is a symplectic manifold \((Y, \omega_{|Y})\).

\item[-] the structure group of this locally trivial fibration \(\mathbb{C} -\{p_1,...,p_n\}\) is contained in the symplectomorphism group \(\mathrm{Symp}(Y, \omega_{|Y})\) of the fiber.
\end{itemize}

\end{defn}

Given a Landau-Ginzburg model \((X,w)\), let \(X_{\infty}\) be its fiber at infinity. We are concerned only with its homotopy type. For instance, let \(U\) be the intersection of a sector in \(\CC\) that contains the positive real axis with the complement of a large compact set containing the critical values of \(w\). Then we can model \(X_{\infty}\) by \(w^{-1}(U) \subset X\). Then clearly we have a (canonical up to contractible ambiguity) homotopy equivalence \(X_{\infty} \simeq Y_{t}\) for any fiber \(Y_{t} := w^{-1}(t)\) with \(t \in U\). If \(Y = Y_s\) is any smooth fiber of \(w\), then we still have an non-canonical homotopy equivalence  \(X_{\infty} \simeq Y_{s}\). In the discussions that follow, the fiber \(Y\) will always be fixed, and we will assume that we have chosen such a homotopy equivalence. Composing the boundary map \(\partial\) in relative homology with this map induced by this homotopy equivalence, we obtain a map

\[
\partial: C_*(X,X_{\infty},k) \rightarrow C_*(Y,k)
\]

which we continue to denote by \(\partial\), abusing notation. 

\begin{defn}\label{def:pre-isotropy}
Let \(\calx\) and \(\caly\) be locally proper \(k\)-linear stable \(\infty\)-categories, let \(F: \calx \rightarrow \caly\) be an exact functor, and let \(\phi_{\caly}: \HH_*(\caly) \rightarrow k[-d]\)  be a weak  Calabi-Yau structure on \(\caly\). A \emph{weak pre-isotropy structure on \(F\) with respect to \(w\)} consists of the following data:

\begin{enumerate}

\item a homotopy commutative diagram as follows in the \(\infty\)-category of \(k\)-modules:
\[
\xymatrix{
\HH_*(\calx) \ar[r] \ar[d]_{\theta_{\calx}} & \HH_*(\caly) \ar[d]^{\theta_{\caly}} \\
\chain_*(X,X_{\infty})[-d-1] \ar[r]^(0.6){\partial} & \chain_{*}(Y)[-d]\\
}
\]

Here \(Y\) is some fixed smooth fiber of \(w\) and \(\chain_*(Y) = \chain_*(Y;k)\) is a chain complex computing the homology \(\hml_*(Y;k)\) of the topological space \(Y\) with coefficients in \(k\), and \(X_{\infty}\) is the fiber at infinity, so that \(\chain_*(X; X_{\infty})\) can be identified with the chain complex of vanishing cycles of the fibration defined by \(w\). 

\item a homotopy commutative diagram of \(k\)-modules:
\[
\xymatrix{
\HH_*(\caly) \ar[r]^{\phi_{\caly}} \ar[d]_{\theta_{\caly}} & k[-d] \\
\chain_*(Y)[-d] \ar[r]_{\pr_0[-d]} & \hml_0(Y)[-d] \ar[u]_{\deg[-d]} \\
}
\]

Here \(\hml_0(Y) \simeq \tau_{\leq 0} \chain_*(Y)\) since \(\chain_*(Y)\) is connective, and \(\pr_0: \chain_*(Y)  \rightarrow \tau_{\leq 0} \chain_*(Y)\) is simply the unit of the natural adjunction between spectra and connective spectra. The map \(\deg\) is the linear extension of the map that sends the class of any point in \(Y\) to the element \(1 \in k\).

\end{enumerate}

Equip the chain complex \(\chain_*(-)\) with a homotopically trivial \(S^1\)-equivariant structure, and suppose that \(\caly\) is equipped with a \emph{strong} Calabi-Yau structure \(\tilde{\phi}_{\caly}: \HC_*(\caly) \rightarrow k[-d]\). Then a \emph{strong pre-isotropy structure on \(F\) with respect to \(w\)} is the data of commutative diagrams as in (1) and (2) above, taking values in the \(\infty\)-category of \emph{\(S^1\)-equivariant} \(k\)-modules. 
\end{defn}

\begin{rmk}\label{rmk:smooth-FS-has-pre-isotropy}
Let \(X\) be an exact symplectic manifold, and let \(w: X \rightarrow \CC\) be a Lefschetz fibration satisfying the hypotheses in \cite{Seidel-book}, with exact generic fiber \(Y\). Then the Fukaya-Seidel category \(\FS(X,w)\) and the Fukaya category \(\Fuk(Y)\) are defined over \(k = \CC\). By \cite{Gan1}, we have \(\HH_*(\Fuk(Y))\simeq \SH^{*-d}(Y)\), the symplectic cohomology of \(F\).  We may take a model of \(\SH^{*-d}(Y)\) whose generators are constant loops so that (as vector spaces) we have \(\SH^{*-d}(Y)\simeq \chain_{d-*}(Y, \CC)\). For \(\FS(X,w)\) we have \(\HH_*(\FS(X,w))\simeq \chain_{d+1-*}(X, X_\infty, \CC)\), the vanishing cycles of \(X\), with the trivial \(S^1\) action \cite{AG}. Furthermore, in the forthcoming paper \cite{AS} a functor \(\cap: \FS(X,w) \rightarrow \Fuk(Y)\) is constructed, which acts on the support of an A-brane by simply intersecting it with a generic smooth fiber. We expect that the results proven therein will imply that the diagram
\[
\xymatrix{
\HH_*(\FS(X,w)) \ar[r] \ar[d] & \HH_*(\Fuk(Y)) \ar[d] \\
\chain_*(X,X_{\infty})[-d-1] \ar[r] & \chain_{*}(Y)[-d]\\
}
\]
commutes in \(S^1\)-modules, thus equipping the functor \(\cap\) with a pre-isotropy structure with respect to \(w\). In the diagram above, the right vertical equivalence is the composite of the open-closed map \(\HH_*(\Fuk(Y) \rightarrow \SH^{*-d}(Y)\) with the equivalence \(\SH^{*-d}(Y)\simeq \chain_{d-*}(Y)\) mentioned earlier.

\end{rmk}

\begin{rmk}\label{rmk:proper-FS-has-pre-isotropy}
Let \(X\) be an symplectic manifold, and let \(w: X \rightarrow \CC\) be a symplectic fibration, with \emph{compact} generic fiber \(Y\). In this situation one expects to be able to define categories \(\FS(X,w)\) and \(\Fuk(Y)\) that are linear over the Novikov field \(\Lambda := \CC((t^{\RR}))\), and a functor \(\cap: \FS(X,w) \rightarrow \Fuk(Y)\) as in Remark \ref{rmk:smooth-FS-has-pre-isotropy}. The properness of the fiber \(Y\) will be reflected in the properness of these categories. It is expected \cite{Kontsevich-HMS} (see also \cite{Gan1}) that \(\HH_*(\Fuk(Y))\) is equivalent to the quantum cohomology \(\mathrm{QH^*}(Y)\) via the open-closed map. The underlying chain complex  \(\mathrm{QH^*}(Y)\) is the cochain complex \(\chain^*(Y, \Lambda)\), which via Poincare duality is identified with \(\chain_*(Y, \Lambda)[-d]\). Similarly, one expects that  \(\HH_*(\FS(X,w)) \simeq \chain_*(X, X_{\infty}, \Lambda)\) as \(S^1\)-equivariant chain complexes. As in Remark \ref{rmk:smooth-FS-has-pre-isotropy}, the functor \(\cap\) should be equipped with a pre-isotropy structure with respect to \(w\); the main difference is that all the structures are defined over \(\Lambda\) instead of \(\CC\). 

\end{rmk}

\begin{construction}\label{constr:pre-isotropy-gives-isotropy}
Suppose that we are given a functor \(F: \calx \rightarrow \caly\) that is equipped with a weak (resp. strong) pre-isotropy structure with respect to a Landau-Ginzburg model \(w: X \rightarrow \CC\) as in Definition \ref{def:pre-isotropy}. We are going to construct a weak (resp. strong) isotropy structure on \(F\) (Definition \ref{def:isotropy}) from this data.  Consider the diagram:
\[
\xymatrix{
\HH_*(\calx) \ar[r]^{\HH_*(F)} \ar[d]_{\theta_{\calx}} & \HH_*(\caly) \ar[d]^{\theta_{\caly}} \\
\chain_*(X,X_{\infty})[-d-1] \ar[r]^(0.6){\partial} \ar[d]_{\pr_1[-d]} & \chain_*(Y)[-d] \ar[d]^{\pr_0[-d]}\\
\hml_1(X,X_{\infty})[-d] \ar[r]^(0.6){\partial} \ar[d] & \hml_0(Y)[-d] \ar[d]^{\deg[-d]} \\
0 \ar[r] & k[-d]\\
}
\]

Observe that \(\chain_*(X,X_{\infty})\) is \(1\)-connective and so \(\hml_1(X,X_{\infty}) \simeq \tau_{\leq 1}\chain_*(X, X_{\infty})\); the map \(\pr_1: \chain_*(X,X_{\infty}) \rightarrow \tau_{\leq 1} \chain_*(X,X_{\infty})\) is simply the unit of the natural adjunction between spectra and \(1\)-connective spectra. The two lower squares are manifestly commutative as diagrams in the \(\infty\)-category of \(S^1\)-equivariant \(k\)-modules. In particular, they are also commutative as diagrams in the \(\infty\)-category \(\Mod_k\) of \(k\)-modules. Thus,

\begin{enumerate}
\item If \(F\) is equipped with a weak pre-isotropy structure with respect to \(w\), then the upper square defines a homotopy commutative square in \(\Mod_k\). Therefore the outer square is also commutative in \(\Mod_k\), and defines a weak isotropy structure on \(F\).

\item If \(F\) is equipped with a strong pre-isotropy structure with respect to \(w\), then the upper square defines a homotopy commutative square of \(S^1\)-equivariant \(k\)-modules. Therefore the outer square is also equipped with the structure of a commutative square  of \(S^1\)-equivariant \(k\)-modules, and thus defines a strong isotropy structure on \(F\).
\end{enumerate}

\end{construction}

Having discussed the geometric structures required to produce an isotropy structure on \(F\), we now turn our attention to non-degeneracy. The following lemma will allow us to reduce the question of non-degeneracy to a generator set of objects that have geometric representatives:

\begin{lem}\label{lemma:non-deg-for-generators}
Let \(F: \calx \rightarrow \caly\) be a functor of locally proper \(k\)-linear stable \(\infty\)-categories equipped with an isotropy structure with respect to a Calabi-Yau structure \(\phi_{\caly}\) on \(\caly\), and let \(\Gen_{\calx} \subset \Ob_{\calx}\) be a set of objects that generates \(\calx\) under finite colimits. For \(x, y \in \Ob_{\calx}\), consider the homotopy commutative square:
\begin{equation}\label{diag:lemma-non-deg-gen}\tag{\(\diamondsuit\)}
\xymatrix{
\Mor_{\calx}(x,y) \ar[r] \ar[d] & \Mor_{\caly}(Fx, Fy) \ar[d] \\
0 \ar[r] & \Mor_{\calx}(y,x)^{\vee}[-d] \\
}
\end{equation}
obtained by applying Construction \ref{const:fiber-sequence}. Assume that \eqref{diag:lemma-non-deg-gen} is a pullback square for all \(x,y\) in \(\Gen_{\calx}\). Then  \eqref{diag:lemma-non-deg-gen} is a pullback square for \(x, y\) in \(\Ob_{\calx} \). 
\end{lem}

\begin{proof}

Let \(x \in \Gen_{\calx}\). Define \(\nondeg^r(x):= \{y \in \Ob_{\calx} \mid\) Diagram \eqref{diag:lemma-non-deg-gen} is a pullback square\(\}\). Let \(y', y \in \nondeg^r(x)\), and suppose \(y' \to y \to y''\) is a fiber-cofiber sequence in the stable \(\infty\)-category \(\calx\). Consider the diagram

\[
\xymatrix{
\Mor_{\calx}(x,y') \ar[r] \ar[d] & \Mor_{\caly}(Fx, Fy') \ar[r] \ar[d] & \Mor_{\calx}(y',x)^{\vee}[-d] \ar[d] \\
\Mor_{\calx}(x,y) \ar[r] \ar[d] & \Mor_{\caly}(Fx, Fy) \ar[r] \ar[d] & \Mor_{\calx}(y,x)^{\vee}[-d] \ar[d] \\
\Mor_{\calx}(x,y'') \ar[r] & \Mor_{\caly}(Fx, Fy'') \ar[r]  & \Mor_{\calx}(y'',x)^{\vee}[-d]  \\
}
\]

The commutativity of all the squares in this diagram derives from the functoriality of \(\Mor(-,-)\), and the fact that the collection of morphisms \(\Xi(\phi): \Mor_{\caly}(Fx,Fy) \to \Mor_{\caly}(Fy, Fx)^{\vee}[-d]\) underlies a natural transformation of functors. For each row in this diagram, the composite of the morphisms in that row is equipped with a nullhomotopy given by the isotropy structure on \(F\). All of this information fits together to give a functor \(\Delta^2 \times \Delta^1  \times \Delta^1 \to \Mod_k\).

The columns of the diagram above are fiber sequences since \(\Mor_{\calx}(x,-)\) and \(F\) are exact functors, and \(y' \to y \to y''\) is a fiber sequence. Thus the bottom row is the cofiber of the map from the top row  to the middle row. Furthermore, the upper two rows of this diagram are fiber sequences, since \(y', y \in \nondeg^r(x)\) by assumption. Since cofiber of a map between fiber sequences is a fiber sequence, the bottom row is also a fiber sequence. We conclude that \(y'' \in \nondeg^r(x)\), and hence that \( \nondeg^r(x)\) is stable under taking cones.

 Clearly \(0 \in \nondeg^r(x)\). Furthermore, by the hypothesis in the statement of the lemma, \(\Gen_{\calx} \subset \nondeg^r(x)\). Since \(\Gen_{\calx}\) generates \(\calx\) under finite colimits, we deduce from the conclusion of the previous paragraph (and the fact that \(\calx\) is stable) that \(\nondeg^r(x) = \Ob_{\calx}\).

Now let \(y \in \Ob_{\calx}\) be an arbitrary object, and let \(\nondeg^l(y) := \{x \in \Ob_{\calx} \mid\) Diagram \eqref{diag:lemma-non-deg-gen} is a pullback square\(\}\). Then clearly \(0 \in \nondeg^l(y)\), and \(\Gen_{\calx} \subset \nondeg^l(y)\) by the conclusion of the previous paragraph. By applying the same argument as above to a fiber sequence \(x' \to x \to x''\), we see that if \(x', x \in \nondeg^l(y)\), then \(x'' \in \nondeg^l(y)\). Since \(\calx\) is generated under finite colimits by \(\Gen_{\calx}\), and \(\Gen_{\calx} \cup \{0\} \subset \nondeg^l(y)\), we conclude that \(\nondeg^l(y) = \Ob_{\calx}\) for \emph{arbitrary} \(y \in \Ob_{\calx}\). This is exactly what we set out to prove.

\end{proof}

In order to prove that the isotropy structure that we have produced is non-degenerate, we will

\begin{notation}\label{not:Ham-traj}
Let \((X,\omega)\) be a symplectic manifold, and let 

\begin{itemize}
\item[-] \(H \in C^{\infty}(X, \mathbb{R})\)
\item[-] \(\calh \subset  C^{\infty}(X, \mathbb{R})\)
\item[-] \(L, L' \in \mathrm{Lag}(X)\), the set of smooth Lagrangian submanifolds of \(X\). 
\end{itemize}

Then we will use the following notation:

\begin{enumerate}
\item \(\traj_H\) denotes the set of time \(1\) \(H\)-Hamiltonian trajectories; i.e., \(\traj_H := \{ \gamma:[0,1] \rightarrow X \mid \gamma\) is smooth and \(\dot{\gamma}(t) = \mathfrak{X}_{H}(\gamma(t))\}\). Here  \(\mathfrak{X}_{H}\) is the vector field on \(X\) characterized by \(\omega(\mathfrak{X}_H,-) = -dH\).

\item \(\traj_{\calh} : = \cup_{H \in \calh} \traj_H\).

\item \(\traj_{\calh}(L,L') := \{\gamma \in \traj_{\calh} \mid \gamma(0) \in L\) and \(\gamma(1) \in L'\}\).

\end{enumerate}

\end{notation}

\begin{rmk}\label{rmk:reverse-Ham-traj}
For any \(\calh \subset C^{\infty}(X)\), and for any Lagrangians \(L,L'\), there is a natural bijection
\[
\xymatrix{
\rev_{\calh}: \traj_{\calh}(L,L') \ar[r]^{\sim} & \traj_{- \calh}(L',L)
}
\]
obtained by sending \(\gamma \in \traj_{\calh}(L,L')\) to \(\gamma \circ \delta\) where \(\delta: [0,1] \rightarrow [0,1]\) is given by \(\delta(t) = 1-t\).

\end{rmk}

\begin{notation}\label{not:vect}

Some notation for linear algebra constructions:

\begin{enumerate}
\item We will write \((-)^{\natural}: \mathrm{Chain}(k) \rightarrow \mathrm{Vect}\) for the forgetful functor from the \(1\)-category of chain complexes of vector spaces over \(k\) to the \(1\)-category of vector spaces.

\item \(\free_k: \mathrm{Sets} \rightarrow \mathrm{Vect}_k\) is the left adjoint to the forgetful functor from vector spaces over \(k\) to sets.

\item Let \(\calb\) be a basis for a graded vector space \(V\). We will write \(\calb_{d}^*\) for the dual basis of \(V^{\vee}[d] = \Mor(V,k[d])\). For \(\gamma \in \calb\), \(\gamma_{d}^* \in \calb_{d}^*\) will denote the dual element, so that \(\gamma^*_d(\gamma) = 1\) and \(\gamma^*_d(\gamma') = 0\) for all \(\gamma' \neq \gamma\). When \(d\) is clear from the context we will \emph{suppress \(d\) in the notation}.

\end{enumerate}
\end{notation}

\begin{defn}\label{def:admissible-LG}
Let \(w: X \rightarrow \CC\) be a Landau-Ginzburg model as in Definition \ref{def:LG}, and let \(k\) be a field. Suppose that the generic fiber \(Y\) is of real dimension \(2d\). A  \emph{weak (resp. strong) admissible categorical LG-formalism} for \((X,w)\) consists of

\begin{enumerate}
\item[(D1)] A smooth function \(H: X \rightarrow \mathbb{R}\). Define \(\calhl := \{H\}\), \(\calhr := - \calhl\) and \(\calh := \calhl \cup \calhr\).

\item[(D2)] An exact functor \(F: \calx \rightarrow \caly\) of locally proper \(k\)-linear stable infinity categories equipped with a weak (resp.\ strong) pre-isotropy structure (Definition \ref{def:pre-isotropy}) for a given weak (resp.\ strong) \(d\)-Calabi-Yau structure \(\phi\) on \(\caly\). 

\item[(D3)] A set \(\Gen_{\calx} \subset \Ob_{\calx}\) that generates \(\calx\) under finite colimits, a map  (Notation \ref{not:Ham-traj})
\[\Gen_{\calx} \times \Gen_{\calx} \rightarrow \mathrm{Lag}(X) \times \mathrm{Lag}(X)\] 
\[(x,y) \mapsto (L_{x,y}, L_{x,y}')\]
and isomorphisms of sets \(\sigma_{x,y}: \traj_{\calhl}(L_{x,y}, L'_{x,y}) \to \traj_{\calhl}(L_{y,x}', L_{y,x})\) (see Remark \ref{rmk:reduction}).

\item[(D4)]\label{model:l} For all \((x,y) \in \Gen^{\times 2}_{\calx}\) a model for \(\Mor_{\calx}(x,y)\) as a chain complex \(\Hom_{\calx}(x,y)\) whose underlying vector space has as basis the set \(\calbl_{x,y} := \traj_{\calhl}(L_{x,y},L_{x,y}')\). In particular, we have an isomorphism of vector spaces \(\mu_{x,y}: \free_k(\calbl_{x,y}) \rightarrow \Hom_{\calx}(x,y)^{\natural}\) (Notation \ref{not:vect}).

\item[(D5)]\label{model:lr} For all \((x,y) \in \Gen^{\times 2}_{\calx}\) a model for \(\Mor_{\caly}(Fx,Fy)\) as a chain complex \(\Hom_{\caly}(Fx,Fy)\) whose underlying vector space has as basis the set \(\calb_{x,y} := \traj_{\calh}(L_{x,y},L_{x,y}')\).  In particular, we have an isomorphism of vector spaces \(\nu_{x,y}: \free_k(\calb_{x,y}) \rightarrow \Hom_{\caly}(Fx,Fy)^{\natural}\).

\end{enumerate}

satisfying the following conditions:

\begin{enumerate}

\item[(A1)] For all \((x,y) \in \Gen_{\calx}^{\times 2}\), we have \(L_{x,y} \cap L'_{x,y} = \emptyset\), so that, in particular,
\[\traj_{\calhl}(L_{x,y}, L_{x,y}') \cap \traj_{\calhr}(L_{x,y}, L_{x,y}') = \emptyset\]
and consequently
\[\calb_{x,y} = \calbl_{x,y} \coprod \calbr_{x,y}\]
is a \emph{disjoint} union.

\item[(A2)] For each \(x,y \in \calx\), the following diagram in the \(1\)-category of graded vector spaces commutes (see Notation \ref{not:vect}):
\[
\xymatrix{
\free_{k}(\calbl_{x,y}) \ar[r]^{\free_k(\iota)} \ar[d]^(0.45)[@]{\sim}_{\mu_{x,y}} & \free_k(\calb_{x,y}) \ar[d]^{\nu_{x,y}}_(0.45)[@]{\sim} \\
\Hom_{\calx}(x,y)^{\natural} \ar[r]_{F^{\natural}} & \Hom_{\caly}(Fx, Fy)^{\natural} \\
}
\]

Here the upper horizontal map is induced by the inclusion \(\iota: \calbl_{x,y} \rightarrow \calb_{x,y}\).

\item[(A3)] For each \(x,y \in \calx\), the following diagram in the \(1\)-category of vector spaces commutes :
\[
\xymatrix{
\free_k( \calb_{x,y}) \ar[r]^{\free_k(\eta)} \ar[d]^(0.45)[@]{\sim} & \free_k(\calb^*_{y,x}) \ar[d]^(0.45)[@]{\sim} \\
\Hom_{\caly}(Fx,Fy)^{\natural} \ar[r]_{\Xi(\phi)^{\natural}}  & \Hom_{\caly}(Fy,Fx)^{\vee}[-d]^{\natural} \\
}
\]
Here

\begin{itemize}
\item[-] The vertical maps are given by (D5).
\item[-] The lower horizontal map is induced by the Calabi-Yau structure on \(\caly\) (Definition \ref{def:cy}).
\item[-] The map \(\eta\) defining the upper horizontal arrow arises as follows. Composing the map \(\rev_{\calhl}\) of Remark \ref{rmk:reverse-Ham-traj} with the map \(\sigma_{x,y}\) from item (D3) gives an isomorphism \(\theta: \rev_{\calhl} \circ \sigma_{x,y}: \traj_{\calhl}(L_{x,y}, L'_{x,y}) \to \traj_{\calhr}(L_{y,x},L'_{y,x})\). Similarly, \(\theta':= \rev_{\calhl} \circ \sigma_{y,x}\) is an isomorphism. Let \(\tau: \calb_{y,x} \to \calb^*_{y,x;d}\) be the tautological isomorphism. Then  \(\eta := \tau \circ (\theta \sqcup (\theta')^{-1}) : \calb_{x,y} \to \calb^*_{y,x}\) is the isomorphism of the upper horizontal row.\\

\end{itemize}
\end{enumerate}

A Landau-Ginzburg model is \emph{admissible} if the functor \(\cap: \FS(X,w) \to \Fuk(Y)\) is defined, and underlies an admissible categorical LG-formalism for \((X,w)\). 
\end{defn}

\begin{rmk}\label{rmk:reduction}
For our purposes, two admissible categorical formalisms will be equivalent if the underlying functors with pre-isotropy data are equivalent in the obvious sense. So, without loss of generality, we may choose (D3) in Definition \ref{def:admissible-LG} so that for all \(x \neq y\), \(L_{x,y} = L'_{y,x}\) and \(\sigma_{x,y} = \id\). 
\end{rmk}

\begin{rmk}\label{rmk:LG-models-are-admissible}
Based on ongoing work of Abouzaid and Ganatra \cite{AG}, we expect that the general Landau-Ginzburg model is in fact admissible in the sense of Definition \ref{def:admissible-LG}. Remarks \ref{rmk:smooth-FS-has-pre-isotropy} and \ref{rmk:proper-FS-has-pre-isotropy} discussed pre-isotropy data on \(\cap\). Specific details pertaining to the other items in \ref{def:admissible-LG} will depend on the particular model we choose for \(\FS(X,w)\). Suppose that \(\FS(X,w)\) is generated by noncompact (decorated) Lagrangians \(L \subset X\) whose image under \(w\) is required to intersect the complement of some fixed compact set \(K \subset \CC\) in a finite union of rays parallel to the positive real axis. In this situation, the Hamiltonian \(H: X \to \mathbb{R}\) from  Definition \ref{def:admissible-LG}, (D1), can be taken to be the pullback of a smooth function on \(H^{\natural}: \CC \to \mathbb{R}\) with \(dH^{\natural}\) supported near infinity in some sector containing the positive real axis and generating a counterclockwise flow on \(\CC\) satisfying some technical hypotheses.  Let \(x,y \in \FS(X,w)\), and suppose that we can choose Lagrangians \(L,L'\) representing \(x, y\) respectively. Then \(\Mor_{\FS(X,w)}(x,y)\) is computed by the Floer complex \(\mathrm{CF}^*(\Phi_{H}^1(L), L')\), where \(\Phi_H^t\) is the flow generated by the Hamiltonian \(H\). A basis for this vector space is given by the set of intersection points of \(\Phi_{H}^1(L)\) and \(L'\), which in turn can be identified with the set of time \(1\) Hamiltonian trajectories starting on \(L\) and ending on \(L'\). This motivates the item (D4) in Definition \ref{def:admissible-LG}.

At a heuristic level, the functor \(\cap\) is given by intersecting a Lagrangian \(L\) as above with a smooth fiber. Fix a smooth fiber \(Y := w^{-1}(t_0)\) for some \(t_0\) near infinity along the positive real axis. Suppose that we are given an object \(x\) in \(\FS(X,w)\) that is represented by a Lagrangian \(L\) that, outside some compact set, projects to a finite union \(l := \sqcup_i l_i\) of rays parallel to the positive real axis in \(\CC\). For \(t \in \mathbb{R}\), let \(S_t := \{\mathrm{Re}(z) = t\} \cap l =: \{t_1,...,t_n\}\) be the intersection of the vertical line through \(t\) with our family of horizontal rays.  For each \(t_i \in S_t\), let \(\gamma_i\) be a path in \(\CC\) connecting \(t_i\) to \(t_0\). Let \(L_i \subset Y\) be the Lagrangian obtained by applying symplectic parallel transport along \(\gamma_i\) to \(L \cap w^{-1}(t_i)\). Then, roughly speaking, \(\cap x\) can be represented by the Lagrangian \( \cup L_i\).

Now, if \(y\) is another object in \(\FS(X,w)\), represented by a Lagrangian \(L'\), then the mapping space \(\Mor_{\Fuk(Y)}(\cap x, \cap y)\) is computed by the Floer complex \(\mathrm{CF}^*(\cup_i L_i, \cup_j L'_j)\), which has as vector space basis the set of intersection points of \(\cup_i L_i\) and \(\cup_j L'_j\). If \(L\) and \(L'\) are disjoint, one can argue that this set coincides with the set \(\traj_{\calh}(L,L')\) of time \(1\) Hamiltonian trajectories starting on one of the Lagrangians and ending on the other. This motivates (D5) and (A1) in Definition \ref{def:admissible-LG}.
\end{rmk}

\begin{thm}\label{A-model-rel-CY}
Let \(w:X\to\CC\) be a Landau-Ginzburg model equipped with a strong (resp.\ weak) admissible categorical LG-formalism (Definition \ref{def:admissible-LG}) with underlying functor \(F: \calx \rightarrow \caly\). Then the functor \(F\) carries a strong (resp.\ weak) relative Calabi-Yau structure (Definition \ref{def:relcy}). 
\end{thm}

\begin{proof}[Sketch of proof.]

By Definition \ref{def:admissible-LG}, (D2), \(F\) is equipped with a  strong (resp.\ weak) pre-isotropy structure, as defined in Definition \ref{def:pre-isotropy}. Applying Construction \ref{constr:pre-isotropy-gives-isotropy}, we obtain a  strong (resp.\ weak) isotropy structure on \(F\). In order to prove that this isotropy structure defines a relative Calabi-Yau structure, it remains only to verify that for all \(x, y \in \Ob_{\calx}\), the square

\begin{equation}\label{diag:non-deg-proof}\tag{\(\spadesuit\)}
\xymatrix{
\Mor_{\calx}(x,y) \ar[r] \ar[d] & \Mor_{\caly}(Fx, Fy) \ar[d] \\
0 \ar[r] & \Mor_{\calx}(y,x)^{\vee}[-d] \\
}
\end{equation}

obtained by applying Construction \ref{const:fiber-sequence} to this isotropy structure is a pullback square of spectra. By Lemma \ref{lemma:non-deg-for-generators}, it suffices to check that this diagram is a pullback square for all \(x,y\) in the set of generators \(\Gen_{\calx}\) given by item  (D3) of Definition \ref{def:admissible-LG}. To this end, for \(x,y \in \Gen_{\calx}\), consider the diagram of vector spaces given by (D4) and (D5):

\[
\xymatrix{
\free_k(\calbl_{x,y}) \ar[r] \ar[d]_{\mu_{x,y}}^(0.45)[@]{\sim} & \free_k(\calbl_{x,y} \coprod \calbr_{x,y}) \ar[r]^(0.45){\sim} \ar[d]_{\nu_{x,y}}^(0.45)[@]{\sim} & \free_k((\calbr_{y,x} \coprod \calbl_{y,x})^*) \ar[r] \ar[d]^(0.45)[@]{\sim} & \free_k((\calbl_{y,x})^*) \ar[d]^(0.45)[@]{\sim} \\
\Hom_{\calx}(x,y)^{\natural} \ar[r] & \Hom_{\caly}(Fx,Fy)^{\natural} \ar[r]^(0.4){\sim} & (\Hom_{\caly}(Fy,Fx)^{\vee}[-d])^{\natural} \ar[r] & (\Hom_{\calx}(y,x)^{\vee}[-d])^{\natural} \\
}
\]

The left hand square is the commutative square of (A2) of Definition \ref{def:admissible-LG}. The middle square is the commutative square of (A3) from the same definition. The right hand square is the \(k\)-linear dual (shifted by \([-d]\)) of the commutative diagram of (A2), with the roles of \(x\) and \(y\) interchanged. Thus the entire diagram commutes.

Let \(\iota^*: \calbl_{y,x}: \rightarrow \calbl_{x,y} \coprod \calbr_{x,y}\) be the composite of the natural inclusion \(\calbl_{y,x} \subset \calbr_{y,x} \coprod \calbl_{y,x}\) and the isomorphism \(\eta^{-1}\) from (A3) of \ref{def:admissible-LG}. Then \(\iota \coprod \iota^*: \calbl_{x,y} \coprod \calbl_{y,x} \rightarrow \calb_{x,y}\) realizes \(\calb_{x,y}\) as a coproduct. Since the free vector space functor carries coproducts to direct sums, we conclude that the sequence

\[
0 \rightarrow \free_k(\calbl_{x,y}) \rightarrow \free_k(\calbl_{x,y} \coprod \calbr_{x,y}) \rightarrow \free_k((\calbl_{y,x})^*) \rightarrow 0
\]

from the top row of the previous diagram is an exact sequence of vector spaces. The vertical arrows in said diagram are all isomorphisms by hypothesis (see (D4) and D5)), so the lower row allows defines an exact sequence. Since the forgetful functor from the abelian \(1\)-category of chain complexes to the category of vector spaces reflects exact sequences, we conclude that the sequence

\[
0 \rightarrow \Hom_{\calx}(x,y) \rightarrow \Hom_{\caly}(Fx,Fy) \rightarrow \Hom_{\calx}(y,x)^{\vee}[-d] \rightarrow 0
\]

from the bottom row is an exact sequence of chain complexes. Passing from the abelian category of chain complexes to the stable \(\infty\)-category \(\Mod_k\) of \(k\)-module spectra, this gives rise to a \emph{pullback-pushout} square that agrees with the diagram \ref{diag:non-deg-proof} on the \emph{outer 1-simplices}. To complete the proof, we need to show that the homotopies witnessing the commutativity of the two diagrams are essentially the same; more precisely, we need to show that the pullback square that we have just constructed is actually equivalent to the square constructed from isotropy data in Diagram \ref{diag:non-deg-proof} as a \emph{homotopy coherent diagram}. We expect that the techniques to be developed in \cite{AG, AS} will allow us to verify this.
\end{proof}


\section{Shifted Symplectic Structures}\label{sec:shifted-symplectic}

This section consists of two logically independent subsections:

\begin{enumerate}

\item \S\ref{subsec:open-var} is devoted to the proof of Theorem \ref{thm:openmap}, which states that the moduli of compactly supported perfect complexes on certain open Calabi-Yau varieties carries a natural shifted symplectic structure. 

\item \S\ref{subsec:pushforward} is devoted to the proof of Theorem \ref{thm:pfperf}, which states that the map on moduli spaces of objects induced by the pushforward functor from perfect complexes on a smooth divisor to perfect complexes on an ambient smooth and proper Calabi-Yau variety carries a natural Lagrangian structure.  

\end{enumerate}

These theorems were motivated and placed within the larger context of this paper in \S\ref{subsec:nc-symp} of the introduction. Here, we begin by cursorily recalling the relevant definitions and results regarding shifted symplectic and Lagrangian structures, referring the reader to \cite{PTVV} for a more detailed and precise discussion. In this section, \(k\) will be a fixed base field, of characteristic \(0\). Let \(X\) be a derived Artin stack with cotangent complex \(\Omega_{X}\).  We can form the de Rham algebra \(\Omega^*_X=\Sym^*_{\calo_X}(\Omega_{X}[1])\).  This is a weighted sheaf whose weight \(p\) piece is \(\Omega^p_X=\Sym^p_{\calo_X}(\Omega_{X}[1])=\wedge^p\Omega_{X}[p]\).

The \emph{space of \(p\)-forms of degree \(n\) on \(X\)} is 
\[
\cala^p(X,n)=\Omega^{\infty} \Mor_{\QCoh(X)}(\calo_X,\wedge^p\Omega_X[n]) \simeq \Map_{\QCoh(X)}(\calo_X,\wedge^p\Omega_X[n])
\]

We also construct the weighted negative cyclic chain complex \(NC^w\), whose degree \(n\), weight \(p\) part is
\[
NC^{w,n}(\Omega_X)(p)=(\bigoplus_{i\geq0}\wedge^{p+i}\Omega_X[n-i],d_{\Omega_X}+d_{dR}).
\]
The \emph{space of closed \(p\)-forms of degree \(n\)} is
\[
\cala^{p,cl}(X,n)=\tau_{\leq 0}\Mor_{\QCoh(X)}(\calo_X,NC^n(\Omega_X)(p)).
\]
There is a natural ``underlying form'' map
\[
\cala^{p,cl}(X,n)\to \cala^{p}(X,n)
\]
corresponding to the projection \(\bigoplus_{i\geq0}\wedge^{p+i}\Omega_X[n-i]\to \wedge^{p}\Omega_X[n]\).

A \(2\)-form \(\omega:\calo_X\to\wedge^2\Omega_X[n]\) of degree \(n\) is \emph{nondegenerate} if the adjoint map \(\TT_X\to\Omega_X[n]\) is a quasi-isomorphism.  An \emph{\(n\)-shifted symplectic form} on \(X\) is a closed \(2\)-form whose underlying form is nondegenerate.

Let \(X\) be a derived Artin stack with an \(n\)-shifted symplectic form \(\omega\) and let \(f:Y\to X\) be a morphism.  An \emph{isotropic structure} on \(f\) is a homotopy \(h:0\sim f^*\omega\) in the space of closed forms on \(Y\). An isotropic structure on \(f\) defines a map \(\Theta_h:\TT_f\to\Omega_Y[n-1]\).  We say \(h\) is \emph{Lagrangian} if \(\Theta_h\) is a quasi-isomorphism of complexes.


\subsection{Perfect Complexes on Open Varieties}\label{subsec:open-var}

Let \(U=X\backslash D\) be an open variety.  In general, it is unrealistic to hope for a symplectic structure on \(\map(U,Y)\), because we need to integrate on \(U\), which is noncompact.  But in the particular case of \(Y=\perf\), we can consider the space \(\perf_c(U)\) of compactly supported perfect complexes on \(U\).  We can express this via a pullback square
\[
\xymatrix{
\perf_c(U)\ar[r]\ar[d]&\perf(X)\ar[d]\\
\bullet\ar[r]^0&\perf(D),
}
\]
with \(\bullet\to\perf(D)\) corresponding to the \(0\) complex.  Thus \(\perf_c(U)\) is a geometric stack, and, in particular, an open substack of \(\perf(X)\).  In general we will not distinguish between a compactly supported complex on \(U\) and its extension by \(0\) to \(X\). We claim that \(\perf_c(U)\) carries a symplectic structure:
\begin{thm}
Let \(X\) be a smooth \(d\)-dimensional variety, \(D\subset X\) a divisor, and \(U=X\backslash D\).  Let \(\alpha\) be a meromorphic section of \(\calo_X(K_X)\) which is holomorphic nonvanishing on \(U\).  Then \(\alpha\) induces a \((2-d)\)-shifted symplectic structure on \(\perf_c(U)\).
\label{thm:openmap}
\end{thm}
\begin{proof}
The construction of the form closely mimics that of Theorem \ref{thm:sympmapthm}.

Let \(V(\alpha)=D_+-D_-\) with \(D_+\) and \(D_-\) effective; note that \(D_+\cup D_-\subseteq D\).  We can consider \(\perf_c(U)\) a substack of \(\perf(X)\).  Thus we have an evaluation \(\ev:\perf_c(U)\times X\to \perf\).  Now, \(\ev^*\TT_{\perf}\) has the following description: for \(\underline{g}:\Spec A\to\perf_c(U)\times X\) corresponding to a perfect complex \(E\) on \(\Spec A\times U\) and \(g:\Spec A\to X\), we have
\[
(\ev^*\TT_{\perf})_{\underline{g}}\simeq \Mor((\id_A\times g)^*E,(\id_A\times g)^*E)[1].
\]
If \(\underline{g}\) factors through \(\perf_c(U)\times D_-\), then this vanishes, as \(E\) is supported away from \(D\).  Then for \(p\geq 1\), \((\ev^*\wedge^p\Omega_{\perf})_{D-}\sim0\), and in particular, if \(\omega\) is a \(p\)-form on \(\perf\), \(\ev^*\omega\) vanishes on \(\perf_c(U)\times D_-\) as well.

Similarly to the proof of Theorem \ref{thm:sympmapthm} (\cite{PTVV}) we then get a map
\[
DR(\perf_c(U)\times X)(-\perf_c(U)\times D_-)\to DR(\perf_c(U))\otimes\Gamma(X,\calo(-D_-)).
\]
Further, we have an orientation map
\begin{align*}
\Gamma(X,\calo(-D_-))&\to\Gamma(X,\calo(D_+-D_-))\\
&\xrightarrow{\alpha}\Gamma(X,K_X)\\
&\to k[-d],
\end{align*}
where the last map is projection onto \(H^{d,d}\).  Combining these yields an integration map
\[
\int_\alpha: DR(\perf_c(U)\times X)(-\perf_c(U)\times D_-)\to DR(\perf_c(U))[-d]
\]
and similarly on the level of negative cyclic complexes
\[
\int_\alpha: NC^w(\perf_c(U)\times X)(-\perf_c(U)\times D_-)\to NC^w(\perf_c(U))[-d].
\]

Then if \(\omega\) is the symplectic form on \(\perf\), we get a closed \(2\)-form \(\int_\alpha \ev^*\omega\) on \(\perf_c(U)\).  As in the proof of Theorem \ref{thm:sympmapthm}, we can describe the pairing on \(\perf_c(U)\) as follows.  For \(g:\Spec A\to\perf_c(U)\) corresponding to a perfect complex \(E\) on \(U\times \Spec A\) compactly supported over \(\Spec A\), the pairing
\[
\Mor(E,E)[1]\wedge\Mor(E,E)[1]\to A[2-d]
\]
is given by cup product, followed by trace and integration multiplied by \(\alpha\).  This is nondegenerate because \(\alpha\) is nonvanishing on the support of \(E\).
\end{proof}
\begin{rmk}
Note from the description at the end of the proof that the symplectic structure on  \(\perf_c(U)\) depends only on \(\alpha|_U\), and not on what the compactification \(X\) is.
\end{rmk}
\begin{rmk}
In the special case that \(X\) is Fano and \(D\) is a smooth effective anticanonical divisor, we can write \(\perf_c(U)=\perf(X)\times_{\perf(D)}\Spec k\), where \(\Spec k\to\perf(D)\) corresponds to the zero complex.  This is a Lagrangian intersection over a \((3-d)\)-shifted symplectic derived stack.
\end{rmk}

We note that Anatoly Preygel has obtained similar results in \cite{Preygel} using different methods.


\subsection{Pushforwards for Perfect Complexes}\label{subsec:pushforward}

For maps \(i:X\to Y\) there is an induced map \(i^*:\map(Y,Z)\to\map(X,Z)\).  In the particular case that \(Z=\perf\) is the derived stack of perfect complexes, we get a map \(i_*:\perf(X)\to\perf(Y)\) going the other way as well.  If \(Y\) is a smooth Calabi-Yau variety, \(\perf(Y)\) will be symplectic and we can investigate the properties of this map.

\begin{thm}
Let \(Y\) be a smooth Calabi-Yau variety and \(i:D\to Y\) a smooth divisor. Then the map \(i_*:\perf(D)\to \perf(Y)\), induced by the functor \(i_*: \Perf(D) \to \Perf(Y)\), carries a natural Lagrangian structure.
\label{thm:pfperf}
\end{thm}

\begin{proof}
Consider the diagram
\[\xymatrixcolsep{5pc}
\xymatrixrowsep{5pc}
\xymatrix{
D\times\perf(D)\ar[d]^{\id_d\times i_*}\ar[rrd]^{\ev_D}\\
D\times\perf(Y)\ar[r]^{i\times \id_{\perf(Y)}}&Y\times\perf(Y)\ar[r]^{\ev_Y}&\perf
}
\]

The symplectic form \(\omega_{\perf(Y)}\) on \(\perf(Y)\) is given by \(\int_{[Y]}\ev^*\ch(\cale)\), where \(\cale\) is the universal perfect complex on \(\perf\).  Write \(\cale_Y=\ev_Y^*\cale\) for the universal complex on \(Y\times\perf(Y)\), and similarly \(\cale_D=\ev_D^*\cale\).  Integration clearly commutes with pullback by \(i_*\), so we have
\[
(i_*)^*\int_{[Y]}\ch(\cale_Y)=\int_{[Y]}\ch((\id_Y\times i_*)^*\cale_Y).
\]
Now, for any point \(p:\Spec A\to \perf(Y)\) corresponding to a perfect complex \(E\in\Perf(Y\times A)\), we have
\begin{align*}
(\id_D\times p)^*(\id_Y\times i_*)^*\cale_Y&\simeq (i\times \id_A)^*E\\
&\simeq (\id_D\times p)^*(i\times \id_{\perf(Y)})_*\cale_D
\end{align*}
as sheaves on \(D\times \Spec A\).  Thus we have an isomorphism
\[
(\id_Y\times i_*)^*\cale_Y\simeq (i\times \id_{\perf(Y)})_*\cale_D.
\]
Thus we have
\[
(i_*)^*\omega_{\perf(Y)}\simeq\int_{[Y]}\ch((i\times \id_{\perf(Y)})_*\cale_D).
\]
Now, \(\ch((i\times \id_{\perf(Y)})_*\cale_D=c_1(D)\omega'\) for some form \(\omega'\).  For the integration map, recall that we use the K{\"u}nneth formula \(DR(Y\times\perf(D))\simeq DR(Y)\otimes DR(\perf(D))\) followed by the projection \(DR(Y)\to\calo_Y\).  But then \(c_1(D)\omega'\) will decompose as some sum of terms \(c_1\alpha\otimes\beta\), and \(c_1(D)\alpha\) will project to \(0\) in \(\calo_Y\) because \(c_1(D)\) is a \((1,1)\) form.  This gives our isotropic structure.

For Lagrangianness, consider an \(A\)-point \(g:\Spec A\to\perf_c(D)\), corresponding to a perfect complex \(E\) on \(D\times\Spec A=D_A\).  Then \(\TT_{\Perf(D),g}\simeq\Mor_{D_A}(E,E)[1]\), and \((i_*)^*\TT_{\Perf(Y),(i_*)^g}\simeq \Mor_{Y_A}(i_*E,i_*E)[1]\).  The symplectic structure \(\omega\) on \(\Perf(U)\) at some \(F\) is given by
\begin{align*}
\wedge^2\Mor_{U_A}(F,F)[1]&\xrightarrow{\cup}\Mor_{U_A}(F,F)[2]\\
&\xrightarrow{tr}\Gamma(U_A,\calo_{U_A})[2]\\
&\to H^d(U_A,\calo_{U_A})[2-d]\\
&\xrightarrow{\int\alpha\wedge(-)} A[2-d].
\end{align*}

We claim that
\begin{equation}
\Mor_{Y_A}(i_*E,i_*E)\simeq\Mor_{D_A}(E,E)\oplus\Mor_{D_A}(E,E\otimes K_D)[-1].
\label{eq:extdec}
\tag{*}
\end{equation}
To see this, we note that \(\Mor_{Y_A}(i_*E,i_*E)\simeq\Mor_{D_A}(i^*i_*E,E)\), and
\[
i^*i_*E\simeq \calo_{D_A}\otimes_{i^{-1}\calo_{Y_A}}E\simeq\{E(-D_A)\to E\},
\]
where the map \(\{E(-D_A)\to E\}\) is multiplication by the defining section of \(D_A\).  Since \(E\) is supported on \(D_A\), this is \(0\), so \(i^*i_*E\simeq E\oplus E(-D_A)[1]\).  Then we have
\begin{align*}
\Mor_{Y_A}(i_*E,i_*E)&\simeq\Mor_{D_A}(E\oplus E(-D_A)[1],E)\\
&\simeq\Mor_{D_A}(E,E)\oplus\Mor_{D_A}(E,E(D_A[-1])\\
&\simeq\Mor_{D_A}(E,E)\oplus\Mor_{D_A}(E,E\otimes K_D)[-1].
\end{align*}
Furthermore, the map
\begin{equation}
\Mor_{D_A}(E,E)\to\Mor_{D_A}(E,E)\oplus\Mor_{D_A}(E,E\otimes K_D)[-1]
\label{eq:pushmap}
\tag{**}
\end{equation}
comes from the counit \(E\oplus E(-D_A)[1]\simeq i^*i_*E\to E\) which is the projection, so (\ref{eq:pushmap}) is the obvious inclusion.

In the decomposition (\ref{eq:extdec}), the multiplication structure of the right hand side is the obvious square zero extension of \(\Mor_{D_A}(E,E)\), where \(\Mor_{D_A}(E,E\otimes K_D)[-1]\) has the obvious (left and right) module structure.  Furthermore, trace map is given by the composition
\[
\Mor_{D_A}(E,E)\oplus\Mor_{D_A}(E,E\otimes K_D)[-1]\xrightarrow{pr_2} \Mor_{D_A}(E,E\otimes K_D)[-1]\xrightarrow{tr} \calo_{D_A}(-D_A)\to \calo_{Y_A}.
\]
Then the pairing of \(\TT_{\perf(D),g}\simeq \Mor_{D_A}(E,E)[1]\) with \(\TT_{i_*, g}[d-1]\simeq\Mor_{D_A}(E,E\otimes K_A)[d-1]\) is the nondegenerate pairing of Serre duality, and we have nondegeneracy.
\end{proof}

\begin{rmk}
 Theorem \ref{thm:pfperf} should be contrasted with the result of \cite{Calaque} saying that if \(X\) is a Fano variety and \(a:Y\to X\) a smooth anticanonical divisor, then \(a^*:\perf(X)\to\perf(Y)\) has a Lagrangian structure (\cite{Calaque} Theorem 2.10).
\end{rmk}

\begin{rmk}
The functor \(i_*: \Perf(D) \to \Perf(Y)\) is one of the compatible spherical functor of Example \ref{spherical-examples}.  By Theorem \ref{intro:thm-relcy-iff-spherical} this functor carries a weak relative Calabi-Yau structure. We have conjectured that this can be promoted to a strong Calabi-Yau structure. Theorem \ref{thm:pfperf} can be viewed as giving evidence for this statement, because the map on moduli spaces of objects induced by a relative Calabi-Yau functor is expected to carry a Lagrangian structure. 
\end{rmk}

\section{Further directions}\label{sec:future-directions}

In the introduction, we described some of the geometric background from symplectic topology that led us to the results presented in this paper. In this section, we point to  potential ramifications of these results in a different direction: we outline a plan to introduce and develop a theory of derived \(n\)-shifted hyperk\"ahler stacks, and noncommutative hyperk\"ahler spaces.

Hyperk\"ahler manifolds were first defined  by Calabi in \cite{Calabi-HK}. The initial development of this subject was pioneered by Bogomolov, Beauville and Hitchin. In recent years there has been a tremendous revival of interest in this subject, with several splendid results being obtained by  Verbitsky, Kamenova, Voisin, Kaledin, Huybrechts,  Kollar, Laza, Sacc\`a and others (see, e.g., \cite{MV1}, \cite{MV2}, \cite{MV3}, \cite{KLSV}).

Recall that a hyperk\"ahler manifold is a real \(C^{\infty}\) Riemannian manifold of dimension \(4n\) whose holonomy is contained in \(Sp(n) = O(4n) \cap GL_n(\mathbb{H})\). More explicitly, a hyperk\"ahler manifold is a Riemannian manifold \((X,g)\) whose tangent bundle is equipped with covariantly constant endomorphisms \(I, J, K\) satisfying the quaternionic identities: \(I^2 = J^2 = K^2 = IJK = -1\). Define symplectic forms \(\omega_1(v,w) := g(Iv,w)\), \(\omega_2(v,w) := g(Jv,w)\), and \(\omega_3(v,w) := g(Kv,w)\). It is straightforward to check that the form \(\omega_+ := \omega_2 + \sqrt{-1} \omega_3\) is holomorphic with respect to the complex structure \(I\). Thus, underlying every hyperk\"ahler manifold is a holomorphic symplectic manifold. Conversely, if \(X\) is a \emph{compact} holomorphic symplectic manifold, then by an application of Yau's celebrated theorem on the existence of Ricci flat metrics and an earlier theorem of Bochner, on can show that every K\"ahler class on \(X\) contains a hyperk\"ahler metric. It is this close relationship between holomorphic symplectic geometry and hyperk\"ahler geometry that serves as the starting point for our proposed generalization of hyperk\"ahler geometry to derived stacks: the basic idea is to replace the holomorphic symplectic form by an \(n\)-shifted symplectic structure on a derived stack. 

In fact, even without appealing to Yau's powerful result, it turns out that it is possible to reformulate the notion of a hyperk\"ahler metric in purely holomorphic terms using Penrose's twistor geometry. It is easy to see that each imaginary unit quaternion \(\mathbf{u}\) defines an almost complex structure \(I_{\mathbf{u}}\) on the tangent bundle of \(X\). Identify the imaginary unit quaternions with \(\mathbb{CP}^1 \simeq S^2\). Then the \(C^{\infty}\)-manifold \(X \times S^2\) can be equipped with a unique almost complex structure that restricts to \(I_{\mathbf{u}}\) on \(X \times \{\mathbf{u}\}\), and is compatible with the projection to \(S^2 \simeq \mathbb{CP}^1\). It turns out that this almost complex structure is integrable, and defining  complex manifold \(Z\) with a projection \(\pi: Z \rightarrow \mathbb{CP}^1\) called the \emph{twistor family of X}. 

\begin{thm}\cite{Hitchin-SUSY}\label{thm:HK-gives-twistor}
The twistor space \(Z\) constructed above admits the following structures

\begin{enumerate}
\item \(\pi: Z \rightarrow \mathbb{CP}^1\) is a holomorphic fiber bundle.

\item There is a real structure \(\tau\) on \(Z\) covering the antipodal map on the projective line.

\item There is a holomorphic symplectic form \(\omega_{\mathrm{rel}}\) on the fibers of \(\pi\), which is real with respect to \(\tau\) and given by a section of \(\wedge^{2} T_{\pi}^{\vee} \otimes \pi^* \mathcal{O}_{\mathbb{P}^1}(2)\). 
\item There is a family of global holomorphic sections of \(\pi\), real with respect to \(\tau\), whose normal bundles are given by \( \mathbb{C}^{2n} \otimes_{\mathbb{C}} \pi^* \mathcal{O}_{\mathbb{P}^1}(1) \).

\end{enumerate}

\end{thm}

Moreover, the twistor family completely characterizes the hyperk\"ahler manifold:

\begin{thm}\cite{Hitchin-SUSY}\label{thm:good-twistor-gives-HK}
Let \(Z\) be a complex manifold of complex dimension \(2n+1\) equipped with the structures of Theorem \ref{thm:HK-gives-twistor}. Then the space of real sections (4) of the fibration \(\pi: Z \rightarrow \mathbb{CP}^1\) can be equipped with a natural Riemannian metric that turns it into a hyperk\"ahler manifold of real dimension \(4n\) whose twistor family is \(\pi: Z \rightarrow \mathbb{CP}^1\).

\end{thm}

Theorems \ref{thm:HK-gives-twistor} and \ref{thm:good-twistor-gives-HK} motivate the following tentative definition:

\begin{defn}\label{def:derived-twistor}
A \emph{\(d\)-shifted derived twistor family of hyperk\"ahler type} is given by a locally geometric derived analytic \(\infty\)-stack \(Z\) over \(\mathbb{C}\) equipped with the following structures:

\begin{enumerate}
\item A map \(\pi: Z \rightarrow \mathbb{P}_{\mathbb{C}}^1\) of derived stacks.

\item A real structure \(\tau\) on \(Z\) covering the antipodal map on the projective line.

\item A \(d\)-shifted symplectic structure \(\omega_{\mathrm{rel}}\) on the fibers of \(\pi\), which is real with respect to \(\tau\) and whose underlying \(2\)-form is given by a section of degree \(d\) of \(\wedge^{2} \mathbb{T}_{\pi}^{\vee} \otimes \pi^* \mathcal{O}_{\mathbb{P}^1}(2)\). Here \(\mathbb{T}_{\pi}\) is the homotopy fiber of the natural map of tangent complexes \(\mathbb{T}_{Z} \rightarrow \pi^* \mathbb{T}_{\mathbb{P}^1}\). More precisely, the closed form \(\omega_{\mathrm{rel}}\) itself is a section of \(\mathrm{HC}^{-,w}(Z/\mathbb{P}^1) \otimes \mathcal{O}_{\mathbb{P}^1}(2)\), the relative weighted negative cyclic complex of \(Z\) over \(\mathbb{P}^1\), twisted by \(\mathcal{O}_{\mathbb{P}^{1}}(2)\).

\item  a connected component \(X \) of the
homotopy fixed points \(\map_{\mathbb{P}^1}(\mathbb{P}^1, Z)^{\tau}\) of the induced action of \(\tau\) on the derived mapping stack of analytic sections of \(\pi\), such that the natural map \(X \times \mathbb{P}^{1} \to Z\) is an equivalence
of \(C^{\infty}\) derived stacks.

\end{enumerate}

\end{defn}

\begin{rmk}\label{rmk:proxy}
It will be an interesting question whether it is possible to define an analogue of the notion of a ``derived hyperk\"ahler metric'' on a real \(C^{\infty}\) derived stack, so that the derived analogues of Theorems \ref{thm:HK-gives-twistor} and \ref{thm:good-twistor-gives-HK}. In the absence of such a definition, we propose to \emph{treat the derived twistor families of \ref{def:derived-twistor} as a proxy for the notion of a \(d\)-shifted derived hyperk\"ahler stack}. 
\end{rmk}

Hyperk\"ahler manifolds are difficult to come by, and thus one of the central problems in hyperk\"ahler geometry is the problem of constructing hyperk\"ahler manifolds. Therefore, the first order of business will be to address the following:

\begin{problem}\label{prob:const-hk}
Formulate and prove twistor family versions of  

\begin{enumerate}
\item the result stating that the mapping stack from \(d\)-oriented derived stack into an \(n\)-shifted symplectic stack admits a \(n-d\) shifted symplectic structure, and the relative version of the same (Theorem \ref{thm:PTVV-transgression} and Theorem \ref{thm:Calaque-transgression}).

\item the theorem stating that Lagrangian intersections are symplectic (Theorem \ref{thm:PTVV-Lag-int}).

\end{enumerate}

\end{problem}

An important technique for constructing new hyperk\"ahler manifolds from old ones is hyperk\"ahler reduction. Indeed, many of the interesting examples of hyperk\"ahler manifolds arise from solutions to the anti-self-dual Yang-Mills equations, and thus can be viewed as arising from infinite-dimensional hyperk\"ahler reduction. Safronov has explained \cite{Safronov-reduction} how to interpret symplectic reduction as a Lagrangian intersection in derived algebraic geometry. By implementing his construction in twistor families using the solution to Problem \ref{prob:const-hk}, it should be possible to address the following:

\begin{problem}\label{prob:derived-HK-reduction}
Introduce a derived version of hyperk\"ahler reduction, using the methods of \cite{Safronov-reduction}, and use this to construct new derived hyperk\"ahler stacks.
\end{problem}

In order to successfully address these problems, it will be necessary to develop a robust theory of shifted symplectic structures in \emph{families} and on derived \emph{analytic} stacks. This theory should be of independent interest. A successful solution to Problems \ref{prob:const-hk} and \ref{prob:derived-HK-reduction} should lead to a new conceptual understanding of classical hyperk\"ahler spaces, such as moduli spaces of sheaves on K3 surfaces, and instanton moduli spaces such as bow varieties. The first examples where we hope to obtain hyperk\"ahler structures with a non-trivial and interesting derived/stacky structure are \emph{singular} moduli spaces of sheaves on K3 surfaces and  coadjoint orbits.

One of our motivations for studying derived hyperk\"ahler geometry comes from nonabelian Hodge theory \cite{Simpson-nah}. Let \(Y\) be a smooth projective variety over \(\mathbb{C}\) of dimension \(d\), and let \(G\) be a reductive algebraic group over \(\mathbb{C}\). Then, by results of Hitchin, Simpson,  Fujiki, etc \cite{Hitchin-sd, Fujiki-hk, Simpson-nah} the moduli space \(\mathcal{M}_{\mathrm{Harm}}(Y,G)\) of harmonic \(G\)-bundles (solutions to Hitchin's equations) on the K\"ahler manifold \(Y(\mathbb{C})\) carries a natural hyperk\"ahler metric. The associated twistor family \(\Tw(\mathcal{M}_{\mathrm{Harm}}(Y,G))\) is the Deligne-Simpson space \(\pi: \mathcal{M}^{ss}_{\mathrm{Del}}(Y,G) \rightarrow \mathbb{P}^1\). The fiber over \(\lambda \in \mathbb{A}^1 \subset \mathbb{P}^1\) is the space of semistable \(\lambda\)-connections on \(Y\). For \(\lambda = 0\) this is the moduli space of semistable Higgs bundle \(\mathcal{M}^{ss}_{\mathrm{Dol}}(Y,G)\), ad for \(\lambda = 1\) this is the de Rham moduli space  \(\mathcal{M}^{ss}_{\mathrm{DR}}(Y,G)\). It follows that \(\mathcal{M}^{ss}_{\mathrm{Dol}}(Y,G)\) and   \(\mathcal{M}^{ss}_{\mathrm{DR}}(Y,G)\) carries natural holomorphic (0-shifted) symplectic structures. 

On the other hand, \cite{PTVV} construct natural \emph{\(2(1-d)\)-shifted} structures on the derived stacks \(\mathcal{M}_{\mathrm{Dol}}(Y,G):= \Map(Y_{\mathrm{Dol}}, BG)\) and  \(\mathcal{M}_{\mathrm{DR}}(Y,G):= \Map(Y_{\mathrm{DR}}, BG)\). This leads to the following question, which we plan to investigate:

\begin{question}

\begin{enumerate}
\item When \(d= 0\) we have two natural \(0\)-shifted holomorphic symplectic forms on the moduli space \(\mathcal{M}^{ss}_{\mathrm{Dol}}(Y,G)\)  (resp. \(\mathcal{M}^{ss}_{\mathrm{DR}}(Y,G)\) ), one coming from its realization as a fiber of the twistor family of the Hitchin's hyperk\"ahler manifold of harmonic bundles, and the other coming from the \cite{PTVV} mapping space construction. Are these two symplectic structures relates in any way?

\item When \(d > 0\), do the \(2(1-d)\)-shifted \cite{PTVV} symplectic structures on these moduli spaces have any relation to the other structures associated with the Hitchin-Deligne-Simspson twistor family?  Do  these \(2(1-d)\)-shifted symplectic structures  give rise to an \(n\)-shifted derived twistor family of hyperk\"ahler type (Definition \ref{def:derived-twistor}) for which the fiber over \(\lambda \in \mathbb{A}^1\) is a suitably rigidified version of the derived moduli stack of \(\lambda\)-connections?\\

\end{enumerate}

\end{question}

Next, we turn to the problem of studying (derived) hyperk\"ahler geometry through the lens of categorical noncommutative geometry.  The functor \(\Perf\) from commutative spaces to noncommutative spaces has an adjoint, namely the functor \(\calx \mapsto \mathrm{Moduli}_{\calx}\) \cite{Toen-dgcat-moduli} carrying a category to the moduli space of compact objects in it.  Abuaf \cite{AB} has introduced the notion of a hyperk\"ahler category that takes \(\Perf\) as the starting point, in the following sense: his definition is designed so that, when \(X\) is an ordinary variety, the category \(\Perf(X)\) is hyperk\"ahler if and only if \(X\) is a hyperk\"ahler manifold. For our purposes, the dual point of view is more natural; thus, the notion of \(\nc\)-hyperk\"ahler space should have the property that the functor \( \mathrm{Moduli}\) carries an \(\nc\)-hyperk\"ahler space of dimension \(d\) to a \((2-d)\)-shifted derived hyperk\"ahler stack in the sense of Definition \ref{def:derived-twistor} and Remark \ref{rmk:proxy}. 

\begin{rmk}\label{rmk:nc-hk-elements}
The functor \( \mathrm{Moduli}\) carries \(d\)-Calabi-Yau structures on categories to \((2-d)\)-shifted symplectic structures on derived stacks. Therefore, the noncommutative analogue of  Definition \ref{def:derived-twistor}, should, very roughly speaking, incorporate the following elements:

\begin{itemize}
\item[-] a quasi-coherent sheaf \(\caly\) of proper \(\mathbb{C}\)-linear \(\infty\)-categories on \(\mathbb{P}^1\)

\item[-] an \(S^1\)-equivariant  morphism 
\[
\HH_*(\caly) \rightarrow \mathcal{O}_{\mathbb{P}^1}(2)[-d]
\]
in \(\Perf(\mathbb{P}^1)\), which defines Calabi-Yau structures on the stalks of \(\caly\). 
\item[-] a real structure \( \tau: \sigma^* \caly \xrightarrow{\sim} \caly\) covering the antipodal map \(\sigma\).

\item[-] a family of preferred real global sections of \(\caly\) with deformation-theoretic properties analogous to Theorem \ref{thm:HK-gives-twistor} (4). 
\end{itemize}

\end{rmk}

\begin{problem}
Give a precise definition of \(\nc\)-hyperk\"ahler spaces incorporating the elements described in Remark \ref{rmk:nc-hk-elements}. Prove a twistor family analogue of Theorem \ref{intro:thm-main-gluing} that allows one to construct new  \(\nc\)-hyperk\"ahler spaces by gluing.  
\end{problem}

Such a noncommutative set-up opens the door to many new avenues for studying hyperk\"ahler geometry. For instance, if the general fiber \(\caly_x\) of the sheaf \(\caly\) in Remark \ref{rmk:nc-hk-elements} arises as the global sections of a perverse Schober, then we can use the techniques of symplectic topology and the methods developed in this paper to decompose and study  \(\caly_x\) in terms of simpler constituents. In addition to being of interest in its own right, the added flexibility afforded by the noncommutative framework could lead to new insights about classical hyperk\"ahler geometry.

\bibliographystyle{amsalpha}
\bibliography{PPerf}

\vspace{5.mm}

  \begin{tabular}{l}
   \small{\textsc{Ludmil Katzarkov}} \\
\hspace{.1in} \small{\textsc{Universit\"at Wien, Fakult\"at f\"ur Mathematik,  1090 Wien, \"Osterreich }}\\
 \hspace{.15in}\small{\textsc{National Research University, Higher School of Economics, Russian Federation}}\\
   \hspace{.1in} \small{\textsc{Email}}: {\bf lkatzarkov@gmail.com} \\
  \end{tabular}

\vspace{2.mm}

  \begin{tabular}{l}
\small{\textsc{Pranav Pandit}} \\
   \hspace{.1in} \small{\textsc{Universit\"at Wien, Fakult\"at f\"ur Mathematik,  1090 Wien, \"Osterreich}}\\
   \hspace{.1in} \small{\textsc{Email}}: {\bf pranav.pandit@univie.ac.at} \\
  \end{tabular}

\vspace{2.mm}

  \begin{tabular}{l}
   \small{\textsc{Ted Spaide}} \\
   \hspace{.1in} \small{\textsc{Universit\"at Wien, Fakult\"at f\"ur Mathematik,  1090 Wien, \"Osterreich}} \\
   \hspace{.1in} \small{\textsc{Email}}: {\bf theodore.spaide@univie.ac.at} \\
  \end{tabular}

\end{document}